\documentclass[sort,times]{elsarticle}
\usepackage{indentfirst}
\usepackage{amsthm}
\usepackage{lmodern}

\usepackage{mathrsfs}
\usepackage{amsmath}
\usepackage{subfig}
\usepackage{amsfonts}
\usepackage{mathtools}
\usepackage{enumerate}
\usepackage{enumitem}
\usepackage{ulem}
\usepackage{amssymb}
\usepackage{graphicx}
\usepackage{wrapfig}
\usepackage{fancyvrb}
\usepackage{verbatim}
\usepackage{booktabs}
\usepackage{xcolor}
\usepackage{rotating}
\usepackage{accents}
\usepackage{stackrel}
\usepackage{stmaryrd}
\usepackage{marvosym}
\usepackage{comment}
\usepackage{tikz}
\usetikzlibrary{decorations.markings}
\usetikzlibrary{arrows}
\usetikzlibrary{patterns}

 \if@mathematic
   \def\vec#1{\ensuremath{\mathchoice
                     {\mbox{\boldmath$\displaystyle\mathbf{#1}$}}
                     {\mbox{\boldmath$\textstyle\mathbf{#1}$}}
                     {\mbox{\boldmath$\scriptstyle\mathbf{#1}$}}
                     {\mbox{\boldmath$\scriptscriptstyle\mathbf{#1}$}}}}
\else
   \def\vec#1{\ensuremath{\mathchoice
                     {\mbox{\boldmath$\displaystyle#1$}}
                     {\mbox{\boldmath$\textstyle#1$}}
                     {\mbox{\boldmath$\scriptstyle#1$}}
                     {\mbox{\boldmath$\scriptscriptstyle#1$}}}}
\fi

\renewcommand{\matrix}[1]{{\mathbf{#1}}}

\usepackage[left=2.5cm,top=1.5cm,right=2.5cm]{geometry}

\newcommand{\pderivative}[2]{\frac{\partial #1}{\partial #2}}
\newcommand{\jump}[1]{\left\llbracket #1 \right\rrbracket}
\newcommand{\average}[1]{\left\{\!\!\left\{#1\right\}\!\!\right\}}

\theoremstyle{plain}
\newtheorem{thm}{Theorem}
\newtheorem{cor}{Corollary}
\newtheorem{lem}{Lemma}

\theoremstyle{definition}

\theoremstyle{remark}
\newtheorem{rem}{Remark}

\allowdisplaybreaks

\numberwithin{equation}{section}

\begin{document}

\begin{frontmatter}
\title{Affordable, Entropy Conserving and Entropy Stable Flux Functions for the Ideal MHD Equations}
\author[unikoeln]{Andrew R.~Winters\corref{cor1}}
\cortext[cor1]{Corresponding author.}
\ead{awinters@math.uni-koeln.de}
\author[unikoeln]{Gregor J.~Gassner}

\address[unikoeln]{Mathematisches Institut, Universit\"{a}t zu K\"{o}ln, Weyertal 86-90, 50931 K\"{o}ln, Germany}
\begin{abstract}
In this work, we design an entropy stable, finite volume approximation for the ideal magnetohydrodynamics (MHD) equations. The method is novel as we design an affordable analytical expression of the numerical interface flux function that discretely preserves the entropy of the system. To guarantee the discrete conservation of entropy requires the addition of a particular source term to the ideal MHD system. Exact entropy conserving schemes cannot dissipate energy at shocks, thus to compute accurate solutions to problems that may develop shocks, we determine a dissipation term to guarantee entropy stability for the numerical scheme. Numerical tests are performed to demonstrate the theoretical findings of entropy conservation and robustness.
\end{abstract}

\end{frontmatter}

\section{Introduction}\label{intro}
 
The entropy in physical systems governed by (nonlinear) conservation laws is an often overlooked quantity that is conserved for smooth solutions but increases (or decreases according to the sign convention adopted) in the presence of shocks. One can design numerical methods to be entropy conservative if, discretely, the local changes of entropy are the same as predicted by the continuous entropy conservation law. The approximation is said to be entropy stable if it produces more entropy than an entropy conservative scheme.

For entropy stable numerical fluxes, Tadmor \cite{tadmor1984,tadmor1987} was the first to introduce the idea of entropy conservation to design stable numerical approximations of nonlinear hyperbolic conservation laws. On a semi-discrete level the principle is that the discrete flux function satisfies discrete analogs of the conservation laws for the conservative variables as well as the scalar conservation law for entropy. Tadmor's flux function performs well for smooth data but can be made stable for shock problems \cite{tadmor2003}. But, Tadmor's flux function, which involves an integral in phase space, is numerically expensive. For large scale simulations we want a more practical, computationally tractable, entropy conserving flux.

In the examination of entropy conservation we discuss an important distinction between entropy conservation and stability, because there is a problem with entropy conservative formulations. Such approximations can suffer breakdown if used without dissipation to capture shocks which results in large amplitude oscillations  \cite{carpenter_esdg}. Thus, an issue of entropy conservative formulations is that they may not converge to the weak solution as there is no mechanism to admit the dissipation physically required at the shock.

For the ideal magnetohydrodynamic (MHD) system the issue of entropy conservation and satisfaction of the divergence-free condition (i.e. $\nabla\cdot\vec{B}=0$) are inextricably linked. In this paper we will show that to guarantee the conservation of entropy in a discrete sense requires the addition of a source term proportional to the divergence of the magnetic field. This source term also operates as a numerical method of \textit{divergence cleaning} similar to the method introduced by Dedner et. al. \cite{dedner2002}. That is, errors in the divergence-free conditions (that can be thought of as ``numerical magnetic monopoles'' \cite{dedner2002}) are advected away from their point of origin with a speed proportional to the fluid velocity. So, for multi-dimensional approximations, this source term offers a simple mechanism to control errors in the divergence-free conditions. 

In this paper we develop affordable, entropy stable methods for the ideal MHD model. There is recent work for entropy stable approximations for the Euler equations from Ismail and Roe \cite{ismail2009} and Chandrashekar \cite{chandrashekar2013}.  The entropy stable approximations were extended to arbitrary order with a finite volume discretization by LeFloch and Rohde \cite{lefloch_rhode_2000} or a discontinuous Galerkin (DG) spectral element formulation by Carpenter et. al. \cite{carpenter_esdg}. Also, there is recent work on entropy-stable, high-order, and well-balanced approximations for the shallow water equations \cite{gassner_skew_burgers,gassner2014}.

The remainder of this paper is organized as follows: Sec. \ref{GoverningEquations} provides an introduction to the ideal MHD equations, the computational issues, and select the governing equations for which entropy will be conserved. In Sec. \ref{FVDiscretization} we briefly describe our finite volume discretization. Sec. \ref{EntropyAnalysis} defines the necessary variables and analytical tools to discuss the entropy of the ideal MHD system in a mathematically rigorous way. We derive an entropy conserving numerical flux in Sec. \ref{Sec:EntropyConservingFlux} and discuss stabilizations of the conservative flux in Sec. \ref{Sec:StableFlux}. Numerical results are presented in Sec. \ref{NumericalResults}, where we verify theoretical predictions and compare the results of our method against known results in the literature. Sec. \ref{conclusion} presents concluding remarks. \ref{3DFluxes} outlines the extension of the entropy conserving and stable flux functions to higher spatial dimensions.
\section{The Ideal Magnetohydrodynamic Equations}\label{GoverningEquations}

The ideal magnetohydrodynamic (MHD) equations are a classical model from plasma physics used to study the dynamics of electrically conducting fluids. The equations are a single set of fluid equations for the total mass density, center-of-mass momentum density, total energy density, and the magnetic field of the system \cite{rossmanith2002}. The ideal MHD equations can be written as a system of conservation laws
\begin{equation}\label{3DIDEALMHD}
\begin{aligned}
\pderivative{}{t}\begin{bmatrix} \rho \\[0.05cm] \rho\vec{u} \\[0.05cm] \rho e \\[0.05cm] \vec{B} \end{bmatrix} + &\nabla\cdot\begin{bmatrix} \rho\vec{u} \\[0.05cm]
\rho(\vec{u}\otimes\vec{u}) + \left(p+\frac{1}{2}\|\vec{B}\|^2\right)\matrix{I}-\vec{B}\otimes\vec{B}  \\[0.05cm]
\vec{u}\left(\rho e + p + \frac{1}{2}\|\vec{B}\|^2 \right) - \vec{B}(\vec{u}\cdot\vec{B}) \\[0.05cm]
\vec{B}\otimes\vec{u} - \vec{u}\otimes\vec{B}
\end{bmatrix} = 0,
\\
&\nabla\cdot\vec{B} = 0,
\end{aligned}
\end{equation}
where $\rho$, $\rho\vec{u}$, and $\rho e$ are the mass, momentum, and energy densities of the plasma system, and $\vec{B}$ is the magnetic field. The thermal pressure, $p$, is related to the conserved quantities through the ideal gas law:
\begin{equation}
p = (\gamma-1)\left(\rho e - \frac{\rho}{2}\|\vec{u}\|^2 -\frac{1}{2}\|\vec{B}\|^2 \right).
\end{equation}

The remainder of this paper focuses on the derivation of an entropy conserving and entropy stable finite volume methods for the one-dimensional ideal MHD system
\begin{equation}\label{1DIdealMHD}
\begin{aligned}
\pderivative{}{t}\begin{bmatrix} \rho \\ \rho u \\ \rho v\\ \rho w \\ \rho e \\ B_1 \\ B_2 \\ B_3 \end{bmatrix} + &\pderivative{}{x}\begin{bmatrix} 
\rho u \\ \rho u^2 + p + \frac{1}{2}\|\vec{B}\|^2 - B_1^2 \\ \rho u v - B_1 B_2 \\ \rho u w - B_1 B_3 \\ u\left(\rho e + p + \frac{1}{2}\|\vec{B}\|^2\right) - B_1(\vec{u}\cdot\vec{B}) \\ 0 \\ u B_2 - v B_1 \\ u B_3 - w B_1 \end{bmatrix} = 0,
\\
&\pderivative{B_1}{x} = 0.
\end{aligned}
\end{equation}
We consider the one-dimensional ideal MHD system to demonstrate the validity of the entropy conserving approach derived in this paper. The entropy analysis tools and flux derivations easily extend to higher spatial dimensions (as we show in \ref{3DFluxes}).

We immediately observe that the sixth equation of \eqref{1DIdealMHD} simplifies to
\begin{equation}
\pderivative{B_1}{t} = 0.
\end{equation}
Meaning that any spatial variation in $B_1$ is stationary in time. Together with the divergence-free condition $\partial_x(B_1) = 0$, we obtain the result that $B_1$ is a constant in space and time for the 1D ideal MHD system. In our derivations, however, we will treat the quantity of $B_1$ as non-constant for two important reasons:
\begin{enumerate}
\item The proof of the discrete entropy conserving flux incorporates contributions from $B_1$ that are important for cases when $B_1\ne constant$.
\item We want the discussion to remain general and easily extendable to higher dimensions.
\end{enumerate}

In general, the divergence-free condition on the magnetic field (sometimes referred to as the \textit{involution} condition or \textit{solenoidal} condition) is a physical law that reflects that magnetic monopoles have never been observed. Numerical methods for multidimensional ideal MHD must, in general, satisfy (or at least control) some discrete version of the divergence-free condition \cite{toth2000}. Failure to do so generally leads to a nonlinear numerical instability, which can cause negative pressures and/or densities. There are several approaches to control the error in $\nabla\cdot\vec{B}$ and in depth review of many methods can be found in T\'{o}th \cite{toth2000}. 

Hyperbolic systems, like the ideal MHD equations, can be efficiently and accurately solved numerically with Godunov-type methods \cite{leveque2002}. These methods require the solution of a Riemann problem at element interfaces. The homogeneous 1D ideal MHD equations \eqref{1DIdealMHD} support seven propagating plane-wave solutions (two fast magnetoacoustic, two Alfv\'{e}n, two slow magnetoacoustic, and an entropy wave) and a stationary plane-wave solution (the divergence wave). The stationary plane-wave solution comes directly from the fact that the MHD equations preserve the divergence constraint:
\begin{equation}
\pderivative{}{t}(\nabla\cdot\vec{B}) = 0.
\end{equation}
Nonlinear numerical instability can be viewed as a direct consequence of the stationary divergence wave \cite{toth2000}. Even if the divergence-free condition is satisfied initially by an approximation, it is not guaranteed to remain satisfied as the solution evolves. These errors, which could be interpreted as numerical monopoles \cite{dedner2002}, are stationary and grow in time. 

In the course of an entropy analysis of the ideal MHD equations \eqref{3DIDEALMHD} Godunov \cite{godunov1972} observed that the divergence-free condition can be incorporated into the ideal MHD equations as a source term proportional to the divergence of the magnetic field (which, on a continuous level, is adding zero). This additional source term not only allows the equations to be put in symmetric hyperbolic form \cite{barth99,godunov1972}, but it also restores Galilean invariance. With the additional term the divergence wave is no longer stationary, and instead, is advected as a passive scalar \cite{janhunen2000,powell1994}:
\begin{equation}
\pderivative{}{t}(\nabla\cdot\vec{B}) + \nabla\cdot(\vec{u}\nabla\cdot\vec{B})=0.
\end{equation}

For numerical approximations the difference between a stationary and a propagating divergence wave is significant. Now, errors generated by ``numerical monopoles'' are advected away from their point of origin with a speed proportional to the fluid velocities. A similar idea lies behind the hyperbolic divergence cleaning method of Dedner et. al. \cite{dedner2002}. Powell \cite{powell1994} demonstrated that numerical methods applied to the MHD equations with the symmetrizing source term were much more stable than the same methods applied to the original MHD equations. The modified form of the MHD equations is often referred to as the 8-wave formulation, since this augmented system supports eight propagating plane wave solutions. Although this approach has been used with some success, it does have a significant drawback: the 8-wave formulation is non-conservative and difficulties with obtaining the correct weak solution have been documented in the literature (see for example T\'{o}th \cite{toth2000}).

Janhunen \cite{janhunen2000} used a proper generalization of Maxwell's equations when magnetic monopoles are present and imposed electromagnetic duality invariance of the Lorentz force to derive {\color{black}{an alternative}} source term for the MHD system:
\begin{equation}\label{JanhunenSource}
\pderivative{}{t}\begin{bmatrix} \rho \\[0.05cm] \rho\vec{u} \\[0.05cm] \rho e \\[0.05cm] \vec{B} \end{bmatrix} + \nabla\cdot\begin{bmatrix} \rho\vec{u} \\[0.05cm]
\rho(\vec{u}\otimes\vec{u}) + \left(p+\frac{1}{2}\|\vec{B}\|^2\right){\color{black}{\matrix{I}}}-\vec{B}\otimes\vec{B}  \\[0.05cm]
\vec{u}\left(\rho e + p + \frac{1}{2}\|\vec{B}\|^2 \right) - \vec{B}(\vec{u}\cdot\vec{B}) \\[0.05cm]
\vec{B}\otimes\vec{u} - \vec{u}\otimes\vec{B}
\end{bmatrix} = -(\nabla\cdot\vec{B})\begin{bmatrix} 0\\\vec{0}\\0\\\vec{u}\end{bmatrix}.
\end{equation}
The Janhunen source term \eqref{JanhunenSource} preserves the conservation of momentum and energy, and treats the magnetic field as an advected scalar. Additionally, the source term \eqref{JanhunenSource} restores the positivity of the Riemann problem \cite{janhunen2000} and the Lorentz invariance of the ideal MHD system \cite{dellar2001}.

The Janhunen source term also plays an important role to guarantee the discrete conservation of entropy, a fact we show in Sec. \ref{Sec:EntropyConservingFlux}. Thus, for the governing equations and analysis of our prototype entropy conserving flux for the ideal MHD system we consider the one-dimensional problem with the addition of the Janhunen source term:
\begin{equation}\label{1DIdealMHDwithSource}
\pderivative{\vec{q}}{t} + \pderivative{\vec{f}}{x} = \pderivative{}{t}\begin{bmatrix} \rho \\ \rho u \\ \rho v\\ \rho w \\ \rho e \\ B_1 \\ B_2 \\ B_3 \end{bmatrix} + \pderivative{}{x}\begin{bmatrix} 
\rho u \\ \rho u^2 + p + \frac{1}{2}\|\vec{B}\|^2 - B_1^2 \\ \rho u v - B_1 B_2 \\ \rho u w - B_1 B_3 \\ u\left(\rho e + p + \frac{1}{2}\|\vec{B}\|^2\right) - B_1(\vec{u}\cdot\vec{B}) \\ 0 \\ u B_2 - v B_1 \\ u B_3 - w B_1 \end{bmatrix} = -\pderivative{B_1}{x}\begin{bmatrix} 0\\0\\0\\0\\0\\u\\v\\w\end{bmatrix} = \vec{s},
\end{equation}
where $\vec{q}$ is the vector of conserved variables, $\vec{f}$ is the vector flux, and $\vec{s}$ is the vector source term.

\section{Finite Volume Discretization}\label{FVDiscretization} 

The finite volume (FV) method is a discretization technique for partial differential equations especially useful for the approximation of systems of conservations laws. The finite volume method is designed to approximate conservation laws in their integral form, e.g.,
\begin{equation}
\int_V\vec{q}_t\,dx + \int_{\partial V} \vec{f}\cdot\hat{\vec{n}}\,dS = 0.
\end{equation}
For instance, in one spatial dimension we break the interval into non-overlapping intervals
\begin{equation}
V_i = \left[x_{i-\tfrac{1}{2}},x_{i+\tfrac{1}{2}}\right],
\end{equation}
and the integral equation of a balance law, with a source term, becomes
\begin{equation}\label{FVSource}
\frac{d}{dt}\int_{x_{i-{1}/{2}}}^{x_{i+{1}/{2}}} \vec{q}\,dx + \vec{f}^*\left(x_{i+{1}/{2}}\right) - \vec{f}^*\left(x_{i-{1}/{2}}\right) = \int_{x_{i-{1}/{2}}}^{x_{i+{1}/{2}}} \vec{s}\,dx.
\end{equation}
A common approximation is to assume that the solution and the source term are constant within the volume. Then we determine, for example, what is analogous to a midpoint quadrature approximation of the solution integral
\begin{equation}
\int_{x_{i-{1}/{2}}}^{x_{i+{1}/{2}}} \vec{q}\,dx \approx \int_{x_{i-{1}/{2}}}^{x_{i+{1}/{2}}} \vec{q}_i\,dx = \vec{q}_i\Delta x_i.
\end{equation}
Note that the solution is typically discontinuous at the boundaries of the volumes. To resolve this, we introduce the idea of a ``numerical flux,'' $\vec{f}^*(\vec{q}^L,\vec{q}^R)$, often derived from the (approximate) solution of a Riemann problem. That is, $\vec{f}^*$ is a function that takes the two states of the solution at an element interface and returns a single flux value. For consistency, we require that 
\begin{equation}\label{consistency}
\vec{f}^*(\vec{q},\vec{q}) = \vec{f},
\end{equation}
that is, the numerical flux is equivalent to the physical flux if the states on each side of the interface are identical. A significant portion of this paper is devoted to the derivation of a numerical flux that conserves the discrete entropy of the system for the 1D ideal MHD equations \eqref{1DIdealMHDwithSource}. So we defer the discussion of the numerical flux to Sec. \ref{Sec:EntropyConservingFlux}. 

We must also address how to discretize the source term $\vec{s}$ in \eqref{FVSource}. There is a significant amount of freedom in the source term discretization. Previous work in \cite{fjordholm2011,winters2015} has demonstrated that designing entropy conserving methods for equations with a source term requires special treatment. In the later derivations a consistent source term discretization necessary for entropy conservation will reveal itself. So, the in-depth discussion of the discrete treatment of the source term is discussed in Sec. \ref{Sec:EntropyConservingFlux}. 

For full clarity, we choose the general form of the finite volume discretization to be
\begin{equation}\label{FVMethodOurs}
{\color{black}{\left(\vec{q}_t\right)_i}} +\frac{1}{\Delta x_i}\left(\vec{f}_{i+\tfrac{1}{2}}^* - \vec{f}_{i-\tfrac{1}{2}}^*\right) = \vec{s}_i = \frac{1}{2}\left(\vec{s}_{i+\tfrac{1}{2}} + \vec{s}_{i-\tfrac{1}{2}}\right),
\end{equation}
where, in a general sense, the discrete source term in cell $i$ will contribute {\color{black}{at each}} interface $i + 1/2$ and $i -1/2$.

\section{Entropy Analysis}\label{EntropyAnalysis}

In this section we define the entropy variables, entropy Jacobian, and other quantities necessary to develop an entropy stable approximation for the ideal MHD equations. We note that one can find a fully general and detailed description of entropy stability theory in, for example, \cite{barth99,fjordholm2012}. 

In the case of ideal MHD a suitable entropy is the physical entropy density (scaled by the constant $(\gamma -1)$ for convenience)
\begin{equation}\label{IdealMHDEntropyFunction}
U(\vec{q}) = -\frac{\rho s}{\gamma -1},
\end{equation}
where $s = \ln(p) - \gamma\ln(\rho)$ is the physical entropy and $\vec{q}$ is the vector of conservative variables. The minus sign in \eqref{IdealMHDEntropyFunction} is conventional in the theory of hyperbolic conservation laws to ensure a decreasing entropy function. The entropy flux for 1D ideal MHD is
\begin{equation}
F(\vec{q}) = uU = -\frac{\rho u s}{\gamma-1}.
\end{equation}

The entropy variables are defined as
\begin{equation}\label{entropyVariables}
\vec{v} \coloneqq U_{\vec{q}} = \left[ \frac{\gamma - s}{\gamma - 1} -\frac{\rho\|\vec{u}\|^2}{2p},\;\frac{\rho u}{p},\;\frac{\rho v}{p},\;\frac{\rho w}{p},\;-\frac{\rho}{p},\;\frac{\rho B_{\color{black}{1}}}{p},\;\frac{\rho B_{\color{black}{2}}}{p},\;\frac{\rho B_{\color{black}{3}}}{p}\right]^T.
\end{equation}
The entropy variables \eqref{entropyVariables} are equipped with the symmetric positive definite (s.p.d) Jacobian matrices
\begin{equation}
\vec{q}_{\vec{v}}\coloneqq\matrix{H}^{-1},
\end{equation}
and
\begin{equation}\label{entropyJacobian}
 \vec{v}_{\vec{q}}\coloneqq\matrix{H} = \begin{bmatrix}
\rho & \rho u & \rho v & \rho w & \rho e - \frac{1}{2}\|\vec{B}\|^2 & 0 & 0 & 0 \\[0.1cm]
\rho u & \rho u^2 + p & \rho u v & \rho u w & \rho {H} u & 0 & 0 & 0 \\[0.1cm]
\rho v & \rho u v & \rho v^2 + p& \rho v w & \rho {H} v & 0 & 0 & 0 \\[0.1cm]
\rho w & \rho u w & \rho v w & \rho w^2 + p & \rho {H} w & 0 & 0 & 0 \\[0.1cm]
\rho e - \frac{1}{2}\|\vec{B}\|^2 & \rho{H} u & \rho{H} v & \rho{H} w & \rho{H}^2 - \frac{a^2 p }{\gamma-1} + \frac{a^2\|\vec{B}\|^2}{\gamma} & \frac{pB_1}{\rho}& \frac{pB_2}{\rho}& \frac{pB_3}{\rho}\\[0.1cm]
0 & 0 & 0 & 0 & \frac{pB_1}{\rho} & \frac{p}{\rho} & 0 & 0 \\[0.1cm]
0 & 0 & 0 & 0 & \frac{pB_2}{\rho} & 0 & \frac{p}{\rho} & 0 \\[0.1cm]
0 & 0 & 0 & 0 & \frac{pB_3}{\rho} & 0 & 0 & \frac{p}{\rho} 
\end{bmatrix},
\end{equation}
where
\begin{equation}
a^2 = \frac{p\gamma}{\rho},\quad\rho e = \frac{p}{\gamma-1} + \frac{\rho}{2}\|\vec{u}\|^2 + \frac{1}{2}\|\vec{B}\|^2,\quad H = \frac{a^2}{\gamma-1} + \frac{1}{2}\|\vec{u}\|^2.
\end{equation}
Finally, as it will be of use in later derivations, we compute the entropy potential to be
\begin{equation}\label{EntropyPotential}
\phi = \vec{v}\cdot\vec{f} - F = \rho u + \frac{\rho u\|\vec{B}\|^2}{2p} - \frac{\rho B_1(\vec{u}\cdot\vec{B})}{p}.
\end{equation}

\subsection{Discrete Entropy Conservation for the 1D Ideal MHD Equations}\label{Sec:DiscreteEntropy}

Next, we introduce the concept of discrete entropy conservation. Let's assume that we have two adjacent states $(L,R)$ with cell areas $(\Delta x_L,\Delta x_R)$. We discretize the ideal MHD equations semi-discretely and examine the approximation at the $i+\tfrac{1}{2}$ interface. We suppress the interface index unless it is necessary for clarity. {\color{black}{Also note the factor of one half from the source term discretization in \eqref{FVMethodOurs}.}}
\begin{equation}\label{FVupdate}
\begin{aligned}
\Delta x_L\pderivative{\vec{q}_L}{t} &= \vec{f}_L - \vec{f}^*  + {\color{black}{\frac{1}{2}}}\Delta x_L{\vec{s}}_{i+\tfrac{1}{2}},\\
\Delta x_R\pderivative{\vec{q}_R}{t} &= \vec{f}^*-\vec{f}_R + {\color{black}{\frac{1}{2}}}\Delta x_R{\vec{s}}_{i+\tfrac{1}{2}}.
\end{aligned}
\end{equation}
We can interpret the update \eqref{FVupdate} as a finite volume scheme where we have left and right cell-averaged values separated by a common flux interface.

We premultiply the expressions \eqref{FVupdate} by the entropy variables to convert to entropy space. From the chain rule we know that $U_t = \vec{v}^T\vec{q}_t$, hence a semi-discrete entropy update is
\begin{equation}\label{EntropyUpdate}
\begin{aligned}
\Delta x_L\pderivative{U_L}{t} &= \vec{v}^T_L\left(\vec{f}_L - \vec{f}^*  +{\color{black}{\frac{1}{2}}}\Delta x_L{\vec{s}}_{i+\tfrac{1}{2}}\right), \\
\Delta x_R\pderivative{U_R}{t} &= \vec{v}^T_R\left(\vec{f}^*-\vec{f}_R + {\color{black}{\frac{1}{2}}}\Delta x_R{\vec{s}}_{i+\tfrac{1}{2}}\right).
\end{aligned}
\end{equation}
If we denote the jump in a state as $\jump{\cdot} = (\cdot)_R - (\cdot)_L$ and the average of a state as $\average{\cdot} = ((\cdot)_R + (\cdot)_L)/2$, then the total update will be
\begin{equation}\label{TotalUpdate}
\pderivative{}{t}(\Delta x_L U_L + \Delta x_R U_R) =  \jump{\,\vec{v}\,}^T\vec{f}^*-\jump{\,\vec{v}\cdot\vec{f}\,}  + \average{\Delta x\vec{v}}^T{\vec{s}}_{i+\tfrac{1}{2}}.
\end{equation}
We want the discrete entropy update to satisfy the discrete entropy conservation law. To achieve this, we require
{\color{black}{\begin{equation}\label{entropyConservationCondition1}
\jump{\,\vec{v}\,}^T\vec{f}^* -\jump{\,\vec{v}\cdot\vec{f}\,} + \average{\Delta x\vec{v}}^T{\vec{s}}_{i+\tfrac{1}{2}}  = -\jump{\, F\,}.
\end{equation}}}
We combine the known entropy potential $\phi$ in \eqref{EntropyPotential} and the linearity of the jump operator to rewrite the entropy conservation condition \eqref{entropyConservationCondition1} as 
\begin{equation}\label{entropyConservationCondition2}
\jump{\,\vec{v}\,}^T\vec{f}^* = \jump{\, \rho u \,} +\jump{\,\frac{\rho u\|\vec{B}\|^2}{2p}\,} - \jump{\,\frac{\rho B_1(\vec{u}\cdot\vec{B})}{p}\,} - \average{\Delta x \vec{v}}^T{\vec{s}}_{i+\tfrac{1}{2}}.
\end{equation}
We denote the constraint \eqref{entropyConservationCondition2} as the discrete entropy conserving condition. This is a single condition on the numerical flux vector $\vec{f}^*$, so there are many potential solutions for the entropy conserving flux. Recall, however, that we have the additional requirement that the numerical flux must be consistent \eqref{consistency}. We develop the expression for $\vec{f}^*$ in Sec. \ref{Sec:EntropyConservingFlux} where we will see that the discretization of the source term ${\vec{s}_{i+1/2}}$ plays an important role to ensure that \eqref{entropyConservationCondition2} is satisfied.

\section{Derivation of an Entropy Stable Numerical Flux}\label{EntropyFlux}

With the necessary entropy variable and Jacobian definitions as well as the formulation of the discrete entropy conserving condition \eqref{entropyConservationCondition2} we are ready to derive an affordable entropy conserving numerical flux in Sec. \ref{Sec:EntropyConservingFlux}. As we previously noted, entropy conserving methods may suffer breakdown in the presence of shocks \cite{carpenter_esdg}. Thus, in Sec. \ref{Sec:StableFlux}, we will design an entropy stable numerical flux that uses the entropy conserving flux as a base and incorporates a dissipation term required for stability.

\subsection{Entropy Conserving Numerical Flux for the 1D Ideal MHD Equations}\label{Sec:EntropyConservingFlux}

We have previously defined the arithmetic mean. To derive an entropy conserving flux we will also require the logarithmic mean
\begin{equation}\label{logMean}
(\cdot)^{\ln} \coloneqq \frac{(\cdot)_L - (\cdot)_R}{\ln((\cdot)_L) - \ln((\cdot)_R)}.
\end{equation}
A {\color{black}{numerically stable}} procedure to compute the logarithmic mean is described by Ismail and Roe \cite{ismail2009} (Appendix B). 

\begin{thm}(Entropy Conserving Numerical Flux)
If we introduce the parameter vector
\begin{equation}\label{zinProof}
\vec{z} = \left[\sqrt{\frac{\rho}{p}},\sqrt{\frac{\rho}{p}}u,\sqrt{\frac{\rho}{p}}v,\sqrt{\frac{\rho}{p}}w,\sqrt{\rho p},B_1, B_2, B_3\right]^T,
\end{equation}
the averaged quantities for the primitive variables and products
\begin{equation}
\begin{aligned}
&\hat{\rho} = \average{z_1}z_5^{\ln},\;\;\hat{u}_1=\frac{\average{z_2}}{\average{z_1}},\;\; \hat{v}_1=\frac{\average{z_3}}{\average{z_1}},\;\;\hat{w}_1 = \frac{\average{z_4}}{\average{z_1}},\;\; \hat{p}_1 = \frac{\average{z_5}}{\average{z_1}},\;\;\hat{p}_2 = \frac{\gamma+1}{2\gamma}\frac{z_5^{\ln}}{z_1^{\ln}} + \frac{\gamma-1}{2\gamma}\frac{\average{z_5}}{\average{z_1}}, \\[0.1cm]
&\qquad\hat{u}_2 = \frac{\average{z_1 z_2}}{\average{z_1^2}},\;\;\hat{v}_2 = \frac{\average{z_1 z_3}}{\average{z_1^2}},\;\;\hat{w}_2 = \frac{\average{z_1 z_4}}{\average{z_1^2}},\;\;\hat{B_1} = \average{z_6},\;\;\hat{B_2} = \average{z_7},\;\;\hat{B_3} = \average{z_8}, \\[0.1cm]
&\qquad\;\;\;\accentset{\circ}{B_1} = \average{z_6^2},\;\;\accentset{\circ}{B_2} = \average{z_7^2},\;\;\accentset{\circ}{B_3} = \average{z_8^2},\;\;\widehat{B_1 B_2} = \average{z_6 z_7},\;\;\widehat{B_1B_3} = \average{z_6 z_8},
\end{aligned}
\end{equation}
and discretize the source term in the finite volume method to contribute to each element as
\begin{equation}\label{SourceTermDisc}
\vec{s}_i = {\color{black}{\frac{1}{2}}}\left(\vec{s}_{i+\tfrac{1}{2}} + \vec{s}_{i-\tfrac{1}{2}}\right) = -\frac{1}{2}\left(
\jump{B_1}_{i+\tfrac{1}{2}}\begin{bmatrix}
0\\
0\\
0\\
0\\
0\\
\frac{\average{z_1 z_2}\hat{B_1}}{\average{\Delta x z_1^2 B_1}}\\[0.15cm]
\frac{\average{z_1 z_3}\hat{B_2}}{\average{\Delta x z_1^2 B_2}}\\[0.15cm]
\frac{\average{z_1 z_4}\hat{B_3}}{\average{\Delta x z_1^2 B_3}}
\end{bmatrix}_{i+\tfrac{1}{2}}
+
\jump{B_1}_{i-\tfrac{1}{2}}\begin{bmatrix} 
0\\
0\\
0\\
0\\
0\\
\frac{\average{z_1 z_2}\hat{B_1}}{\average{\Delta x z_1^2 B_1}}\\[0.15cm]
\frac{\average{z_1 z_3}\hat{B_2}}{\average{\Delta x z_1^2 B_2}}\\[0.15cm]
\frac{\average{z_1 z_4}\hat{B_3}}{\average{\Delta x z_1^2 B_3}}
\end{bmatrix}_{i-\tfrac{1}{2}}
\right),
\end{equation}
then we can determine a discrete, entropy conservative flux to be
\begin{equation}\label{Eq:entropyconservative}
\vec{f}^{*,ec} = \begin{bmatrix}
\hat{\rho}\hat{u}_1 \\
\hat{p}_1 + \hat{\rho}\hat{u}_1^2 + \frac{1}{2}\left(\accentset{\circ}{B_1}+\accentset{\circ}{B_2}+\accentset{\circ}{B_3}\right) - \accentset{\circ}{B_1} \\
\hat{\rho}\hat{u}_1\hat{v}_1 - \widehat{B_1 B_2} \\ 
\hat{\rho}\hat{u}_1\hat{w}_1 -\widehat{B_1 B_3} \\
\frac{\gamma \hat{u}_1\hat{p}_2}{\gamma -1} + \frac{\hat{\rho}\hat{u}_1}{2}(\hat{u}_1^2 + \hat{v}_1^2 + \hat{w}_1^2) + \hat{u}_2\left(\hat{B}_2^2 +\hat{B}_3^2\right)- \hat{B_1}\left(\hat{v}_2\hat{B_2} +\hat{w}_2\hat{B_3}\right) \\
0 \\
\hat{u}_2\hat{B_2} - \hat{v}_2\hat{B_1} \\
\hat{u}_2\hat{B_3} - \hat{w}_2\hat{B_1}
\end{bmatrix}.
\end{equation}
\end{thm}

\begin{proof}
To derive an affordable entropy conservative flux for the one-dimensional ideal MHD equations we first expand the discrete entropy conserving condition \eqref{entropyConservationCondition2} componentwise to find
\begin{equation}\label{componentJumpCondition}
\begin{aligned}
-f_1^*\left(\frac{\jump{s}}{\gamma-1} + \jump{\frac{\rho\|\vec{u}\|^2}{2p}}\right) + f_2^*\jump{\frac{\rho u}{p}}+ f_3^*\jump{\frac{\rho v}{p}}&+ f_4^*\jump{\frac{\rho w}{p}}- f_5^*\jump{\frac{\rho}{p}}+ f_6^*\jump{\frac{\rho B_1}{p}}+ f_7^*\jump{\frac{\rho B_2}{p}}+ f_8^*\jump{\frac{\rho B_3}{p}} \\ &=
\jump{\, \rho u \,} +\jump{\,\frac{\rho u\|\vec{B}\|^2}{2p}\,} - \jump{\,\frac{\rho B_1(\vec{u}\cdot\vec{B})}{p}\,} - \average{\Delta x \vec{v}}^T{\vec{s}}_{i+\tfrac{1}{2}}.
\end{aligned}
\end{equation}
To determine the unknown components of $\vec{f}^*$ we want to expand each jump term in \eqref{componentJumpCondition} into linear jump components. This will provide us with a system of eight equations from which we determine $\vec{f}^*$. To obtain linear jumps we define the parameter vector $\vec{z}$ such that there is no mixing of the hydrodynamic and magnetic field variables:
\begin{equation}
\vec{z} \coloneqq \left[\sqrt{\frac{\rho}{p}},\sqrt{\frac{\rho}{p}}u,\sqrt{\frac{\rho}{p}}v,\sqrt{\frac{\rho}{p}}w,\sqrt{\rho p},B_1, B_2, B_3\right]^T,
\end{equation}
with the identities
\begin{equation}\label{zIdentities}
\begin{aligned}
\frac{\rho}{p} = z_1^2,\quad\frac{\rho u}{p} = z_1z_2,\quad\frac{\rho v}{p} = z_1z_3,\quad\frac{\rho w}{p} = z_1z_4,\quad\frac{\rho u^2}{p} = z_2^2,\quad\frac{\rho v^2}{p} = z_3^2,\quad\frac{\rho w^2}{p} = z_4^2,\quad \rho u = z_2z_5, \\[0.1cm]
s = -(\gamma-1)\ln(z_5)-(\gamma+1)\ln(z_1),\quad\frac{\rho B_1}{p} = z_1^2 z_6,\quad\frac{\rho B_2}{p} = z_1^2 z_7,\quad\frac{\rho B_3}{p} = z_1^2 z_8,\quad\frac{\rho u B_1^2}{2p} = \frac{1}{2}z_1 z_2 z_6^2, \\[0.1cm]
\frac{\rho u B_2^2}{2p} = \frac{1}{2}z_1 z_2 z_7^2,\quad\frac{\rho u B_3^2}{2p} = \frac{1}{2}z_1 z_2 z_8^2,\quad\frac{\rho v B_1 B_2}{p} = z_1 z_3 z_6 z_7,\quad\frac{\rho w B_1 B_3}{p} = z_1 z_4 z_6 z_8.\qquad\quad
\end{aligned}
\end{equation}
Finally, we use the logarithmic mean \eqref{logMean} to rewrite the jump in the entropy as
\begin{equation}
\jump{s} = -(\gamma-1)\frac{\jump{z_5}}{z_5^{\ln}} - (\gamma+1)\frac{\jump{z_1}}{z_1^{\ln}}.
\end{equation}

We use the parameter identities \eqref{zIdentities} and algebraic identities of the jump operator
\begin{equation}
\jump{ab} = \average{a}\jump{b} + \average{b}\jump{a},\quad \jump{a^2} = 2\average{a}\jump{a},
\end{equation}
to rewrite the left and right hand sides of the entropy conserving condition \eqref{componentJumpCondition}. First we have from the left of the discrete entropy conserving condition \eqref{entropyConservationCondition2}:
\begin{equation}\label{LHSCondition}
\begin{aligned}
f_1^*&\left\{\left(\frac{\jump{z_5}}{z_5^{\ln}} + \frac{\gamma+1}{\gamma-1}\frac{\jump{z_1}}{z_1^{\ln}}\right) - \average{z_2}\jump{z_2} - \average{z_3}\jump{z_3} - \average{z_4}\jump{z_4}\right\} + f_2^*(\average{z_2}\jump{z_1} + \average{z_1}\jump{z_2}) \\[0.1cm]
&+ f_3^*(\average{z_3}\jump{z_1} + \average{z_1}\jump{z_3})+ f_4^*(\average{z_4}\jump{z_1} + \average{z_1}\jump{z_4}) - 2f_5^*\average{z_1}\jump{z_1} \\[0.1cm] &+ f_6^*(2\average{z_1}\average{z_6}\jump{z_1}+\average{z_1^2}\jump{z_6}) \\[0.1cm]&+ f_7^*(2\average{z_1}\average{z_7}\jump{z_1}+\average{z_1^2}\jump{z_7}) \\[0.1cm] &+ f_8^*(2\average{z_1}\average{z_8}\jump{z_1}+\average{z_1^2}\jump{z_8}).
\end{aligned}
\end{equation}
Next we expand the right hand side of \eqref{componentJumpCondition} into a combination of linear jumps:
\begin{align}\label{RHSCondition}\nonumber
\jump{\rho u} &= \jump{z_2 z_5} = \average{z_5}\jump{z_2} + \average{z_2}\jump{z_5}, \\[0.1cm]\nonumber
\jump{\frac{\rho u B_1^2}{2p}} &= \frac{1}{2}\jump{z_1 z_2 z_6^2} = \frac{1}{2}\average{z_2}\average{z_6^2}\jump{z_1} + \frac{1}{2}\average{z_1}\average{z_6^2}\jump{z_2} + \average{z_1z_2}\average{z_6}\jump{z_6}, \\[0.1cm]\nonumber
\jump{\frac{\rho u B_2^2}{2p}} &= \frac{1}{2}\jump{z_1 z_2 z_7^2} = \frac{1}{2}\average{z_2}\average{z_7^2}\jump{z_1} + \frac{1}{2}\average{z_1}\average{z_7^2}\jump{z_2} + \average{z_1z_2}\average{z_7}\jump{z_7}, \\[0.1cm]
\jump{\frac{\rho u B_3^2}{2p}} &= \frac{1}{2}\jump{z_1 z_2 z_8^2} = \frac{1}{2}\average{z_2}\average{z_8^2}\jump{z_1} + \frac{1}{2}\average{z_1}\average{z_8^2}\jump{z_2} + \average{z_1z_2}\average{z_8}\jump{z_8}, \\[0.1cm]\nonumber
\jump{\frac{\rho u B_1^2}{p}} &= \jump{z_1 z_2 z_6^2} = \average{z_2}\average{z_6^2}\jump{z_1} +\average{z_1}\average{z_6^2}\jump{z_2} + 2\average{z_1z_2}\average{z_6}\jump{z_6}, \\[0.1cm]\nonumber
\jump{\frac{\rho v B_1 B_2}{p}} &= \jump{z_1 z_3 z_6 z_7} = \average{z_3}\average{z_6 z_7}\jump{z_1} +\average{z_1}\average{z_6 z_7}\jump{z_3} + \average{z_1z_3}\average{z_7}\jump{z_6} +\average{z_1 z_3}\average{z_6}\jump{z_7}, \\[0.1cm]\nonumber
\jump{\frac{\rho w B_1 B_3}{p}} &= \jump{z_1 z_4 z_6 z_8} = \average{z_4}\average{z_6 z_8}\jump{z_1} +\average{z_1}\average{z_6 z_8}\jump{z_4} + \average{z_1z_4}\average{z_8}\jump{z_6} +  \average{z_1z_4}\average{z_6}\jump{z_8}.
\end{align}
Finally, we expand the source term contribution on the right hand side. For now we leave the specific discretization of the source term general as a consistent approximation will reveal itself in the later analysis. First we note that the derivative term in the Janhunen source term \eqref{JanhunenSource} is of the form 
\begin{equation}\label{firstOrderApprox}
    \pderivative{}{x}\!\left(B_1\right) \approx \frac{\jump{B_1}}{\Delta x}=\frac{\jump{z_6}}{\Delta x}. 
\end{equation}
The source term in cell $i$ contributes to interface $i+\tfrac{1}{2}$ and interface $i-\tfrac{1}{2}$. {\color{black}{We choose the source term to be of the form
\begin{equation}\label{sourceTermReminder1}
{\vec{s}}_{i+\tfrac{1}{2}} = \jump{z_6}\begin{bmatrix}0\\0\\0\\0\\0\\s_6\\s_7\\s_8\end{bmatrix},
\end{equation}
where we have extra degrees of freedom when selecting $s_6$, $s_7$, and $s_8$ in \eqref{sourceTermReminder1}. Then, the source term contribution is given by
\begin{equation}\label{sourceTermReminder2}
-\average{\Delta x \vec{v}}^T{\vec{s}}_{i+\tfrac{1}{2}} = -\jump{z_6}\left(\average{\Delta x z_1^2z_6}s_6 + \average{\Delta x z_1^2z_7}s_7 + \average{\Delta x z_1^2z_8}s_8\right).
\end{equation}}}
 
Every term in the discrete entropy conservation condition \eqref{componentJumpCondition} is now rewritten into linear jump components of the parameter vector $\vec{z}$. Though algebraically laborious this provides us with a set of eight equations for which we can determine the yet unknown components in the entropy conserving numerical flux. Next we gather the like terms of each jump component. Once we have grouped all the like terms for each linear jump it will become clear how to discretize the Janhunen source term in order to guarantee consistency.  Gathering terms from \eqref{LHSCondition}, \eqref{RHSCondition}, and \eqref{sourceTermReminder2} we determine the system of eight equations:
\begin{align}
\jump{z_1}:& \quad\frac{\gamma+1}{\gamma-1}\frac{f_1^*}{z_1^{\ln}} + f_2^*\average{z_2} + f_3^*\average{z_3} + f_4^*\average{z_4} -2 f_5^*\average{z_1}+ 2f_6^*\average{z_1}\average{z_6} + 2f_7^*\average{z_1}\average{z_7} \label{Jump1}\\\nonumber 
&\;+2f_8^*\average{z_1}\average{z_8} = \frac{\average{z_2}}{2}\left(\average{z_6^2}+\average{z_7^2}+\average{z_8^2}\right) - \average{z_2}\average{z_6^2} - \average{z_3}\average{z_6 z_7} - \average{z_4}\average{z_6 z_8},  \\[0.1cm]
\jump{z_2}:& \quad-f_1^*\average{z_2} + f_2^*\average{z_1} = \average{z_5} + \frac{\average{z_1}}{2}\left(\average{z_6^2}+\average{z_7^2}+\average{z_8^2}\right) - \average{z_1}\average{z_6^2},\label{Jump2}\\[0.1cm]
\jump{z_3}:& \quad-f_1^*\average{z_3} + f_3^*\average{z_1} = -\average{z_1}\average{z_6 z_7}, \label{Jump3}\\[0.1cm]
\jump{z_4}:& \quad-f_1^*\average{z_4} + f_4^*\average{z_1} = -\average{z_1}\average{z_6 z_8}, \label{Jump4}\\[0.1cm]
\jump{z_5}:& \quad\frac{f_1^*}{z_5^{\ln}} = \average{z_2}, \label{Jump5}\\[0.1cm]
\jump{z_6}:& \quad f_6^*\average{z_1^2} = -\average{z_1z_2}\average{z_6} -\average{z_1z_3}\average{z_7} -\average{z_1z_4}\average{z_8} {\color{black}{- \average{\Delta x z_1^2z_6}s_6 }}\label{Jump6}\\\nonumber
&\qquad\qquad\qquad {\color{black}{- \average{\Delta x z_1^2z_7}s_7 - \average{\Delta x z_1^2z_8}s_8}}, \\[0.1cm]
\jump{z_7}:& \quad f_7^*\average{z_1^2} = \average{z_1 z_2}\average{z_7} - \average{z_1 z_3}\average{z_6}, \label{Jump7} \\[0.1cm]
\jump{z_8}:& \quad f_8^*\average{z_1^2} =  \average{z_1 z_2}\average{z_8} - \average{z_1 z_4}\average{z_6}. \label{Jump8} 
\end{align}

With the collection of equations \eqref{Jump1} - \eqref{Jump8} we find a rather alarming result. We know that the sixth component of the physical flux for the ideal MHD system is zero, i.e. $f_6 = 0$. However, we have found in our entropy conservation condition that the sixth component of the numerical flux (computed from \eqref{Jump6}) is 
\begin{equation}\label{sixthComponentProblem}
\begin{aligned}
f_6^* &= -\frac{1}{\average{z_1^2}}\left(\average{z_1z_2}\average{z_6} +\average{z_1z_3}\average{z_7} +\average{z_1z_4}\average{z_8}+ \average{\Delta x z_1^2z_6}s_6\right.\\&\left.\;\;\;\;\;+ \average{\Delta x z_1^2z_7}s_7 + \average{\Delta x z_1^2z_8}s_8\right).
\end{aligned}
\end{equation}
In general we cannot guarantee that \eqref{sixthComponentProblem} will vanish. In one dimension the argument could be made that $\jump{z_6} = \jump{B_1} = 0$ (as it is constant) and there is, in fact, no issue. However, this assumption is too restrictive to discuss higher dimensional entropy conservative flux formul\ae. Our assumption from Sec. \ref{GoverningEquations} that $B_1\ne constant$ revealed extra terms which otherwise would have been hidden from the analysis. 

To remove this inconsistency introduced by the $\jump{z_6}$ equation we discretize the source term to cancel the problematic terms \eqref{sixthComponentProblem}. 
{\color{black}{We compare the structure of the extra terms in \eqref{sixthComponentProblem} and the degrees of freedom $s_6$, $s_7$, and $s_8$ to determine a consistent discretization to cancel the extraneous terms in the $\jump{z_6}$ equation in \eqref{sixthComponentProblem}:
\begin{equation}\label{SourceTermDiscinProof}
\begin{aligned}
s_6 &= -\frac{\average{z_1z_2}\average{z_6}}{ \average{\Delta x z_1^2z_6}},\\
s_7 &= -\frac{\average{z_1z_3}\average{z_7}}{ \average{\Delta x z_1^2z_7}},\\
s_8 &= -\frac{\average{z_1z_4}\average{z_8}}{ \average{\Delta x z_1^2z_8}}.
\end{aligned}
\end{equation}
The source term at the $i-1/2$ interface has an identical structure to \eqref{SourceTermDiscinProof}. We collect the total source term discretization in cell $i$ for clarity:
 \begin{equation}\label{SourceInI}
 \vec{s}_i = \frac{1}{2}\left(\vec{s}_{i+\tfrac{1}{2}} + \vec{s}_{i-\tfrac{1}{2}}\right),
 \end{equation}
where
\begin{equation}\label{SourceTermDiscinProof2}
\begin{aligned}
\vec{s}_{i+\tfrac{1}{2}} &=-\jump{B_1}_{i+\tfrac{1}{2}}\left[
0,\;
0,\;
0,\;
0,\;
0,\;
\frac{\average{z_1z_2}\average{z_6}}{\average{\Delta x z_1^2z_6}},\;
\frac{\average{z_1z_3}\average{z_7}}{\average{\Delta x z_1^2z_7}},\;
\frac{\average{z_1z_4}\average{z_8}}{\average{\Delta x z_1^2z_8}}
\right]_{i+\tfrac{1}{2}}^T,\\
\vec{s}_{i-\tfrac{1}{2}} &=-\jump{B_1}_{i-\tfrac{1}{2}}\left[
0,\;
0,\;
0,\;
0,\;
0,\;
\frac{\average{z_1z_2}\average{z_6}}{\average{\Delta x z_1^2z_6}},\;
\frac{\average{z_1z_3}\average{z_7}}{\average{\Delta x z_1^2z_7}},\;
\frac{\average{z_1z_4}\average{z_8}}{\average{\Delta x z_1^2z_8}}
\right]_{i-\tfrac{1}{2}}^T.
\end{aligned}
\end{equation}
It is straightforward to check the consistency of the source term discretization \eqref{SourceInI}. }}

We substitute the source term discretization \eqref{SourceTermDiscinProof} into the entropy constraint \eqref{Jump6} and find
that the source term components exactly cancel the extraneous terms in the $\jump{z_6}$ equation \eqref{sixthComponentProblem}. Thus, we recover a consistent term for the sixth numerical flux component and it is now true that 
\begin{equation}
f_6^* = 0.
\end{equation}
Finally, we are able solve the remaining seven equations \eqref{Jump1} - \eqref{Jump5}, \eqref{Jump7}, and \eqref{Jump8} for the numerical flux components:
\begin{align}
f_1^* &= \average{z_2} z_5^{\ln}, \label{flux1}\\[0.1cm] 
f_2^* &= \frac{\average{z_5}}{\average{z_1}} + \frac{\average{z_2}^2 z_5^{\ln}}{\average{z_1}} + \frac{1}{2}\left(\average{z_6^2}+\average{z_7^2}+\average{z_8^2}\right) - \average{z_6^2}, \label{flux2} \\[0.1cm] 
f_3^* &= \frac{\average{z_2}\average{z_3}z_5^{\ln}}{\average{z_1}} - \average{z_6 z_7}, \label{flux3}\\[0.1cm] 
f_4^* &= \frac{\average{z_2}\average{z_4}z_5^{\ln}}{\average{z_1}} - \average{z_6 z_8}, \label{flux4}\\[0.1cm] 
f_5^* &= \frac{\gamma}{\gamma-1}\frac{\average{z_2}}{\average{z_1}}\left(\frac{\gamma +1}{2\gamma}\frac{z_5^{\ln}}{z_1^{\ln}} + \frac{\gamma-1}{2\gamma}\frac{\average{z_5}}{\average{z_1}}\right) + \frac{\average{z_2}z_5^{\ln}}{2}\left(\frac{\average{z_2}^2}{\average{z_1}^2} + \frac{\average{z_3}^2}{\average{z_1}^2} + \frac{\average{z_4}^2}{\average{z_1}^2}\right) \label{flux5} \\ 
&\qquad +\frac{\average{z_7}}{\average{z_1^2}}\left(\average{z_1 z_2}\average{z_7} - \average{z_1 z_3}\average{z_6}\right) + \frac{\average{z_8}}{\average{z_1^2}}\left(\average{z_1 z_2}\average{z_8} - \average{z_1 z_4}\average{z_6}\right), \nonumber\label{flux6}\\[0.1cm] 
f_6^* &= 0, \\[0.1cm] 
f_7^* &= \frac{1}{\average{z_1^2}}\left( \average{z_1 z_2}\average{z_7} - \average{z_1 z_3}\average{z_6}\right),\label{flux7}\\[0.1cm] 
f_8^* &= \frac{1}{\average{z_1^2}}\left( \average{z_1 z_2}\average{z_8} - \average{z_1 z_4}\average{z_6}\right).\label{flux8}
\end{align}

The newly derived numerical flux \eqref{flux1} - \eqref{flux8} conserves the discrete entropy by construction. Next we verify that the numerical flux $\vec{f}^{*,ec}$ is consistent to the physical flux. It will make the demonstration of consistency more straightforward if we rewrite the fifth component of the physical flux, $f_5$, by substituting the definition of $\rho e$:
\begin{equation}
\begin{aligned}
f_5 = u(\rho e + p + \frac{1}{2}\|\vec{B}\|^2) - B_1(\vec{u}\cdot{\vec{B}}) &= u\left(\frac{\rho}{2}\|\vec{u}\|^2 + \frac{p}{\gamma-1} + \frac{1}{2}\|\vec{B}\|^2+ p + \frac{1}{2}\|\vec{B}\|^2\right) - B_1(uB_1 + vB_2 + wB_3), \\ 
&= \frac{\rho u}{2}\|\vec{u}\|^2 + \frac{\gamma u p}{\gamma -1} + uB_2^2 + uB_3^2 - vB_1B_2 - wB_1B_3.
\end{aligned}
\end{equation}
Now, if we assume that the left and right states are identical in the numerical flux \eqref{flux1} - \eqref{flux8}, we find that 
\begin{align}
f_1^* &\rightarrow \rho u &= f_1 \label{consistent1},\\[0.1cm] 
f_2^* &\rightarrow p + \rho u^2 + \frac{1}{2}\|\vec{B}\|^2 - B_1^2 &= f_2 \label{consistent2},\\[0.1cm] 
f_3^* &\rightarrow \rho uv - B_1 B_2 &= f_3 \label{consistent3},\\[0.1cm] 
f_4^* &\rightarrow \rho uw - B_1B_3 &= f_4 \label{consistent4},\\[0.1cm] 
f_5^* &\rightarrow \frac{\gamma u p}{\gamma-1} + \frac{\rho u}{2}\|\vec{u}\|^2 + uB_2^2 + uB_3^2 - vB_1B_2 - wB_1B_3 &= f_5 \label{consistent5},\\[0.1cm] 
f_6^* &\rightarrow 0 &= f_6 \label{consistent6},\\[0.1cm] 
f_7^* &\rightarrow uB_2 - vB_1 &= f_7 \label{consistent7},\\[0.1cm] 
f_8^* &\rightarrow uB_3 - wB_1 &= f_8 \label{consistent8}.
\end{align}
Thus, we have shown that the numerical flux given by \eqref{flux1} - \eqref{flux8} is consistent and entropy conservative.

Though the parametrization vector $\vec{z}$ was necessary to develop the entropy conserving numerical flux it obfuscates the information about which quantities are being averaged and their contribution in the numerical flux. It is therefore convenient to define how the primitive variables, through averages, are related to the parametrized numerical flux given by \eqref{flux1} - \eqref{flux8}. We select a computationally efficient averaging procedure identical to that of Ismail and Roe \cite{ismail2009} for the hydrodynamic variables, however the magnetic field terms will necessitate extra variable definitions:
\begin{equation}
\begin{aligned}
&\hat{\rho} = \average{z_1}z_5^{\ln},\;\;\hat{u}_1=\frac{\average{z_2}}{\average{z_1}},\;\; \hat{v}_1=\frac{\average{z_3}}{\average{z_1}},\;\;\hat{w}_1 = \frac{\average{z_4}}{\average{z_1}},\;\; \hat{p}_1 = \frac{\average{z_5}}{\average{z_1}},\;\;\hat{p}_2 = \frac{\gamma+1}{2\gamma}\frac{z_5^{\ln}}{z_1^{\ln}} + \frac{\gamma-1}{2\gamma}\frac{\average{z_5}}{\average{z_1}}, \\[0.1cm]
&\qquad\hat{u}_2 = \frac{\average{z_1 z_2}}{\average{z_1^2}},\;\;\hat{v}_2 = \frac{\average{z_1 z_3}}{\average{z_1^2}},\;\;\hat{w}_2 = \frac{\average{z_1 z_4}}{\average{z_1^2}},\;\;\hat{B_1} = \average{z_6},\;\;\hat{B_2} = \average{z_7},\;\;\hat{B_3} = \average{z_8}, \\[0.1cm]
&\qquad\;\;\;\accentset{\circ}{B_1} = \average{z_6^2},\;\;\accentset{\circ}{B_2} = \average{z_7^2},\;\;\accentset{\circ}{B_3} = \average{z_8^2},\;\;\widehat{B_1 B_2} = \average{z_6 z_7},\;\;\widehat{B_1B_3} = \average{z_6 z_8},
\end{aligned}
\end{equation}
Then we can write, respectively, the entropy conserving flux and source term discretization in the more illuminating forms
\begin{equation}\label{MHDEntConsinProof}
\vec{f}^{*,ec} = \begin{bmatrix}
\hat{\rho}\hat{u}_1 \\
\hat{p}_1 + \hat{\rho}\hat{u}_1^2 + \frac{1}{2}\left(\accentset{\circ}{B_1}+\accentset{\circ}{B_2}+\accentset{\circ}{B_3}\right) - \accentset{\circ}{B_1} \\
\hat{\rho}\hat{u}_1\hat{v}_1 - \widehat{B_1 B_2} \\ 
\hat{\rho}\hat{u}_1\hat{w}_1 -\widehat{B_1 B_3} \\
\frac{\gamma \hat{u}_1\hat{p}_2}{\gamma -1} + \frac{\hat{\rho}\hat{u}_1}{2}(\hat{u}_1^2 + \hat{v}_1^2 + \hat{w}_1^2) + \hat{u}_2\left(\hat{B}_2^2 +\hat{B}_3^2\right)- \hat{B_1}\left(\hat{v}_2\hat{B_2} +\hat{w}_2\hat{B_3}\right) \\
0 \\
\hat{u}_2\hat{B_2} - \hat{v}_2\hat{B_1} \\
\hat{u}_2\hat{B_3} - \hat{w}_2\hat{B_1}
\end{bmatrix},
\end{equation}
and
\begin{equation}\label{SourceTermDiscat the End}
\vec{s}_i ={\color{black}{\frac{1}{2}}}\left( \vec{s}_{i+\tfrac{1}{2}} + \vec{s}_{i-\tfrac{1}{2}}\right) = -\frac{1}{2}\left(
\jump{B_1}_{i+\tfrac{1}{2}}\begin{bmatrix}
0\\
0\\
0\\
0\\
0\\
\frac{\average{z_1 z_2}\hat{B_1}}{\average{\Delta x z_1^2 B_1}}\\[0.15cm]
\frac{\average{z_1 z_3}\hat{B_2}}{\average{\Delta x z_1^2 B_2}}\\[0.15cm]
\frac{\average{z_1 z_4}\hat{B_3}}{\average{\Delta x z_1^2 B_3}}
\end{bmatrix}_{i+\tfrac{1}{2}}
+
\jump{B_1}_{i-\tfrac{1}{2}}\begin{bmatrix} 
0\\
0\\
0\\
0\\
0\\
\frac{\average{z_1 z_2}\hat{B_1}}{\average{\Delta x z_1^2 B_1}}\\[0.15cm]
\frac{\average{z_1 z_3}\hat{B_2}}{\average{\Delta x z_1^2 B_2}}\\[0.15cm]
\frac{\average{z_1 z_4}\hat{B_3}}{\average{\Delta x z_1^2 B_3}}
\end{bmatrix}_{i-\tfrac{1}{2}}
\right).
\end{equation}
\end{proof}

\begin{rem} \textit{(Consistency with Entropy Conserving Euler Flux)}
If the flow occurs in a medium without a magnetic field then the ideal MHD model becomes the compressible Euler equations. The structure of $\vec{f}^{*,ec}$ \eqref{MHDEntConsinProof} can be separated into an Euler component and magnetic field component, i.e.,
\begin{equation}
\vec{f}^{*,ec} = \begin{bmatrix}
\hat{\rho}\hat{u}_1 \\
\hat{p}_1 + \hat{\rho}\hat{u}_1^2\\
\hat{\rho}\hat{u}_1\hat{v}_1  \\ 
\hat{\rho}\hat{u}_1\hat{w}_1\\
\frac{\gamma \hat{u}_1\hat{p}_2}{\gamma -1} + \frac{\hat{\rho}\hat{u}_1}{2}(\hat{u}_1^2 + \hat{v}_1^2 + \hat{w}_1^2) \\
0 \\
0 \\
0
\end{bmatrix}
+
\begin{bmatrix}
0 \\
\frac{1}{2}\left(\accentset{\circ}{B_1}+\accentset{\circ}{B_2}+\accentset{\circ}{B_3}\right) - \accentset{\circ}{B_1} \\
- \widehat{B_1 B_2} \\ 
-\widehat{B_1 B_3} \\
\hat{u}_2\left(\hat{B}_2^2 +\hat{B}_3^2\right)- \hat{B_1}\left(\hat{v}_2\hat{B_2} +\hat{w}_2\hat{B_3}\right) \\
0 \\
\hat{u}_2\hat{B_2} - \hat{v}_2\hat{B_1} \\
\hat{u}_2\hat{B_3} - \hat{w}_2\hat{B_1}
\end{bmatrix}.
\end{equation}
Thus, if the magnetic field components are zero we find 
\begin{equation}
\vec{f}^{*,ec} = \begin{bmatrix}
\hat{\rho}\hat{u}_1 \\
\hat{p}_1 + \hat{\rho}\hat{u}_1^2\\
\hat{\rho}\hat{u}_1\hat{v}_1  \\ 
\hat{\rho}\hat{u}_1\hat{w}_1\\
\frac{\gamma \hat{u}_1\hat{p}_2}{\gamma -1} + \frac{\hat{\rho}\hat{u}_1}{2}(\hat{u}_1^2 + \hat{v}_1^2 + \hat{w}_1^2) \\
0 \\
0 \\
0
\end{bmatrix}
+
\begin{bmatrix}
0 \\
0\\
0\\
0\\
0 \\
0\\
0\\
0
\end{bmatrix}.
\end{equation}
and $\vec{f}^{*,ec}$ becomes the entropy conserving flux for the Euler equations described by Ismail \& Roe \cite{ismail2009}. {\color{black}{This separation of the Euler components and magnetic components is useful as it grants flexibility in the underlying flux for the hydrodynamic components. In \ref{EKEP}, inspired by the work of Chandrashekar \cite{chandrashekar2013}, we outline an alternative numerical flux that is entropy and kinetic energy conserving.}}
\end{rem}

\begin{rem} \textit{(Multi-Dimensional Entropy Conserving Fluxes)}
A similar form of the proof of entropy conservation for the flux $\vec{f}$ in the $x-$direction can be used to derive the entropy conservative fluxes for the ideal MHD equations in $y$ and $z-$directions. Full details are given in \ref{3DFluxes} of this paper.
\end{rem}

\subsection{Dissipation Terms for an Entropy Stable Flux}\label{Sec:StableFlux}

To create an entropy stable numerical flux function we use the entropy conserving flux \eqref{Eq:entropyconservative} as a base and subtract a general form of numerical dissipation, e.g,
\begin{equation}\label{dissipation}
\vec{f}^* = \vec{f}^{*,ec} - \frac{1}{2}\matrix{D}\jump{\vec{q}},
\end{equation}
where $\matrix{D}$ is a dissipation matrix. Of utmost concern for entropy stability of the approximation is to formulate the dissipation term \eqref{dissipation} such that it is guaranteed to cause a negative contribution in the discrete entropy equation. {\color{black}{To guarantee entropy stability we will reformulate the dissipation term \eqref{dissipation} to incorporate the jump in the entropy variables (rather than the jump in conservative variables) \cite{barth99}.}} The remainder of this section is divided as follows: we will select a specific form for the dissipation matrix $\matrix{D}$ in Sec. \ref{Sec:DissMat}. Next, in Sec. \ref{Sec:Eigen}, we examine the eigenstructure of the particular dissipation matrix chosen. Finally, Sec. \ref{Sec:EntropySclaed} presents a specific entropy scaling on the eigenvectors of the dissipation matrix to guarantee the negativity of the dissipation term. 

\subsubsection{The Dissipation Matrix}\label{Sec:DissMat}

First, we select the dissipation matrix to be $|\widehat{\matrix{A}}|$ that is the absolute value of the flux Jacobian for the ideal MHD 8-wave formulation in the $x-$direction:
\begin{equation}\label{fluxJacobianConservativeVars}
\widehat{\matrix{A}}\coloneqq \vec{f}_{\vec{q}} + \matrix{P} = \matrix{A} + \matrix{P},
\end{equation}
where $\matrix{A}$ is the flux Jacobian for the homogeneous ideal MHD equations and $\matrix{P}$ is the Powell source term \cite{powell1994} written in matrix form, i.e.,
\begin{equation}\label{PowellMatrix}
\matrix{P}\pderivative{\vec{q}}{x} = \begin{bmatrix}
0 & 0 & 0 & 0 & 0 & 0 & 0 & 0 \\
0 & 0 & 0 & 0 & 0 & B_1 & 0 & 0 \\
0 & 0 & 0 & 0 & 0 & B_2 & 0 & 0 \\
0 & 0 & 0 & 0 & 0 & B_3 & 0 & 0 \\
0 & 0 & 0 & 0 & 0 & \vec{u}\cdot\vec{B} & 0 & 0 \\
0 & 0 & 0 & 0 & 0 & u & 0 & 0 \\
0 & 0 & 0 & 0 & 0 & v & 0 & 0 \\
0 & 0 & 0 & 0 & 0 & w & 0 & 0 
\end{bmatrix}
\pderivative{}{x}\begin{bmatrix}
\rho \\
\rho u\\
\rho v \\
\rho w\\
\rho e \\
B_1\\
B_2\\
B_3
\end{bmatrix} = \pderivative{B_1}{x}\begin{bmatrix}
0 \\
B_1 \\
B_2\\
B_3\\
\vec{u}\cdot\vec{B} \\
u\\
v\\
w
\end{bmatrix}.
\end{equation}
The fact that we force the positivity of $\matrix{D} = |\widehat{\matrix{A}}|$ does not guarantee that the dissipative term \eqref{dissipation}
is negative \cite{barth99}. Thus, in the remaining sections we will motivate a reformulation of the dissipation term \eqref{dissipation} in order to restore negativity.

It is important to distinguish that we {\color{black}{use}} the Janhunen source term \eqref{1DIdealMHDwithSource} to {\color{black}{derive an}} entropy {\color{black}{conservative numerical flux function}} in Sec. \ref{Sec:EntropyConservingFlux}. However, to design an entropy stable approximation we require {\color{black}{that}} the {\color{black}{eigendecomposition of the}} flux Jacobian matrix {\color{black}{can be related to the entropy Jacobian \eqref{entropyJacobian}. This particular scaling, first examined by Merriam \cite{merriam1989} and explored more thoroughly by Barth \cite{barth99}, requires that the PDE system is symmetrizable.}} Previous {\color{black}{analysis of}} the ideal MHD equations \cite{barth99,godunov1972} have demonstrated that the Powell source term is necessary to restore a symmetric MHD system. We {\color{black}{reiterate}} that the altered flux Jacobian is used only to derive the dissipation term. {\color{black}{Just as Lax-Friedrichs differs from Roe in the construction of a dissipation term, we use the Powell source term only to build our dissipation term.}} Thus, no inconsistency with the previous {\color{black}{entropy conserving flux}} derivations is introduced.

\subsubsection{The Eigenstructure of the Matrix $\widehat{\matrix{A}}$}\label{Sec:Eigen}
The background discussion of the eigenstructure of the augmented flux Jacobian matrix $\widehat{\matrix{A}}$ is algebraically intense, so for clarity we divide it into the following steps:
\begin{enumerate}
\item We compute the eigendecomposition for the symmetrizable MHD system written in the primitive variables.
\item We use previous results from Roe and Balsara \cite{roe1996} and rescale the eigenvectors to remove degeneracies.
\item We recover the, now stabilized, eigendecomposition for the matrix $\widehat{\matrix{A}}$.
\end{enumerate}

To discuss the eigenstructure of the matrix $\widehat{\matrix{A}}$ it is easiest to work with primitive variables, which we denote $\vec{p_r}$, and convert back to conservative variables, denoted by $\vec{q}$, when necessary. We first write the ideal MHD system modified by the Powell source term in terms of the conservative variables
\begin{equation}\label{MHDPowellSystem}
\pderivative{\vec{q}}{t} + \widehat{\matrix{A}}\pderivative{\vec{q}}{x} = 0.
\end{equation}
We are free to move between primitive and conservative variables in the system \eqref{MHDPowellSystem} with the matrix
\begin{equation}\label{MMatrix}
\matrix{M} = \vec{q}_{\vec{p_r}} = 
\begin{bmatrix}
1 & 0 & 0 & 0 & 0 & 0 & 0 & 0 \\
u & \rho & 0 & 0 & 0 &0&  0 & 0 \\
v & 0 & \rho & 0 & 0 &0& 0 & 0 \\
w & 0 & 0 & \rho & 0 & 0 &0& 0 \\
\frac{\|u\|^2}{2} & \rho u & \rho v & \rho w & \frac{1}{\gamma - 1} &B_1& B_2& B_3 \\
0 & 0 & 0 & 0 & 0 & 1 & 0&0 \\
0 & 0 & 0 & 0 & 0 & 0 &1 &0 \\
0 & 0 & 0 & 0 & 0 & 0 & 0&1 
\end{bmatrix},
\end{equation}
and from conservative to primitive variables with $\matrix{M}^{-1}$. Then we can rewrite the system \eqref{MHDPowellSystem} in terms of the primitive variables $\vec{p_r}$ 
\begin{equation}\label{MHDPowellPrim}
\pderivative{\vec{p_r}}{t} + \widehat{\matrix{B}}\pderivative{\vec{p_r}}{x} = 0,
\end{equation}
where 
\begin{equation}\label{BtoA}
\widehat{\matrix{B}} = \matrix{M}^{-1}\widehat{\matrix{A}}\matrix{M}.
\end{equation}

To describe the eigenstructure of the 8-wave ideal MHD system flux Jacobian in conservative variables $\widehat{\matrix{A}}$ we first investigate the eigendecompostion of $\widehat{\matrix{B}}$, the flux Jacobian in primitive variables:
\begin{equation}\label{Bmatrix}
\widehat{\matrix{B}} =
\begin{bmatrix}
u & \rho & 0 & 0 & 0 &0& 0 & 0 \\
0 & u & 0 & 0 & \frac{1}{\rho}&0&\frac{B_2}{\rho} & \frac{B_3}{\rho}  \\
0 & 0 & u & 0 & 0&0& -\frac{B_1}{\rho} & 0  \\
0 & 0 & 0 & u & 0 &0&0& -\frac{B_1}{\rho}  \\
0 & \gamma p & 0 & 0 & u &0& 0 & 0 \\
0 & 0 & 0 & 0 & 0&u& 0 & 0 \\
0 & B_2 & -B_1 & 0 &0& 0 & u & 0 \\
0 & B_3 & 0 & -B_1 & 0&0 & 0& u
\end{bmatrix}.
\end{equation}
From \eqref{BtoA} we see that we can convert the resulting eigendecomposition to conservative variables with straightforward matrix multiplication and the identity
\begin{equation}
\widehat{\matrix{A}} = \matrix{M}\widehat{\matrix{B}}\matrix{M}^{-1}.
\end{equation}
So, once we compute the eigendecomposition
\begin{equation}
\widehat{\matrix{B}} = \matrix{R}\widehat{\boldsymbol{\Lambda}}\matrix{R}^{-1},
\end{equation}
we can recover the eigendecomposition of the flux Jacobian matrix in conservative variables as
\begin{equation}
\widehat{\matrix{A}} = \matrix{M}\widehat{\matrix{B}}\matrix{M}^{-1} = \matrix{M}\matrix{R}\widehat{\boldsymbol{\Lambda}}\matrix{R}^{-1}\matrix{M}^{-1}=\widehat{\matrix{R}}\widehat{\boldsymbol{\Lambda}}\widehat{\matrix{R}}^{-1},\quad {\color{black}{\widehat{\matrix{R}} = \matrix{M}\matrix{R}}}.
\end{equation}

We begin with a naively scaled eigendecomposition of the matrix $\widehat{\matrix{B}}$. The 8-wave formulation supports eight traveling wave solutions with eigenvalues
\begin{equation}\label{eigenvalues}
\lambda_{\pm f} = u \pm c_f,\quad \lambda_{\pm s} =u \pm c_s,\quad \lambda_{\pm a} = u+c_a,\quad \lambda_{E} =u,\quad \lambda_{D} =u,
\end{equation}
where $c_f$, $c_s$ are the fast and slow magnetoacoustic wave speeds and $c_a$ is the Alfv\'{e}n wave speed. The double eigenvalue $u$ represent the entropy wave and divergence wave. The divergence wave is a direct result of including the Powell source. As we previously mentioned, the Powell source term turns the divergence wave into an advected scalar which is directly reflected in the eigenstructure. The values for the characteristic wave speeds may be written as 
\begin{equation}\label{characteristicSpeeds}
c_a^2 = b_1^2,\qquad c_{f,s}^2 = \frac{1}{2}(a^2+b^2)\pm\frac{1}{2}\sqrt{(a^2+b^2)^2 - 4a^2b_1^2},
\end{equation}
with the conventional notation
\begin{equation}
\vec{b} = \frac{\vec{B}}{\sqrt{\rho}},\quad b^2=b_1^2+b_2^2+b_3^2,\quad b_{\perp}^2 = b_2^2+b_3^2,\quad a^2 = \frac{p\gamma}{\rho}.
\end{equation}
In \eqref{characteristicSpeeds} the plus sign is for the fast speed $c_f$ and minus sign is the slow speed $c_s$. We also have the complete set of naively scaled eigenvectors (as presented in \cite{barth99})
\begin{itemize}
\item[] \underline{Entropy and Divergence Waves}: $\lambda_{E,D} = u$
\begin{equation}\label{entropyB}
\vec{r}_E = \begin{bmatrix} 1 \\ 0 \\0 \\0 \\0 \\0 \\0 \\0 \end{bmatrix},\quad\vec{r}_D = \begin{bmatrix} 0 \\ 0 \\0 \\0 \\0 \\1 \\0 \\0 \end{bmatrix},
\end{equation}
\item[] \underline{Alfv\'{e}n Waves}: $\lambda_{\pm a} = u\pm b_1$ 
\begin{equation}\label{AlfvenwaveB}
\vec{r}_{\pm a} = \begin{bmatrix}
0 \\
0 \\
\mp B_3 \\
\pm B_2 \\
0 \\
0 \\
-\sqrt{\rho} B_3 \\
\sqrt{\rho}B_2
\end{bmatrix},
\end{equation}
\item[]\underline{Magnetoacoustic Waves}: $\lambda_{\pm f,\pm s} = u\pm c_{f,s}$ 
\begin{equation}\label{MHDB}
\vec{r}_{\pm f} = \begin{bmatrix}
\rho \\[0.1cm]
\pm c_{f} \\[0.1cm]
\mp \frac{c_{f}b_1b_2}{c_{f}^2-b_1^2} \\[0.1cm]
\mp \frac{c_{f}b_1b_3}{c_{f}^2-b_1^2} \\[0.1cm]
\rho a^2 \\[0.1cm]
0 \\[0.1cm]
\frac{\sqrt{\rho}c_{f}^2b_2}{c_{f}^2-b_1^2} \\[0.1cm]
\frac{\sqrt{\rho}c_{f}^2b_3}{c_{f}^2-b_1^2} 
\end{bmatrix},
\qquad
\vec{r}_{\pm s} = \begin{bmatrix}
\rho \\[0.1cm]
\pm c_{s} \\[0.1cm]
\mp \frac{c_{s}b_1b_2}{c_{s}^2-b_1^2} \\[0.1cm]
\mp \frac{c_{s}b_1b_3}{c_{s}^2-b_1^2} \\[0.1cm]
\rho a^2 \\[0.1cm]
0 \\[0.1cm]
\frac{\sqrt{\rho}c_{s}^2b_2}{c_{s}^2-b_1^2} \\[0.1cm]
\frac{\sqrt{\rho}c_{s}^2b_3}{c_{s}^2-b_1^2} 
\end{bmatrix}.
\end{equation}
\end{itemize}
In this form the magnetoacoustic eigenvectors exhibit several forms of degeneracy that are carefully described by Roe and Balsara \cite{roe1996}.

We follow the same rescaling procedure of Roe and Balsara for the fast/slow magnetoacoustic eigenvectors \eqref{MHDB}. The algebra is simplified greatly if we introduce the parameters
\begin{equation}\label{rescaleParams}
\alpha_f^2 = \frac{a^2 - c_s^2}{c_f^2 - c_s^2},\quad\alpha_s^2 = \frac{c_f^2 - a^2}{c_f^2 - c_s^2}.
\end{equation}
The parameters \eqref{rescaleParams} have several useful properties:
\begin{equation}\label{rescaleIdent}
\alpha_f^2+\alpha_s^2 = 1,\quad \alpha_f^2c_f^2 + \alpha_s^2c_s^2 = a^2,\quad \alpha_f\alpha_s = \frac{a^2b_{\perp}}{c_f^2-c_s^2}.
\end{equation}
The parameters $\alpha_{f,s}$ measure how closely the fast/slow waves approximate the behavior of acoustic waves \cite{roe1996}. In the rescaling process we utilize several identities that arise from the quartic equation for the magnetoacoustic wave speeds $\pm c_{f,s}$
\begin{equation}
c^4 -(a^2+b^2)c^2 + a^2b_1^2 = 0,
\end{equation}
which are
\begin{equation}\label{identities2}
\begin{aligned}
c_fc_s &= a|b_1|, \\[0.1cm] c_f^2 + c_s^2 &= a^2 + b^2, \\[0.1cm] c_{f,s}^4-a^2b_1^2 &= c_{f,s}^2\left(c_{f,s}^2 - c_{s,f}^2\right),\\[0.1cm] \left(c_{f,s}^2 - a^2\right)\left(c_{f,s}^2 - b_1^2\right) &= c_{f,s}^2b_{\perp}^2.
\end{aligned}
\end{equation}
Applying the identities \eqref{rescaleIdent} and \eqref{identities2} we rewrite the eigenvectors for the fast/slow waves in a more stable form in terms of the parameters $\alpha_{f,s}$. We also rewrite the Alfv\'{e}n wave vectors in terms of the variable $\vec{b}$ for convenience:
\begin{itemize}
\item[] \underline{Entropy and Divergence Waves}: $\lambda_{E,D} = u$
\begin{equation}\label{entropyR}
\vec{r}_E = \begin{bmatrix} 1 \\ 0 \\0 \\0 \\0 \\0 \\0 \\0 \end{bmatrix},\quad\vec{r}_D = \begin{bmatrix} 0 \\ 0 \\0 \\0 \\0 \\1 \\0 \\0 \end{bmatrix},
\end{equation}
\item[] \underline{Alfv\'{e}n Waves}: $\lambda_{\pm a} = u\pm b_1$ 
\begin{equation}\label{AlfvenR}
\vec{r}_{\pm a} = \begin{bmatrix}
0 \\
0 \\
\mp \sqrt{\rho}\,b_3 \\
\pm \sqrt{\rho}\,b_2 \\
0 \\
0 \\
-\rho b_3 \\
\rho b_2
\end{bmatrix},
\end{equation}
\item[] \underline{Magnetoacoustic Waves}: $\lambda_{\pm f,\pm s} = u\pm c_{f,s}$ 
\begin{equation}\label{MHDR}
\vec{r}_{\pm f} = \begin{bmatrix}
\alpha_f\rho \\[0.1cm]
\pm \alpha_f c_{f} \\[0.1cm]
\mp \alpha_s c_s \beta_2 sgn(b_1)\\[0.1cm]
\mp \alpha_s c_s \beta_3 sgn(b_1) \\[0.1cm]
\alpha_f \rho a^2 \\[0.1cm]
0 \\[0.1cm]
\alpha_s a \beta_2 \sqrt{\rho} \\[0.1cm]
\alpha_s a \beta_3 \sqrt{\rho} 
\end{bmatrix},
\qquad 
\vec{r}_{\pm s} = \begin{bmatrix}
\alpha_s\rho \\[0.1cm]
\pm \alpha_s c_s \\[0.1cm]
\pm \alpha_f c_f \beta_2 sgn(b_1)\\[0.1cm]
\pm \alpha_f c_f \beta_3 sgn(b_1) \\[0.1cm]
\alpha_s \rho a^2 \\[0.1cm]
0 \\[0.1cm]
-\alpha_f a \beta_2 \sqrt{\rho} \\[0.1cm]
-\alpha_f a \beta_3 \sqrt{\rho} 
\end{bmatrix},
\end{equation}
\end{itemize}
where the normalized, tangential magnetic field components are given by
\begin{equation}
\qquad
\beta_{2,3} = \frac{b_{2,3}}{b_{\perp}}.
\end{equation}
These quantities $\beta_{2,3}$ are indeterminate if $b_{\perp}\approx 0$. It was suggested in \cite{brio1988} that, in this degenerate case, the value of $\beta$ may be given arbitrarily, e.g., $1/\sqrt{2}$. Such a definition preserves the orthogonality of the eigenvectors. We denote the matrix of rescaled right eigenvectors $\matrix{R}$ of $\widehat{\matrix{B}}$ with columns given by the vectors $\eqref{entropyR} -\eqref{MHDR}$ in the following order
\begin{equation}\label{rightEV1}
{\matrix{R}} = \Big[ {\,\vec{r}}_{-f} \big|\,{\vec{r}}_{-a} \big|\,{\vec{r}}_{-s} \big|\,{\vec{r}}_{E} \big|\,{\vec{r}}_{D} \big|\,{\vec{r}}_{+s} \big|\,{\vec{r}}_{+a} \big|\,{\vec{r}}_{+f} \Big],
\end{equation}
and ${\matrix{L}} = {\matrix{R}}^{-1}$.

Recall that the right eigenvectors for the matrix $\widehat{\matrix{A}}$ are given by $\widehat{\matrix{R}} = \matrix{M}\matrix{R}$. That is, we simply multiply the matrix $\matrix{M}$ \eqref{MMatrix} and the matrix of rescaled right eigenvectors \eqref{rightEV1}. We summarize the simplified right eigenvectors for $\widehat{\matrix{R}} = \matrix{M}\matrix{R}$:
\begin{itemize}
\item[] \underline{Entropy and Divergence Waves}: $\lambda_{E,D} = u$
\begin{equation}\label{entropyAS}
\widehat{\vec{r}}_E = \begin{bmatrix} 1 \\ u \\v \\w \\ \frac{\|\vec{u}\|^2}{2} \\[0.05cm]0 \\0 \\0 \end{bmatrix},\quad\widehat{\vec{r}}_D = \begin{bmatrix} 0 \\ 0 \\0 \\0 \\\sqrt{\rho}\,b_1 \\[0.05cm]1 \\0 \\0 \end{bmatrix},
\end{equation}
\item[] \underline{Alfv\'{e}n Waves}: $\lambda_{\pm a} = u\pm b_1$ 
\begin{equation}\label{AlfvenAS}
\widehat{\vec{r}}_{\pm a} = \begin{bmatrix}
0 \\
0 \\
\mp \rho^{\frac{3}{2}}\,b_3 \\
\pm \rho^{\frac{3}{2}}\,b_2 \\
\pm \rho^{\frac{3}{2}}(b_2 w - b_3 v) \\
0 \\
-\rho b_3 \\
\rho b_2
\end{bmatrix},
\end{equation}
\item[] \underline{Magnetoacoustic Waves}: $\lambda_{\pm f,\pm s} = u\pm c_{f,s}$ 
\begin{equation}\label{MHDAS}
\widehat{\vec{r}}_{\pm f} = \begin{bmatrix}
\alpha_f\rho \\[0.1cm]
\alpha_f\rho(u \pm c_{f}) \\[0.1cm]
\rho\left(\alpha_f v \mp \alpha_s c_s \beta_2 sgn(b_1) \right) \\[0.1cm]
\rho\left(\alpha_f w \mp \alpha_s c_s \beta_3 sgn(b_1) \right) \\[0.1cm]
\Psi_{\pm f} \\[0.1cm]
0 \\[0.1cm]
\alpha_s a \beta_2 \sqrt{\rho} \\[0.1cm]
\alpha_s a \beta_3 \sqrt{\rho} 
\end{bmatrix},
\qquad 
\widehat{\vec{r}}_{\pm s} = \begin{bmatrix}
\alpha_s\rho \\[0.1cm]
\alpha_s\rho\left(u \pm c_s\right) \\[0.1cm]
\rho\left(\alpha_s v \pm \alpha_f c_f \beta_2 sgn(b_1)\right) \\[0.1cm]
\rho\left(\alpha_s w \pm \alpha_f c_f \beta_3 sgn(b_1)\right) \\[0.1cm]
\Psi_{\pm s} \\[0.1cm]
0 \\[0.1cm]
-\alpha_f a \beta_2 \sqrt{\rho} \\[0.1cm]
-\alpha_f a \beta_3 \sqrt{\rho} 
\end{bmatrix},
\end{equation}
where we introduce the auxiliary variables
\begin{equation}
\begin{aligned}
\Psi_{\pm f} &= \frac{\alpha_f\rho}{2}\|\vec{u}\|^2 + a\alpha_s\rho b_{\perp} + \frac{\alpha_f\rho a^2}{\gamma - 1} \pm \alpha_f c_f \rho u \mp \alpha_s c_s \rho\,sgn(b_1)\left(v\beta_2 + w\beta_3\right), \\ 
\Psi_{\pm s} &= \frac{\alpha_s\rho}{2}\|\vec{u}\|^2 - a\alpha_f\rho b_{\perp} + \frac{\alpha_s\rho a^2}{\gamma-1} \pm \alpha_s c_s \rho u \pm \alpha_f c_f \rho\,sgn(b_1)\left(v\beta_2 + w\beta_3\right).
\end{aligned}
\end{equation}
\end{itemize}
For convenience in later derivations, we denote the matrix of right eigenvectors $\widehat{\matrix{R}}$ with columns given by the vectors $\eqref{entropyAS} -\eqref{MHDAS}$ in the following order
\begin{equation}\label{rightEV}
\widehat{\matrix{R}} = \Big[ \hat{\,\vec{r}}_{-f} \big|\,\hat{\vec{r}}_{-a} \big|\,\hat{\vec{r}}_{-s} \big|\,\hat{\vec{r}}_{E} \big|\,\hat{\vec{r}}_{D} \big|\,\hat{\vec{r}}_{+s} \big|\,\hat{\vec{r}}_{+a} \big|\,\hat{\vec{r}}_{+f} \Big],
\end{equation}
and $\widehat{\matrix{L}} = \widehat{\matrix{R}}^{-1}$.

\subsubsection{Entropy Scaled Right Eigenvectors}\label{Sec:EntropySclaed}

Now we have a symmetrizable matrix $\widehat{\matrix{A}}$ with a complete set of eigenvalues and right eigenvectors. We next utilize a previous result from Barth \cite{barth99} which provides a systematic approach to restructure a general eigenvalue problem to a symmetric eigenvalue problem. To do so we rescale the right eigenvectors of an eigendecomposition with respect to a right symmetrizer matrix in the following way:
\begin{lem}(Eigenvector Scaling)
Let $\matrix{A}\in\mathbb{R}^{n\times n}$ be an arbitrary diagonalizable matrix and $S$ the set of all right symmetrizers:
\begin{equation}
S=\left\{\matrix{B}\in\mathbb{R}^{n\times n}\,\big|\;\matrix{B}\;is\; s.p.d,\;\;\matrix{AB} = (\matrix{AB})^T\right\}.
\end{equation}
Further, let $\matrix{R}\in\mathbb{R}^{n\times n}$ denote the right eigenvector matrix which diagonalizes $\matrix{A}$, i.e., $\matrix{A}=\matrix{R}\boldsymbol\Lambda\matrix{R}^{-1}$, with $r$ distinct eigenvalues. Then for each $\matrix{B}\in S$ there exists a symmetric block diagonal matrix $\matrix{T}$ that block scales columns of $\matrix{R}$, {\color{black}{$\widetilde{\matrix{R}} = \matrix{RT}$}}, such that
\begin{equation}
\matrix{B}=\widetilde{\matrix{R}}\widetilde{\matrix{R}}^T,\; \matrix{A}=\widetilde{\matrix{R}}\boldsymbol\Lambda\widetilde{\matrix{R}}^{-1},
\end{equation}
which implies
\begin{equation}
\matrix{AB}=\widetilde{\matrix{R}}\boldsymbol\Lambda\widetilde{\matrix{R}}^{T}.
\end{equation}
\end{lem}

\begin{proof} The proof of the eigenvector scaling lemma is given in \cite{barth99}.\end{proof}

\begin{thm} (Entropy Stable - Roe Type Stabilization (ES\textrm{-}Roe)) If we apply the diagonal scaling matrix
\begin{equation}\label{scalingMatrix}
\matrix{T} = diag\left(\frac{1}{\sqrt{2\rho\gamma}},\,\sqrt{\frac{p}{2\rho^3b_{\perp}^2}},\,\frac{1}{\sqrt{2\rho\gamma}},\,\sqrt{\frac{\rho(\gamma-1)}{\gamma}},\,\sqrt{\frac{p}{\rho}},\,\frac{1}{\sqrt{2\rho\gamma}},\,\sqrt{\frac{p}{2\rho^3b_{\perp}^2}},\,\frac{1}{\sqrt{2\rho\gamma}}\right),
\end{equation}
to the matrix of right eigenvectors $\widehat{\matrix{R}}$ \eqref{rightEV}, then we obtain the Merriam identity \cite{merriam1989} (Eq. 7.3.1 pg. 77) 
\begin{equation}\label{MerriamIdentity}
\matrix{H} = \widetilde{\matrix{R}}\widetilde{\matrix{R}}^T = \left(\widehat{\matrix{R}}\matrix{T}\right) \left(\widehat{\matrix{R}}\matrix{T}\right)^T = \widehat{\matrix{R}}\matrix{S}\widehat{\matrix{R}}^T,
\end{equation}
that relates the right eigenvectors of $\widehat{\matrix{A}}$ to the entropy Jacobian matrix \eqref{entropyJacobian}. For convenience, we introduce the diagonal scaling matrix $\matrix{S}=\matrix{T}\,^2$ in \eqref{MerriamIdentity}. We then have the guaranteed entropy stable flux interface contribution
\begin{equation}\label{minimalDiss}
\vec{f}^{*,ES\textrm{-}Roe} = \vec{f}^{*,ec} - \frac{1}{2} \widehat{\matrix{R}}|\widehat{\boldsymbol\Lambda}|\matrix{S}\widehat{\matrix{R}}^T\jump{\vec{v}}.
\end{equation}
\end{thm}

\begin{proof}
We define the dissipation term in the numerical flux \eqref{dissipation} to be
\begin{equation}\label{dissTerm1}
-\frac{1}{2}\matrix{D}\jump{\vec{q}} = -\frac{1}{2} |\widehat{\matrix{A}}|\jump{\vec{q}} = -\frac{1}{2} \widehat{\matrix{R}}|\widehat{\boldsymbol\Lambda}|\widehat{\matrix{L}}\jump{\vec{q}},
\end{equation}
where the eigendecomposition of $\widehat{\matrix{A}} = \widehat{\matrix{R}}\widehat{\boldsymbol\Lambda}\widehat{\matrix{L}}$ is given by \eqref{eigenvalues} and \eqref{rightEV}. We define entropy stability to mean the approximation guarantees that the entropy within the system is a decreasing function, satisfying the following inequality
\begin{equation}
\pderivative{U}{t} + \pderivative{F}{x} - \vec{v}^T\vec{s} \leq 0.
\end{equation}
From the previously computed discrete entropy update \eqref{TotalUpdate} {\color{black}{and the condition \eqref{entropyConservationCondition1} we find}} the total entropy within an element (now including the dissipative term \eqref{dissTerm1}) {\color{black}{to be}}
\begin{equation}\label{TotalUpdate2}
\begin{aligned}
\pderivative{}{t}(\Delta x_L U_L + \Delta x_R U_R) &=  \jump{\vec{v}}^T\vec{f}^* -\jump{\vec{v}\cdot\vec{f}}  + \average{\Delta x\vec{v}}{\color{black}{^T}}{\vec{s}}_{i+\tfrac{1}{2}}, \\
&=  \jump{\vec{v}}^T\vec{f}^{*,ec} - \frac{1}{2}\jump{\vec{v}}^T \widehat{\matrix{R}}|\widehat{\boldsymbol\Lambda}|\widehat{\matrix{L}}\jump{\vec{q}}-\jump{\vec{v}\cdot\vec{f}} + \average{\Delta x\vec{v}}{\color{black}{^T}}{\vec{s}}_{i+\tfrac{1}{2}}, \\
{\color{black}{\pderivative{}{t}(\Delta x_L U_L + \Delta x_R U_R)}} &={\color{black}{ - \jump{F}- \frac{1}{2}\jump{\vec{v}}^T \widehat{\matrix{R}}|\widehat{\boldsymbol\Lambda}|\widehat{\matrix{L}}\jump{\vec{q}},}} \\
\pderivative{}{t}(\Delta x_L U_L + \Delta x_R U_R) + \jump{F} &= - \frac{1}{2}\jump{\vec{v}}^T \widehat{\matrix{R}}|\widehat{\boldsymbol\Lambda}|\widehat{\matrix{L}}\jump{\vec{q}},
\end{aligned}
\end{equation}
from the design of the entropy conserving flux $\vec{f}^{*,ec}$. To ensure entropy stability, we must guarantee that the RHS term in \eqref{TotalUpdate2} is non-positive. Unfortunately, it was shown by Barth \cite{barth99} that the term 
\begin{equation}\label{RHSEntropy}
- \frac{1}{2}\jump{\vec{v}}^T \widehat{\matrix{R}}|\widehat{\boldsymbol\Lambda}|\widehat{\matrix{L}}\jump{\vec{q}},
\end{equation}
may become positive in the presence of very strong shocks. However, we know from entropy symmetrization theory, e.g \cite{barth99,merriam1989}, that the entropy Jacobian $\matrix{H}$, given by \eqref{entropyJacobian}, is a right symmetrizer for the flux Jacobian that incorporates the Powell source term $\widehat{\matrix{A}}$. Therefore, with the proper scaling matrix $\matrix{T}$ we acquire the Merriam identity 
\begin{equation}\label{MerriamIdentityinproof}
\matrix{H} = \widetilde{\matrix{R}}\widetilde{\matrix{R}}^T = \left(\widehat{\matrix{R}}\matrix{T}\right) \left(\widehat{\matrix{R}}\matrix{T}\right)^T = \widehat{\matrix{R}}\matrix{S}\widehat{\matrix{R}}^T.
\end{equation}
The rescaling of the right eigenvectors of $\widehat{\matrix{A}}$ to satisfy the Merriam identity \eqref{MerriamIdentityinproof} is sufficient to guarantee the negativity of \eqref{RHSEntropy}. We see from \eqref{RHSEntropy} and \eqref{MerriamIdentityinproof}
\begin{equation}\label{signSatified2}
\begin{aligned}
-\frac{1}{2}\jump{\vec{v}}^T \widehat{\matrix{R}}|\widehat{\boldsymbol\Lambda}|\widehat{\matrix{L}}\jump{\vec{q}} &\simeq - \frac{1}{2}\jump{\vec{v}}^T \widehat{\matrix{R}}|\widehat{\boldsymbol\Lambda}|\widehat{\matrix{L}}\vec{q}_{\vec{v}}\jump{\vec{v}},  \\
&= - \frac{1}{2}\jump{\vec{v}}^T \widehat{\matrix{R}}|\widehat{\boldsymbol\Lambda}|\widehat{\matrix{L}}\matrix{H}\jump{\vec{v}}, \\
&=- \frac{1}{2}\jump{\vec{v}}^T \widehat{\matrix{R}}|\widehat{\boldsymbol\Lambda}|\widehat{\matrix{L}}\left(\widehat{\matrix{R}}\matrix{S}\widehat{\matrix{R}}^T\right)\jump{\vec{v}}, \\
&=- \frac{1}{2}\jump{\vec{v}}^T \widehat{\matrix{R}}|\widehat{\boldsymbol\Lambda}|\matrix{S}\widehat{\matrix{R}}^T\jump{\vec{v}},
\end{aligned}
\end{equation}
where we used that $\widehat{\matrix{L}}=\widehat{\matrix{R}}^{-1}$. So with the appropriate diagonal scaling matrix $\matrix{S}$ we have shown that \eqref{signSatified2} is guaranteed negative because the product is a quadratic form scaled by a negative. We use the right eigenvectors from \eqref{rightEV}, the constraint \eqref{MerriamIdentityinproof}, and after a considerable amount of algebraic manipulation we determine the diagonal scaling matrix
\begin{equation}\label{scalingMatrixinProof}
\matrix{T} = diag\left(\frac{1}{\sqrt{2\rho\gamma}},\,\sqrt{\frac{p}{2\rho^3b_{\perp}^2}},\,\frac{1}{\sqrt{2\rho\gamma}},\,\sqrt{\frac{\rho(\gamma-1)}{\gamma}},\,\sqrt{\frac{p}{\rho}},\,\frac{1}{\sqrt{2\rho\gamma}},\,\sqrt{\frac{p}{2\rho^3b_{\perp}^2}},\,\frac{1}{\sqrt{2\rho\gamma}}\right).
\end{equation}
\end{proof}

\begin{rem} \textit{(Entropy Stable - Local Lax-Friedrichs Type Stabilization (ES-LLF))}
There are other possible, negativity guaranteeing (but more dissipative) choices for the dissipation term \eqref{dissipation}. For example, if we make the simple choice of dissipation matrix to be
\begin{equation}\label{LFDMat}
\matrix{D} = |\lambda_{max}|\matrix{I},
\end{equation}
where $\lambda_{max}$ is the largest eigenvalue of the system from \eqref{eigenvalues} and $\matrix{I}$ is the identity matrix, then we can rewrite the dissipation term
\begin{equation}\label{signSatified}
\begin{aligned}
-\frac{1}{2} |\lambda_{max}|\matrix{I}\jump{\vec{q}} &\simeq -\frac{1}{2} |\lambda_{max}|\matrix{H}\jump{\vec{v}}, \\
 &=- \frac{1}{2}|\lambda_{max}|\widehat{\matrix{R}}\matrix{S}\widehat{\matrix{R}}^T\jump{\vec{v}},
\end{aligned}
\end{equation}
and we obtain a local Lax-Friedrichs type interface stabilization
\begin{equation}\label{LFDiss}
\begin{aligned}
\vec{f}^{*,ES\textrm{-}LLF} &=  \vec{f}^{*,ec} - \frac{1}{2}|\lambda_{max}|\matrix{H}\jump{\vec{v}}, \\
&= \vec{f}^{*,ec} - \frac{1}{2}|\lambda_{max}|\widehat{\matrix{R}}\matrix{S}\widehat{\matrix{R}}^T\jump{\vec{v}},
\end{aligned}
\end{equation} 
where, again, we use the Merriam identity \eqref{MerriamIdentity} for the entropy Jacobian $\matrix{H}$.
\end{rem}
{\color{black}{\begin{rem}\textit{(Evaluation of Dissipative Terms)} For the entropy stable numerical results in Sec. \ref{ESRiemann} and \ref{ROTOR} we use the arithmetic mean of the left and right states at an interface to create the matrices $\widehat{\matrix{R}}$, $\matrix{S}$, $\widehat{\boldsymbol{\Lambda}}$, and $\matrix{H}$ in the diffusion terms for the ES-Roe scheme \eqref{minimalDiss} or the ES-LLF scheme \eqref{LFDiss}.
\end{rem}}}

\section{Numerical Results}\label{NumericalResults}

In this section, we numerically verify the theoretical findings for the entropy conserving and entropy stable approximations for the ideal MHD equations. To integrate the semi-discrete formulation in time we use a low storage five-stage, fourth-order accurate Runge-Kutta time integrator of Carpenter and Kennedy \cite{Carpenter&Kennedy:1994}. First, in Sec. \ref{EOC}, we consider a test problem with a known analytical solution to demonstrate the accuracy of the entropy conserving method as well as the two stabilized formulations. Next, Sec. \ref{EntropyCons} demonstrates the entropy conservation of the approximation for a three Riemann problem configurations on regular and irregular grids. In Sec. \ref{ECRiemann} we demonstrate the computed solution of the three Riemann problems for the entropy conserving method. These solutions will exhibit significant oscillations in shocked regions. Finally, in Sec. \ref{ESRiemann} we compare the two entropy stable approximations \eqref{minimalDiss} and \eqref{LFDiss} against a high-resolution approximation comparable to that presented in the literature \cite{brio1988,rossmanith2002,ryu1994,torrilhon2003}.

\subsection{Convergence}\label{EOC}

For the convergence test, we switch to the manufactured solution technique and generate a smooth and periodic solution 
\begin{equation}
\begin{bmatrix}
\rho(x,t)\\
u(x,t)\\
v(x,t)\\
w(x,t)\\
p(x,t)\\
B_1(x,t) \\
B_2(x,t)\\
B_3(x,t)
\end{bmatrix}
=%
\begin{bmatrix}
2+\sin{(2\pi\,(x-t))}\\
1\\
1\\
1\\
(2+\sin{(2\pi\,(x-t))})^2\\
1\\
2+\sin{(2\pi\,(x-t))}\\
2+\sin{(2\pi\,(x-t))}
\end{bmatrix},
\end{equation}
with an additional analytic source term on the right hand side
\begin{equation}
\begin{bmatrix}
s_1(x,t)\\
s_2(x,t)\\
s_3(x,t)\\
s_4(x,t)\\
s_5(x,t)\\
s_6(x,t)\\
s_7(x,t)\\
s_8(x,t)
\end{bmatrix}
=%
\begin{bmatrix}
0\\
4\rho\rho_x \\
-\rho_x\\
-\rho_x\\
4\rho\rho_x-2\rho_x\\
0\\
0\\
0
\end{bmatrix},
\end{equation}
on the domain $\Omega=[-1,1]$ with periodic boundary conditions and final time $T=2$. We select the time step for the RK method small enough such that the error in the approximation is dominated by the error in the spatial discretizations. For all computations, we select a regular grid chosen according to the number of grid cells listed in the tables below as well as an irregular stretched mesh with a fixed ratio of 
\begin{equation}\label{gridRatio}
\frac{\Delta x_{max}}{\Delta x_{min}}=10.
\end{equation}

First, we test the entropy conserving (EC) scheme on a regular grid. The $L_2$-errors for all conserved quantities are shown in Tbl.~\ref{tab:EOC_EC_regular}. We note that for the specific regular grid used in the numerical experiment, the entropy conserving scheme is second order accurate. The average accuracy is hovering around $1.85$, however looking closely at the finest grid results, it is clear that the experimental order of convergence is close to $2$. 
The higher order convergence for the EC finite volume scheme is an effect of approximating the solution on a regular grid. The scheme drops to first order accuracy when an irregular grid is chosen as shown in Tbl.~\ref{tab:EOC_EC_irregular}, where we stretch the grid with a constant factor of \eqref{gridRatio}.

\begin{table}[!ht]
\footnotesize
\begin{center}
\begin{tabular}{lcccccccc}
\toprule
$\#$ elements & $\rho$ & $\rho u$ & $\rho v$ & $\rho w$ & $\rho e$ & $B_1$ & $B_2$ & $B_3$ \\[0.075cm]\toprule
50 & $5.29\text{E \!-}03$ &  $5.53\text{E \!-}03$ & $5.53\text{E \!-}03$ & $5.53\text{E \!-}03$ & $3.07\text{E \!-}03$ & $0.00$ &  $1.64\text{E \!-}03$ & $1.64\text{E \!-}03$ \\\midrule
100& $2.12\text{E \!-}03$&  $2.36\text{E \!-}03$&  $ 2.36\text{E \!-}03$ &  $2.36\text{E \!-}03$ & $ 1.05\text{E \!-}03$ & $ 0.00$ &  $2.67\text{E \!-}04$ & $2.67\text{E \!-}04$ \\\midrule
200 & $9.65\text{E \!-}04$&   $9.39\text{E \!-}04$&  $ 9.39\text{E \!-}04$ & $ 9.39\text{E \!-}04$ & $ 2.56\text{E \!-}04$ & $ 0.00$ & $ 1.50\text{E \!-}04$ & $ 1.50\text{E \!-}04$\\\midrule
400 & $ 2.40\text{E \!-}04$ & $ 1.33\text{E \!-}04$ &  $1.33\text{E \!-}04$ & $1.33\text{E \!-}04$ & $ 6.39\text{E \!-}05$ & $ 0.00$ & $ 3.56\text{E \!-}05$ & $ 3.56\text{E \!-}05$ \\\midrule\midrule
avg EOC& $\mathbf{1.93}$ & $\mathbf{1.79}$ & $\mathbf{1.79}$ & $\mathbf{1.79}$ & $\mathbf{1.86}$ & --& $\mathbf{1.85}$ & $\mathbf{1.85}$ \\
\bottomrule
\end{tabular}
\end{center}
\caption{$L_2$ error of approximation to demonstrate the experimental order of convergence (EOC) for the \textit{entropy conserving (EC)} scheme on a regular grid.}
\label{tab:EOC_EC_regular}
\end{table}%

\begin{table}[!ht]
\footnotesize
\begin{center}
\begin{tabular}{lcccccccc}
\toprule
$\#$ elements & $\rho$ & $\rho u$ & $\rho v$ & $\rho w$ & $\rho e$ & $B_1$ & $B_2$ & $B_3$ \\[0.075cm]\toprule
50   & $1.32\text{E \!-}01$&  $1.74\text{E \!-}01$&  $ 1.24\text{E \!-}01$ &  $1.24\text{E \!-}01$ & $ 1.57\text{E \!-}00$ & $ 0.00$ &  $1.67\text{E \!-}01$ & $1.67\text{E \!-}01$ \\\midrule
100 & $6.20\text{E \!-}02$ &  $7.63\text{E \!-}02$ & $5.59\text{E \!-}02$ & $5.59\text{E \!-}02$ & $7.80\text{E \!-}01$ & $0.00$ &  $8.26\text{E \!-}02$ & $8.26\text{E \!-}02$ \\\midrule
200 & $3.12\text{E \!-}02$&   $3.56\text{E \!-}02$&  $ 2.64\text{E \!-}02$ & $ 2.64\text{E \!-}02$ & $ 3.85\text{E \!-}01$ & $ 0.00$ & $ 4.07\text{E \!-}02$ & $ 4.07\text{E \!-}02$\\\midrule
400 & $ 1.54\text{E \!-}02$ & $ 1.72\text{E \!-}02$ &  $1.29\text{E \!-}02$ & $1.29\text{E \!-}02$ & $ 1.91\text{E \!-}01$ & $ 0.00$ & $ 2.02\text{E \!-}02$ & $ 2.02\text{E \!-}02$ \\\midrule\midrule
avg EOC& $\mathbf{1.03}$ & $\mathbf{1.11}$ & $\mathbf{1.09}$ & $\mathbf{1.09}$ & $\mathbf{1.01}$ & --& $\mathbf{1.02}$ & $\mathbf{1.02}$ \\
\bottomrule
\end{tabular}
\end{center}
\caption{$L_2$ error of approximation to demonstrate the experimental order of convergence (EOC) for the \textit{entropy conserving (EC)} scheme on an irregular grid with a stretching factor of \eqref{gridRatio}.}
\label{tab:EOC_EC_irregular}
\end{table}%

Next, we demonstrate the convergence of the two entropy stable finite volume schemes. The convergence results for the ES-Roe method are shown for the regular grid test in Tbl.~\ref{tab:EOC_ES-Roe_regular} and the irregular grid test in Tbl.~\ref{tab:EOC_ES-Roe_irregular}, where we see the average experimental convergence rate for either the regular or irregular grids are both first order accurate. Similar results for the ES-LLF scheme are given for the regular grid in Tbl.~\ref{tab:EOC_ES-LLF_regular} and the irregular grid in Tbl.~\ref{tab:EOC_ES-LLF_irregular}. We see that the ES-LLF method is decidedly first order accurate as well. 

\begin{table}[!ht]
\footnotesize
\begin{center}
\begin{tabular}{lcccccccc}
\toprule
$\#$ elements & $\rho$ & $\rho u$ & $\rho v$ & $\rho w$ & $\rho e$ & $B_1$ & $B_2$ & $B_3$ \\[0.075cm]\toprule
50 & $2.49\text{E \!-}02$ &  $1.90\text{E \!-}01$ & $3.46\text{E \!-}02$ & $3.46\text{E \!-}02$ & $4.07\text{E \!-}01$ & $0.00$ &  $4.01\text{E \!-}02$ & $4.01\text{E \!-}02$ \\\midrule
100& $1.50\text{E \!-}02$&  $1.01\text{E \!-}01$&  $1.01\text{E \!-}02$ &  $1.01\text{E \!-}02$ & $ 2.04\text{E \!-}01$ & $ 0.00$ &  $1.99\text{E \!-}02$ & $1.99\text{E \!-}02$ \\\midrule
200 & $9.00\text{E \!-}03$&   $5.85\text{E \!-}02$&  $ 5.50\text{E \!-}03$ & $ 5.50\text{E \!-}03$ & $ 1.12\text{E \!-}01$ & $ 0.00$ & $ 1.18\text{E \!-}02$ & $ 1.18\text{E \!-}02$\\\midrule
400 & $ 3.26\text{E \!-}03$ & $ 3.05\text{E \!-}02$ &  $2.85\text{E \!-}03$ & $2.85\text{E \!-}03$ & $ 6.04\text{E \!-}02$ & $ 0.00$ & $ 5.93\text{E \!-}03$ & $ 5.93\text{E \!-}03$ \\\midrule\midrule
avg EOC& $\mathbf{0.98}$ & $\mathbf{0.88}$ & $\mathbf{0.91}$ & $\mathbf{0.91}$ & $\mathbf{0.92}$ & --& $\mathbf{0.92}$ & $\mathbf{0.92}$ \\
\bottomrule
\end{tabular}
\end{center}
\caption{$L_2$ error of approximation to demonstrate the experimental order of convergence (EOC) for the \textit{ES-Roe} scheme on a regular grid.}
\label{tab:EOC_ES-Roe_regular}
\end{table}%

\begin{table}[!ht]
\footnotesize
\begin{center}
\begin{tabular}{lcccccccc}
\toprule
$\#$ elements & $\rho$ & $\rho u$ & $\rho v$ & $\rho w$ & $\rho e$ & $B_1$ & $B_2$ & $B_3$ \\[0.075cm]\toprule
50 & $8.05\text{E \!-}02$ &  $3.38\text{E \!-}01$ & $8.17\text{E \!-}02$ & $8.17\text{E \!-}02$ & $5.52\text{E \!-}01$ & $0.00$ &  $4.50\text{E \!-}02$ & $4.50\text{E \!-}02$ \\\midrule
100& $4.25\text{E \!-}02$&  $1.75\text{E \!-}01$&  $5.13\text{E \!-}02$ &  $5.13\text{E \!-}02$ & $ 2.55\text{E \!-}01$ & $ 0.00$ &  $2.30\text{E \!-}02$ & $2.30\text{E \!-}02$ \\\midrule
200 & $2.19\text{E \!-}02$&   $9.05\text{E \!-}02$&  $2.74\text{E \!-}02$ & $ 2.74\text{E \!-}02$ & $ 1.35\text{E \!-}01$ & $ 0.00$ & $ 1.22\text{E \!-}02$ & $ 1.22\text{E \!-}02$\\\midrule
400 & $ 1.11\text{E \!-}02$ & $ 4.65\text{E \!-}02$ &  $1.41\text{E \!-}03$ & $1.41\text{E \!-}02$ & $ 6.54\text{E \!-}02$ & $ 0.00$ & $ 6.27\text{E \!-}03$ & $ 6.27\text{E \!-}03$ \\\midrule\midrule
avg EOC& $\mathbf{0.95}$ & $\mathbf{0.95}$ & $\mathbf{0.95}$ & $\mathbf{0.95}$ & $\mathbf{1.02}$ & --& $\mathbf{0.95}$ & $\mathbf{0.95}$ \\
\bottomrule
\end{tabular}
\end{center}
\caption{$L_2$ error of approximation to demonstrate the experimental order of convergence (EOC) for the \textit{ES-Roe} scheme on an irregular grid with stretching factor \eqref{gridRatio}.}
\label{tab:EOC_ES-Roe_irregular}
\end{table}%

\begin{table}[!ht]
\footnotesize
\begin{center}
\begin{tabular}{lcccccccc}
\toprule
$\#$ elements & $\rho$ & $\rho u$ & $\rho v$ & $\rho w$ & $\rho e$ & $B_1$ & $B_2$ & $B_3$ \\[0.075cm]\toprule
50 & $3.24\text{E \!-}02$ &  $2.16\text{E \!-}01$ & $3.94\text{E \!-}02$ & $3.94\text{E \!-}02$ & $2.77\text{E \!-}01$ & $0.00$ &  $3.06\text{E \!-}02$ & $3.06\text{E \!-}02$ \\\midrule
100& $2.03\text{E \!-}02$&  $1.14\text{E \!-}01$&  $2.02\text{E \!-}02$ &  $2.02\text{E \!-}02$ & $ 1.32\text{E \!-}01$ & $ 0.00$ &  $1.49\text{E \!-}02$ & $1.49\text{E \!-}02$ \\\midrule
200 & $1.19\text{E \!-}02$&   $6.82\text{E \!-}02$&  $ 1.13\text{E \!-}02$ & $ 1.13\text{E \!-}02$ & $ 6.70\text{E \!-}02$ & $ 0.00$ & $ 7.87\text{E \!-}03$ & $ 7.87\text{E \!-}03$\\\midrule
400 & $ 5.39\text{E \!-}03$ & $ 3.25\text{E \!-}02$ &  $5.77\text{E \!-}03$ & $5.77\text{E \!-}03$ & $ 3.50\text{E \!-}02$ & $ 0.00$ & $ 4.12\text{E \!-}03$ & $ 4.12\text{E \!-}03$ \\\midrule\midrule
avg EOC& $\mathbf{0.87}$ & $\mathbf{0.91}$ & $\mathbf{0.92}$ & $\mathbf{0.92}$ & $\mathbf{1.00}$ & --& $\mathbf{0.96}$ & $\mathbf{0.96}$ \\
\bottomrule
\end{tabular}
\end{center}
\caption{$L_2$ error of approximation to demonstrate the experimental order of convergence (EOC) for the \textit{ES-LLF} scheme on a regular grid.}
\label{tab:EOC_ES-LLF_regular}
\end{table}%

\begin{table}[!ht]
\footnotesize
\begin{center}
\begin{tabular}{lcccccccc}
\toprule
$\#$ elements & $\rho$ & $\rho u$ & $\rho v$ & $\rho w$ & $\rho e$ & $B_1$ & $B_2$ & $B_3$ \\[0.075cm]\toprule
50 & $5.23\text{E \!-}02$ &  $3.70\text{E \!-}01$ & $6.67\text{E \!-}02$ & $6.67\text{E \!-}02$ & $4.19\text{E \!-}01$ & $0.00$ &  $2.34\text{E \!-}02$ & $2.34\text{E \!-}02$ \\\midrule
100& $2.66\text{E \!-}02$&  $1.91\text{E \!-}01$&  $3.48\text{E \!-}02$ &  $3.48\text{E \!-}02$ & $ 2.11\text{E \!-}01$ & $ 0.00$ &  $1.16\text{E \!-}02$ & $1.16\text{E \!-}02$ \\\midrule
200 & $1.33\text{E \!-}02$&   $1.03\text{E \!-}01$&  $1.81\text{E \!-}02$ & $ 1.81\text{E \!-}02$ & $ 1.05\text{E \!-}02$ & $ 0.00$ & $ 6.04\text{E \!-}03$ & $ 6.04\text{E \!-}03$\\\midrule
400 & $ 6.91\text{E \!-}03$ & $ 6.56\text{E \!-}02$ &  $9.25\text{E \!-}03$ & $9.25\text{E \!-}03$ & $ 5.52\text{E \!-}03$ & $ 0.00$ & $ 3.12\text{E \!-}03$ & $ 3.12\text{E \!-}03$ \\\midrule\midrule
avg EOC& $\mathbf{0.97}$ & $\mathbf{0.96}$ & $\mathbf{0.95}$ & $\mathbf{0.95}$ & $\mathbf{0.97}$ & --& $\mathbf{0.97}$ & $\mathbf{0.97}$ \\
\bottomrule
\end{tabular}
\end{center}
\caption{$L_2$ error of approximation to demonstrate the experimental order of convergence (EOC) for the \textit{ES-LLF} scheme on an irregular grid with stretching factor \eqref{gridRatio}.}
\label{tab:EOC_ES-LLF_irregular}
\end{table}%

\subsection{Mass, Momentum, Energy, and Entropy Conservation}\label{EntropyCons}
By design we know the finite volume scheme for the ideal MHD equations with the Janhunen source term \eqref{1DIdealMHDwithSource} will conserve mass and momentum. From the derivations in Sec. \ref{EntropyFlux} we know that the scheme also exactly preserves the entropy on a general grid, if we use the newly designed flux $\vec{f}^{*,ec}$ \eqref{Sec:EntropyConservingFlux}. We measure the change in any of the conservative variables with
\begin{equation}
\label{eq:errortotalenergy}
\Delta e(T) \coloneqq |e_{int}(t=0) - e_{int}(T)|,
\end{equation}
where, for example, $e_{int}$ is the entropy function
\begin{equation}\label{thisAgain}
U =  -\frac{\rho s}{\gamma -1},
\end{equation} 
integrated over the whole domain.

To demonstrate the conservative properties we consider three one-dimensional Riemann problems from the literature. The test problems we consider are:
\begin{enumerate}
\item \textit{\underline{Brio and Wu Shock Tube}} \cite{brio1988}: 
\begin{equation}\label{StrongRiemann_Brio}
\begin{bmatrix}\rho, \rho u, \rho v, \rho w, p, B_1, B_2, B_3 \end{bmatrix}^T = \left\{
\begin{array}{lc}
\left[1,0,0,{\color{black}{0,1}},0.75,1,0\right]^T, & \textrm{if}\quad x \leq 0.5, \\ 
\left[0.125,0,0,0,0.1,0.75,-1,0\right]^T, & \textrm{if}\quad x > 0.5,
\end{array}
\right.
\end{equation}
on the domain $\Omega = [0,1]$, $\gamma = 2$, and final time $T=0.12$.
\item \textit{\underline{Ryu and Jones Riemann Problem}} \cite{ryu1994}: 
\begin{equation}\label{StrongRiemann_Ryu}
\begin{bmatrix}\rho, \rho u, \rho v, \rho w, p, B_1, B_2, B_3 \end{bmatrix}^T = \left\{
\begin{array}{lc}
\left[1,0,0,0,1,0.7,0,0\right]^T, & \textrm{if}\quad x \leq 0, \\ 
\left[0.3,0,0,1,0.2,0.7,1,0\right]^T, & \textrm{if}\quad x > 0,
\end{array}
\right.
\end{equation}
on the domain $\Omega = [-1,1]$, $\gamma = 5/3$, and final time $T=0.4$. 
\item \textit{\underline{Torrilhon Riemann Problem}} \cite{torrilhon2003}: 
\begin{equation}\label{StrongRiemann_Tor}
\begin{bmatrix}\rho, \rho u, \rho v, \rho w, p, B_1, B_2, B_3 \end{bmatrix}^T = \left\{
\begin{array}{lc}
\left[3,0,0,0,3,1.5,1,0\right]^T, & \textrm{if}\quad x \leq 0, \\ 
\left[1,0,0,0,1,1.5,\cos(1.5),\sin(1.5)\right]^T, & \textrm{if}\quad x > 0,
\end{array}
\right.
\end{equation}
on the domain $\Omega = [-1,1.5]$, $\gamma = 5/3$, and final time $T=0.4$. 
\end{enumerate}
 For each of the conservation test problems we set \textit{periodic boundary conditions}.

We compute the error in the change of each conservative variable for three values of the Courant-Friedrichs-Lewy (CFL) number: 1.0, 0.1, and 0.01. For each simulation the entropy conserving finite volume scheme used 100 regular grid cells to compute the results in Tbls. \ref{tab:conservation_Brio}, \ref{tab:conservation_Ryu}, and \ref{tab:conservation_Torr} and 100 irregular grid cells for the results in Tbls. \ref{tab:conservation_Brio_irregular}, \ref{tab:conservation_Ryu_irregular}, and \ref{tab:conservation_Torr_irregular}. 

For each value of the CFL number we obtain errors in the mass, momentum, energy, and magnetic field variables on the order of finite precision. This is not surprising as for one-dimensional MHD flow we know analytically that $B_1\equiv constant$ and thus the Janhunen source term is identically zero. We reiterate the source term was necessary for the proof of entropy conservation in Sec. \ref{Sec:EntropyConservingFlux}, but is expected to vanish in 1D. For multi-dimensional flows one would see that the addition of the Janhunen source term to enforce the divergence-free condition will cause the loss of conservation of the magnetic field quantities $\vec{B}$. 

{\color{black}{Due to the dissipative influence of the time integrator we know that the entropy should not be conserved. However, due to the entropy conservative flux the change in the total entropy will converge to zero as $\Delta t$ converges to zero. In each of the Tbls. \ref{tab:conservation_Brio} - \ref{tab:conservation_Torr_irregular} we demonstrate this property. }}
 The dissipation introduced by the temporal discretization is significantly reduced if we shrink the time step. We see that the {\color{black}{change}} of the entropy can be lowered to single or double machine precision by decreasing the CFL number, and hence the time step.

In each of the tables in this section it is also possible to see the temporal accuracy of the approximations. If we shrink the time step by a factor ten we see that the error in the entropy shrinks by a little over a factor of $10^4$, as we expect for a fourth order time integrator. 
\begin{table}[!ht]
\small
\begin{center}
\begin{tabular}{lccccccccc}
\toprule
CFL & $\rho$ & $\rho u$ & $\rho v$ & $\rho w$ & $\rho e$ & $B_1$ & $B_2$ & $B_3$ & $U$ \\[0.075cm]\toprule
$1.0$ & $2.00\text{E \!-}16$ & $2.08\text{E \!-}16$ & $4.51\text{E \!-}16$ &  $2.44\text{E \!-}15$ & $8.88\text{E \!-}16$ & $0.00$ & $3.13\text{E \!-}16$ & $2.22\text{E \!-}16$ & $5.64\text{E \!-}04$ \\\midrule
$0.1$ & $2.66\text{E \!-}16$ & $2.77\text{E \!-}16$ & $2.77\text{E \!-}16$ & $-2.22\text{E \!-}15$ & $1.11\text{E \!-}15$ & $0.00$ & $1.36\text{E \!-}16$ & $2.22\text{E \!-}16$ & $1.61\text{E \!-}08$ \\\midrule
$0.01$ & $2.66\text{E \!-}16$ & $2.22\text{E \!-}16$ & $2.22\text{E \!-}16$ & $-2.22\text{E \!-}15$ & $1.11\text{E \!-}15$ & $0.00$ & $2.01\text{E \!-}16$ & $8.88\text{E \!-}16$ & $1.41\text{E \!-}12$ \\\bottomrule
\end{tabular}
\end{center}
\caption{Conservation errors (integrated over the whole domain) of the entropy conserving approximation applied to the Brio and Wu shock tube problem \eqref{StrongRiemann_Brio} for different CFL numbers, final time $T=0.12$, and $100$ regular grid cells.}
\label{tab:conservation_Brio}
\end{table}%

\begin{table}[!ht]
\small
\begin{center}
\begin{tabular}{lccccccccc}
\toprule
CFL & $\rho$ & $\rho u$ & $\rho v$ & $\rho w$ & $\rho e$ & $B_1$ & $B_2$ & $B_3$ & $U$ \\[0.075cm]\toprule
$1.0$   & $2.22\text{E \!-}16$ & $2.22\text{E \!-}16$ & $2.22\text{E \!-}16$ &  $4.44\text{E \!-}16$ & $4.44\text{E \!-}16$ & $0.00$ & $2.22\text{E \!-}16$ & $4.34\text{E \!-}16$ & $2.86\text{E \!-}05$ \\\midrule
$0.1$   & $2.22\text{E \!-}16$ & $2.87\text{E \!-}16$ & $1.13\text{E \!-}16$ & $4.44\text{E \!-}16$ & $4.44\text{E \!-}16$ & $0.00$ & $2.22\text{E \!-}16$ & $2.22\text{E \!-}16$ & $1.97\text{E \!-}09$ \\\midrule
$0.01$ & $8.88\text{E \!-}16$ & $2.22\text{E \!-}16$ & $2.22\text{E \!-}16$ & $-2.22\text{E \!-}16$ & $4.44\text{E \!-}16$ & $0.00$ & $4.44\text{E \!-}16$ & $8.88\text{E \!-}16$ & $1.62\text{E \!-}13$ \\\bottomrule
\end{tabular}
\end{center}
\caption{Conservation errors (integrated over the whole domain) of the entropy conserving approximation applied to the Ryu and Jones Riemann problem \eqref{StrongRiemann_Ryu} for different CFL numbers, final time $T=0.4$, and $100$ regular grid cells.}
\label{tab:conservation_Ryu}
\end{table}%

\begin{table}[!ht]
\small
\begin{center}
\begin{tabular}{lccccccccc}
\toprule
CFL & $\rho$ & $\rho u$ & $\rho v$ & $\rho w$ & $\rho e$ & $B_1$ & $B_2$ & $B_3$ & $U$ \\[0.075cm]\toprule
$1.0$   & $1.55\text{E \!-}15$ & $2.22\text{E \!-}16$ & $2.22\text{E \!-}16$ &  $4.44\text{E \!-}16$ & $2.22\text{E \!-}16$ & $0.00$ & $2.22\text{E \!-}16$ & $2.22\text{E \!-}16$ & $1.02\text{E \!-}05$ \\\midrule
$0.1$   & $6.66\text{E \!-}16$ & $2.22\text{E \!-}16$ & $4.44\text{E \!-}16$ & $4.44\text{E \!-}16$ & $2.22\text{E \!-}16$ & $0.00$ & $6.66\text{E \!-}16$ & $1.11\text{E \!-}16$ & $1.08\text{E \!-}09$ \\\midrule
$0.01$ & $8.88\text{E \!-}16$ & $2.22\text{E \!-}16$ & $2.22\text{E \!-}16$ & $4.44\text{E \!-}16$ & $2.22\text{E \!-}16$ & $0.00$ & $-2.22\text{E \!-}16$ & $2.22\text{E \!-}16$ & $1.06\text{E \!-}13$ \\\bottomrule
\end{tabular}
\end{center}
\caption{Conservation errors (integrated over the whole domain) of the entropy conserving approximation applied to the Torrilhon Riemann problem \eqref{StrongRiemann_Tor} for different CFL numbers, final time $T=0.4$, and $100$ regular grid cells.}
\label{tab:conservation_Torr}
\end{table}%

\begin{table}[!ht]
\small
\begin{center}
\begin{tabular}{lccccccccc}
\toprule
CFL & $\rho$ & $\rho u$ & $\rho v$ & $\rho w$ & $\rho e$ & $B_1$ & $B_2$ & $B_3$ & $U$ \\[0.075cm]\toprule
$1.0$ & $5.55\text{E \!-}16$ & $4.44\text{E \!-}16$ & $2.22\text{E \!-}16$ &  $2.22\text{E \!-}16$ & $4.44\text{E \!-}16$ & $0.00$ & $1.84\text{E \!-}15$ & $8.88\text{E \!-}16$ & $5.54\text{E \!-}04$ \\\midrule
$0.1$ & $2.22\text{E \!-}16$ & $8.88\text{E \!-}16$ & $2.22\text{E \!-}16$ & $-2.22\text{E \!-}15$ & $2.22\text{E \!-}15$ & $0.00$ & $1.11\text{E \!-}15$ & $8.88\text{E \!-}16$ & $5.05\text{E \!-}08$ \\\midrule
$0.01$ & $4.44\text{E \!-}16$ & $2.22\text{E \!-}16$ & $2.22\text{E \!-}16$ & $-5.55\text{E \!-}15$ & $1.33\text{E \!-}15$ & $0.00$ & $8.82\text{E \!-}15$ & $4.44\text{E \!-}16$ & $2.73\text{E \!-}12$ \\\bottomrule
\end{tabular}
\end{center}
\caption{Conservation errors (integrated over the whole domain) of the entropy conserving approximation applied to the Brio and Wu shock tube problem \eqref{StrongRiemann_Brio} for different CFL numbers, final time $T=0.12$, and $100$ irregular grid cells with a stretching factor \eqref{gridRatio}.}
\label{tab:conservation_Brio_irregular}
\end{table}%

\begin{table}[!ht]
\small
\begin{center}
\begin{tabular}{lccccccccc}
\toprule
CFL & $\rho$ & $\rho u$ & $\rho v$ & $\rho w$ & $\rho e$ & $B_1$ & $B_2$ & $B_3$ & $U$ \\[0.075cm]\toprule
$1.0$   & $2.22\text{E \!-}16$ & $2.22\text{E \!-}16$ & $2.22\text{E \!-}16$ &  $4.44\text{E \!-}16$ & $2.22\text{E \!-}16$ & $0.00$ & $2.22\text{E \!-}16$ & $1.11\text{E \!-}16$ & $3.51\text{E \!-}04$ \\\midrule
$0.1$   & $2.22\text{E \!-}16$ & $2.22\text{E \!-}16$ & $1.11\text{E \!-}16$ & $1.11\text{E \!-}16$ & $2.22\text{E \!-}16$ & $0.00$ & $4.44\text{E \!-}16$ & $2.22\text{E \!-}16$ & $2.56\text{E \!-}08$ \\\midrule
$0.01$ & $2.88\text{E \!-}15$ & $2.22\text{E \!-}16$ & $2.22\text{E \!-}16$ & $1.33\text{E \!-}16$ & $8.88\text{E \!-}16$ & $0.00$ & $6.66\text{E \!-}16$ & $2.22\text{E \!-}16$ & $2.30\text{E \!-}12$ \\\bottomrule
\end{tabular}
\end{center}
\caption{Conservation errors (integrated over the whole domain) of the entropy conserving approximation applied to the Ryu and Jones Riemann problem \eqref{StrongRiemann_Ryu} for different CFL numbers, final time $T=0.4$, and $100$ irregular grid cells with a stretching factor \eqref{gridRatio}.}
\label{tab:conservation_Ryu_irregular}
\end{table}%

\begin{table}[!ht]
\small
\begin{center}
\begin{tabular}{lccccccccc}
\toprule
CFL & $\rho$ & $\rho u$ & $\rho v$ & $\rho w$ & $\rho e$ & $B_1$ & $B_2$ & $B_3$ & $U$ \\[0.075cm]\toprule
$1.0$   & $2.22-16$ & $2.22\text{E \!-}16$ & $2.22\text{E \!-}16$ &  $4.44\text{E \!-}16$ & $1.11\text{E \!-}16$ & $0.00$ & $2.22\text{E \!-}16$ & $2.22\text{E \!-}16$ & $6.28\text{E \!-}04$ \\\midrule
$0.1$   & $2.22\text{E \!-}16$ & $4.44\text{E \!-}16$ & $4.44\text{E \!-}16$ & $4.44\text{E \!-}16$ & $2.22\text{E \!-}16$ & $0.00$ & $-2.22\text{E \!-}16$ & $-5.55\text{E \!-}16$ & $5.00\text{E \!-}08$ \\\midrule
$0.01$ & $8.88\text{E \!-}16$ & $2.22\text{E \!-}16$ & $2.22\text{E \!-}16$ & $4.44\text{E \!-}16$ & $4.44\text{E \!-}16$ & $0.00$ & $-4.44\text{E \!-}16$ & $6.66\text{E \!-}16$ & $4.08\text{E \!-}12$ \\\bottomrule
\end{tabular}
\end{center}
\caption{Conservation errors (integrated over the whole domain) of the entropy conserving approximation applied to the Torrilhon Riemann problem \eqref{StrongRiemann_Tor} for different CFL numbers, final time $T=0.4$, and $100$ irregular grid cells with a stretching factor \eqref{gridRatio}.}
\label{tab:conservation_Torr_irregular}
\end{table}%
 
\subsection{Entropy Conserving Riemann Problem}\label{ECRiemann}

We now apply the entropy conserving (EC) scheme to the Riemann problems $\eqref{StrongRiemann_Brio} - \eqref{StrongRiemann_Tor}$, except we choose inflow/outflow type boundary conditions. The computation is performed on 200 regular grid cells with CFL number 0.1 and to the final time indicated by the test problem subscribed. {\color{black}{As we discussed in Sec. \ref{EntropyFlux}, entropy conserving methods suffer breakdown in the presence of shocks. The numerical approximation does not dissipate energy at a shock. Thus, we expect large post-shock oscillations for the EC scheme applied to a Riemann problem.}} To create a reference solution for each of the test problems we use the ES-Roe scheme on 5000 grid cells. The ES-LLF scheme gives the same reference solution with 10,000 grid cells.

\subsubsection{Brio and Wu Shock Tube}\label{BrioWuEC}

Figure \ref{fig:EC_Brio} shows results for the entropy conserving (EC) finite volume scheme for the Brio and Wu shock tube. The EC method captures the fast/slow rarefractions, slow compound wave in the density, $u$ and the pressure, as well as shocks. However, this comes at the expense of large post-shock oscillations. These oscillations are expected as energy must be dissipated across the shock but the EC scheme is basically dissipation free except for the influence of the time integrator.
\begin{figure}[!ht]
\begin{center}
{
\includegraphics[scale=0.55]{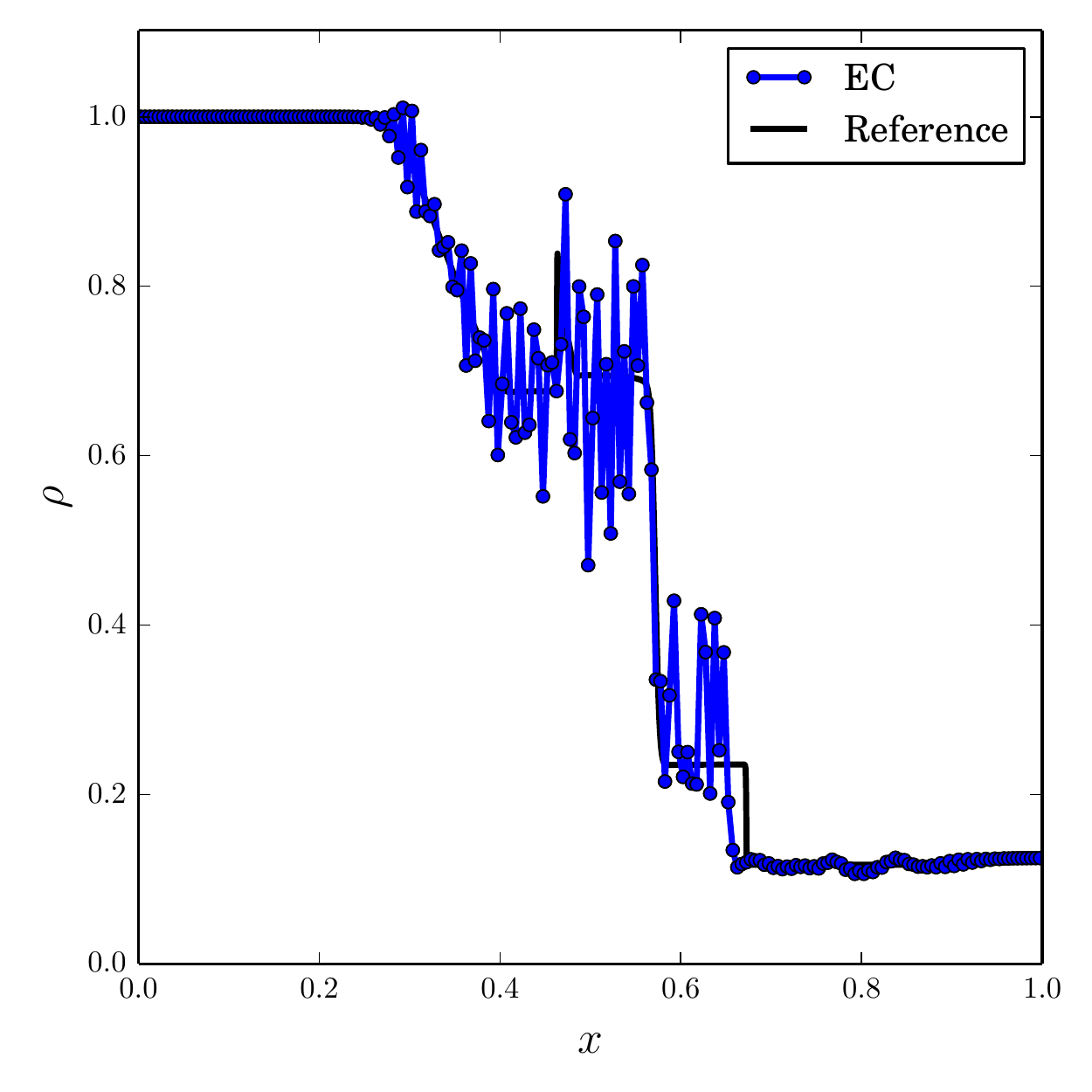}
}
{
\includegraphics[scale=0.55]{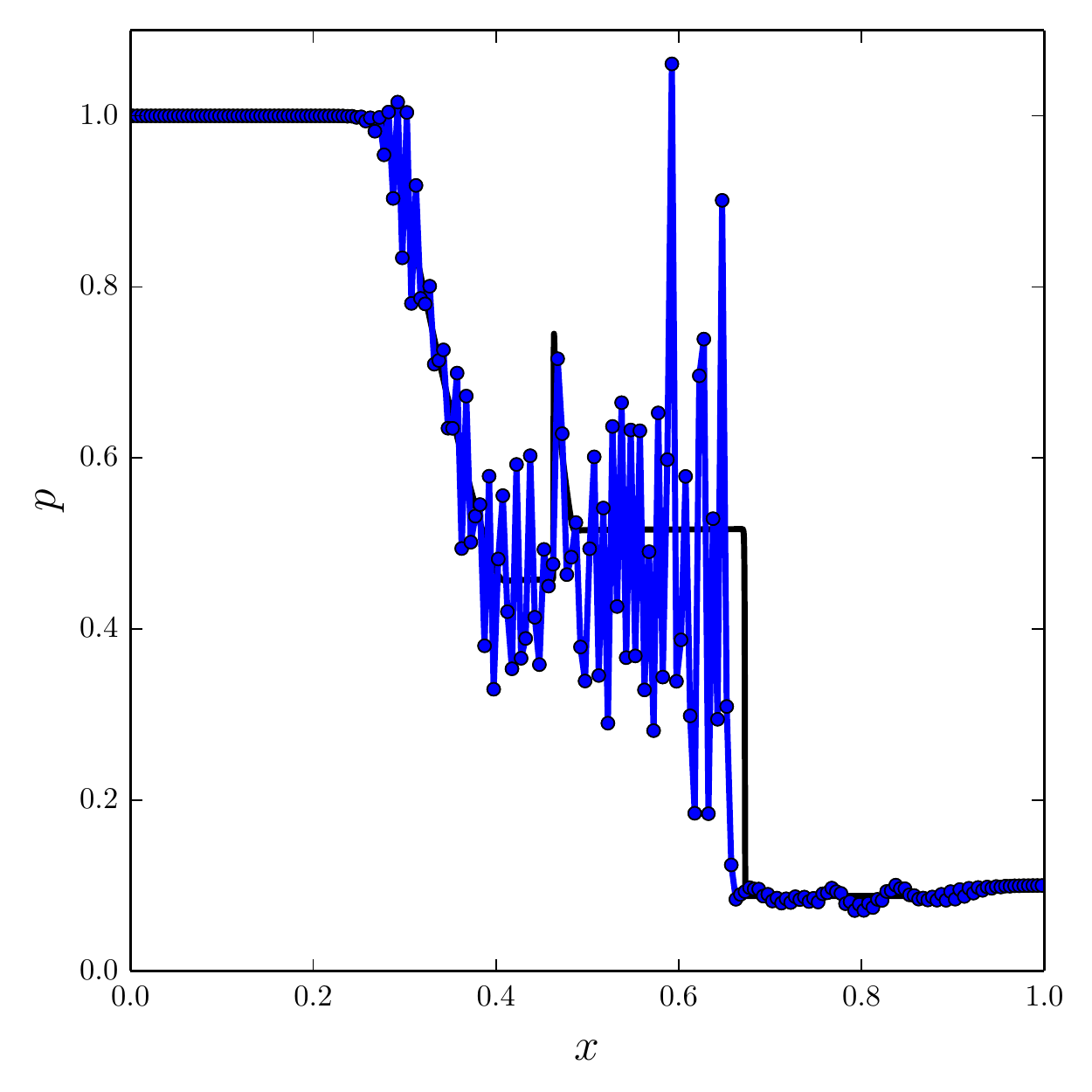}
}
\\
{
\includegraphics[scale=0.55]{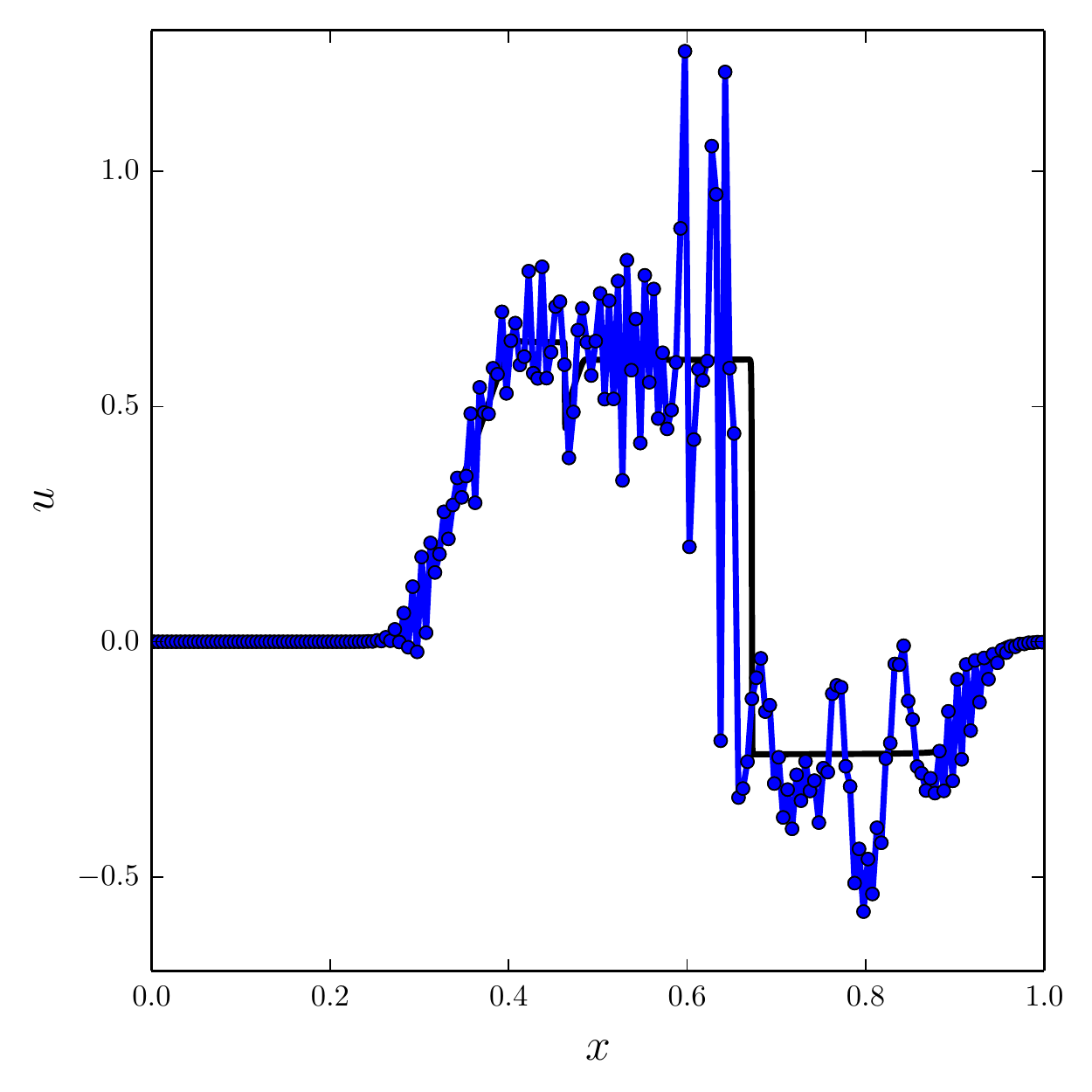}
}
{
\includegraphics[scale=0.55]{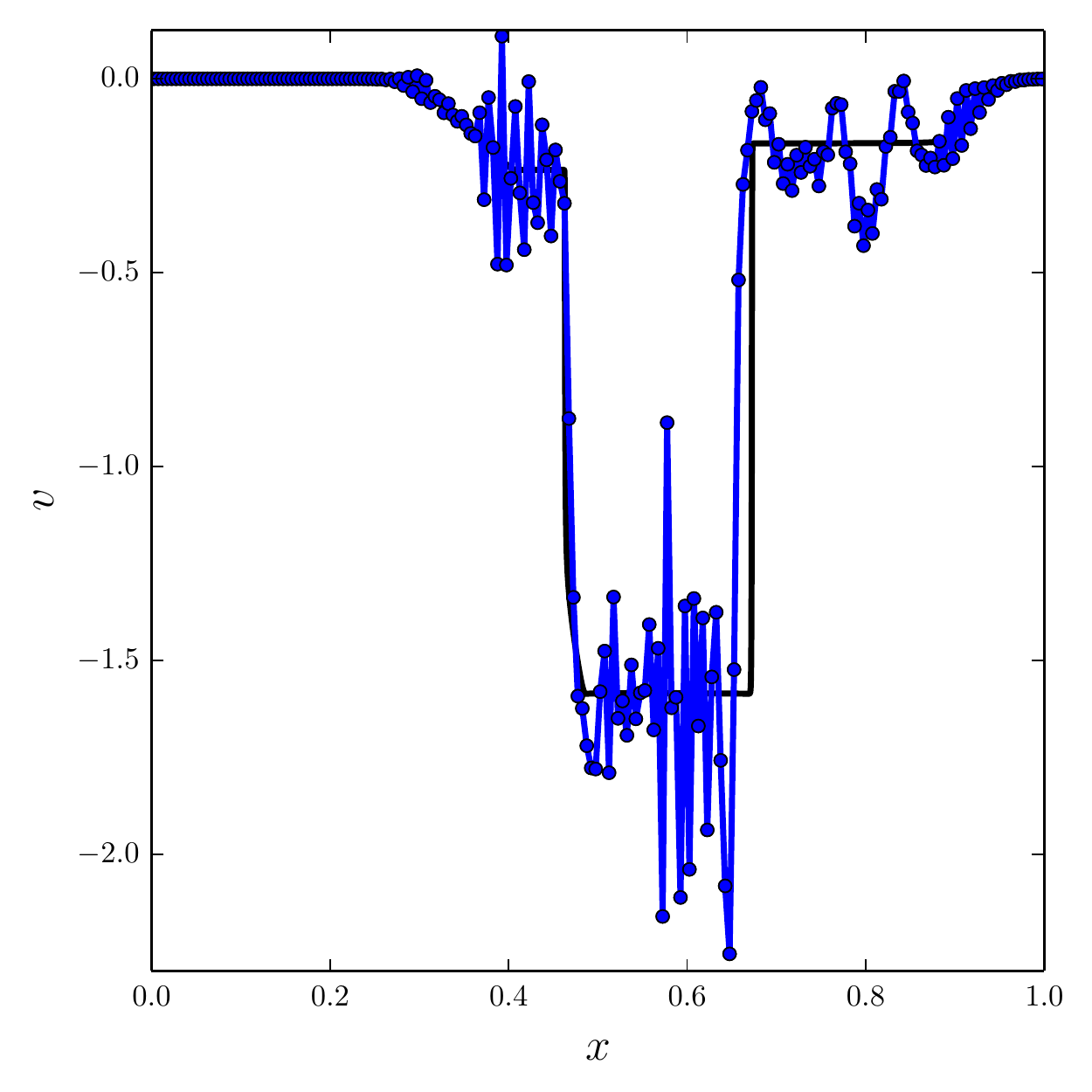}
}
\\
{
\includegraphics[scale=0.55]{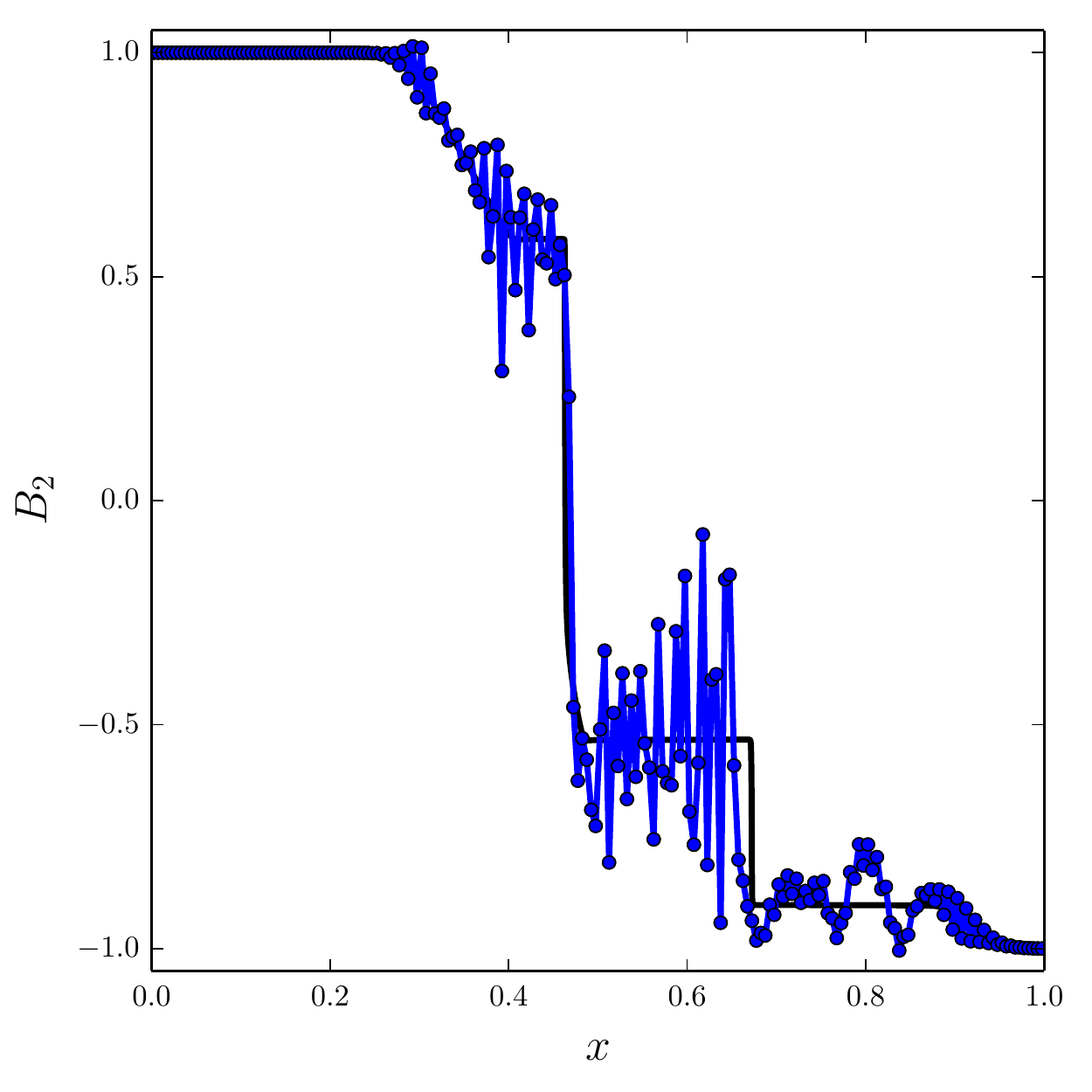}
}
\caption{In blue we present the entropy conserving approximations for the Brio and Wu shock tube of the density $\rho$, pressure $p$, $x$ and $y$ velocity components, and the $y$ of the magnetic field at $T = 0.12$. Solid black represents the reference solution.}
\label{fig:EC_Brio}
\end{center}
\end{figure}

\subsubsection{Ryu and Jones Riemann Problem}\label{RyuEC}

Figure \ref{fig:EC_Ryu} presents the computed EC solution for the Ryu and Jones Riemann problem. The EC method captures the fast/slow rarefractions and shocks well. But, again, there are large oscillations generated in the post-shock regions.
\begin{figure}[!ht]
\begin{center}
{
\includegraphics[scale=0.575]{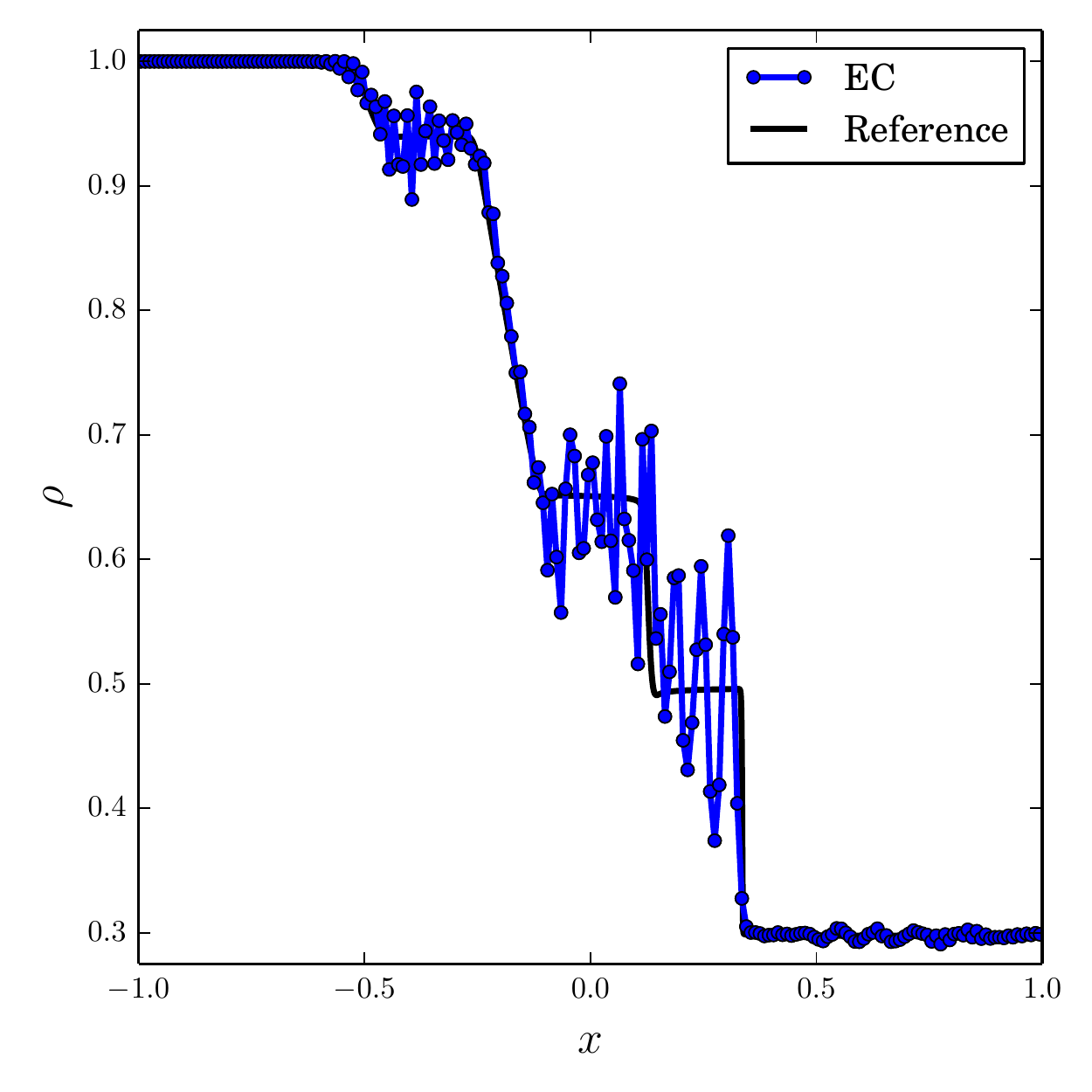}
}
{
\includegraphics[scale=0.575]{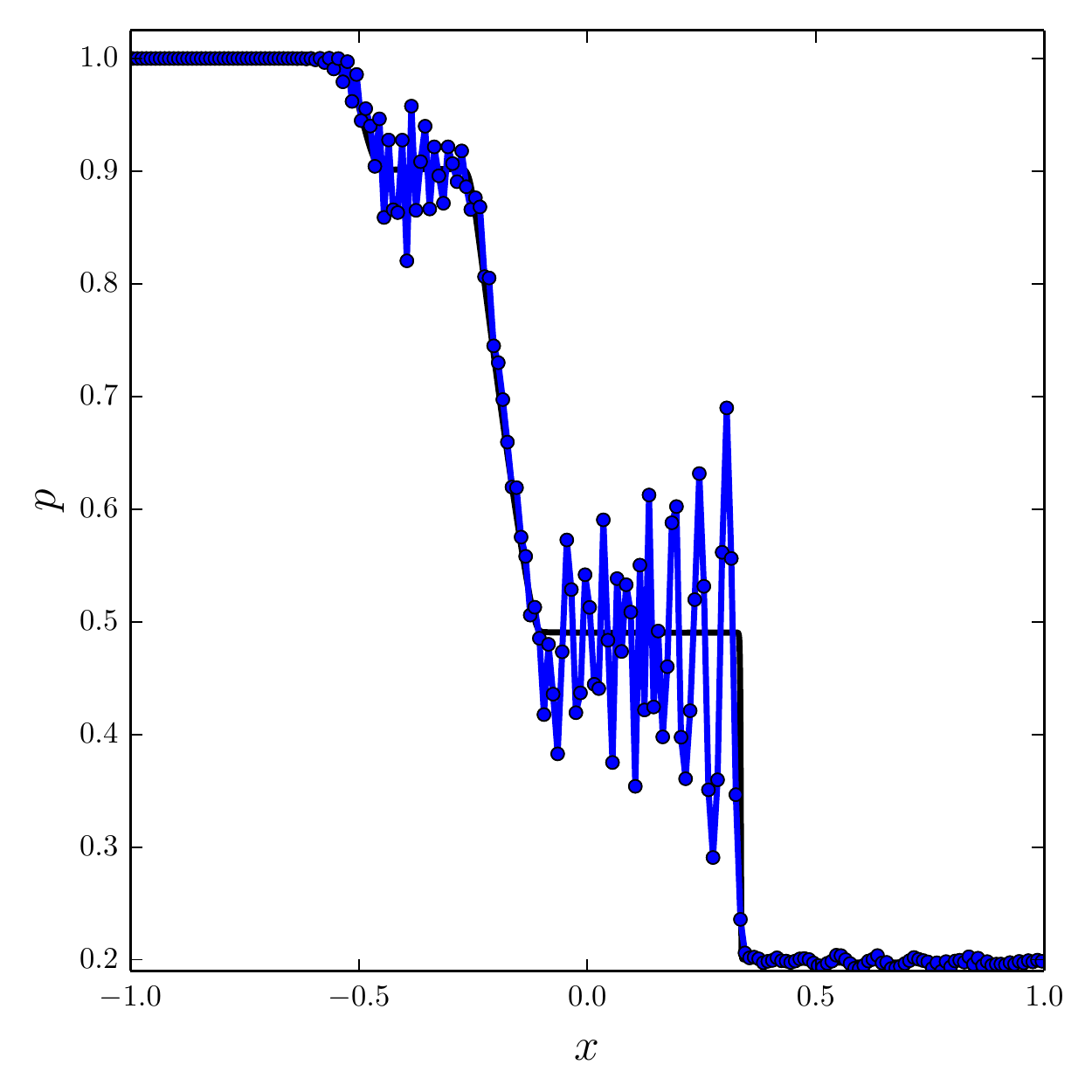}
}
\\
{
\includegraphics[scale=0.415]{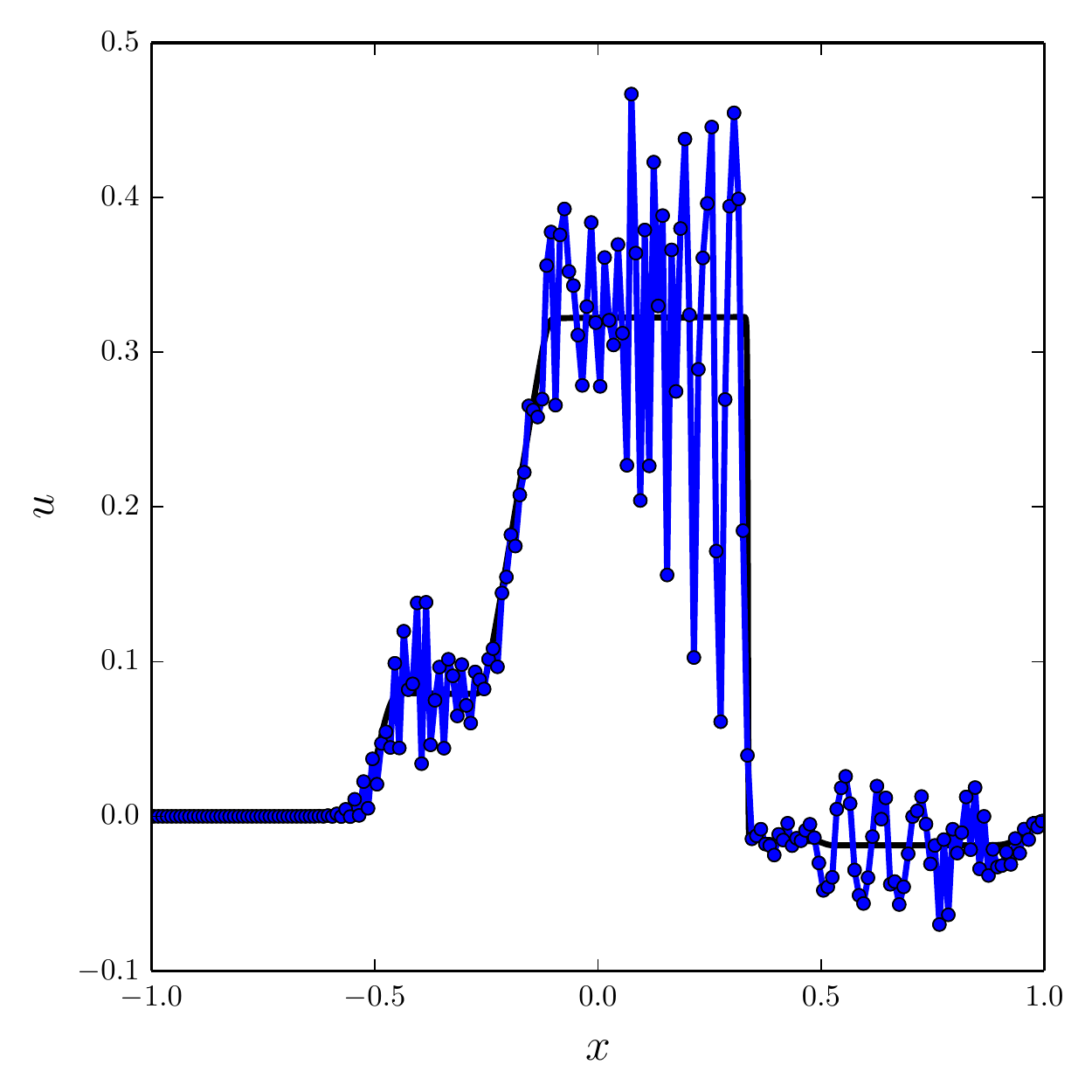}
}
{
\includegraphics[scale=0.415]{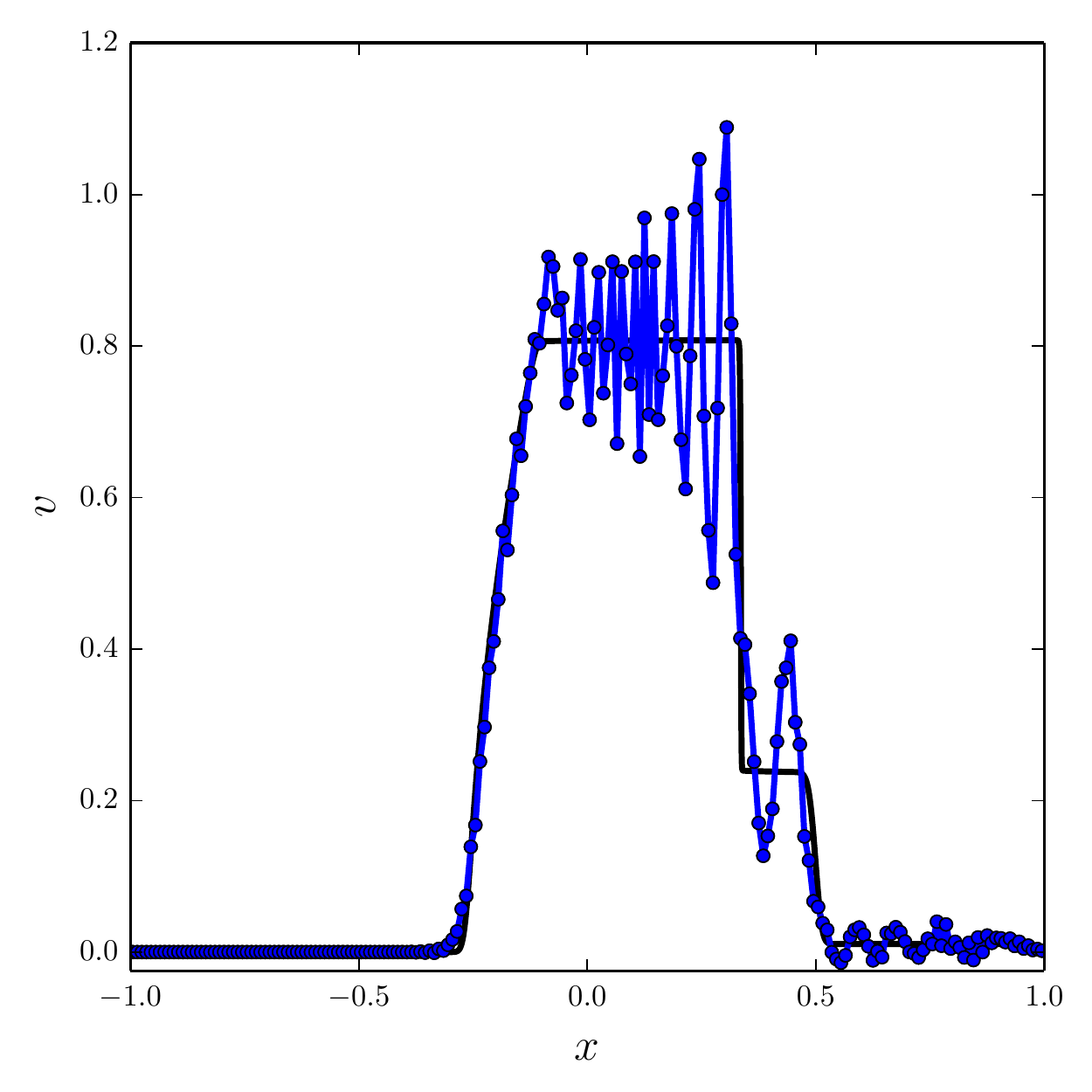}
}
{
\includegraphics[scale=0.415]{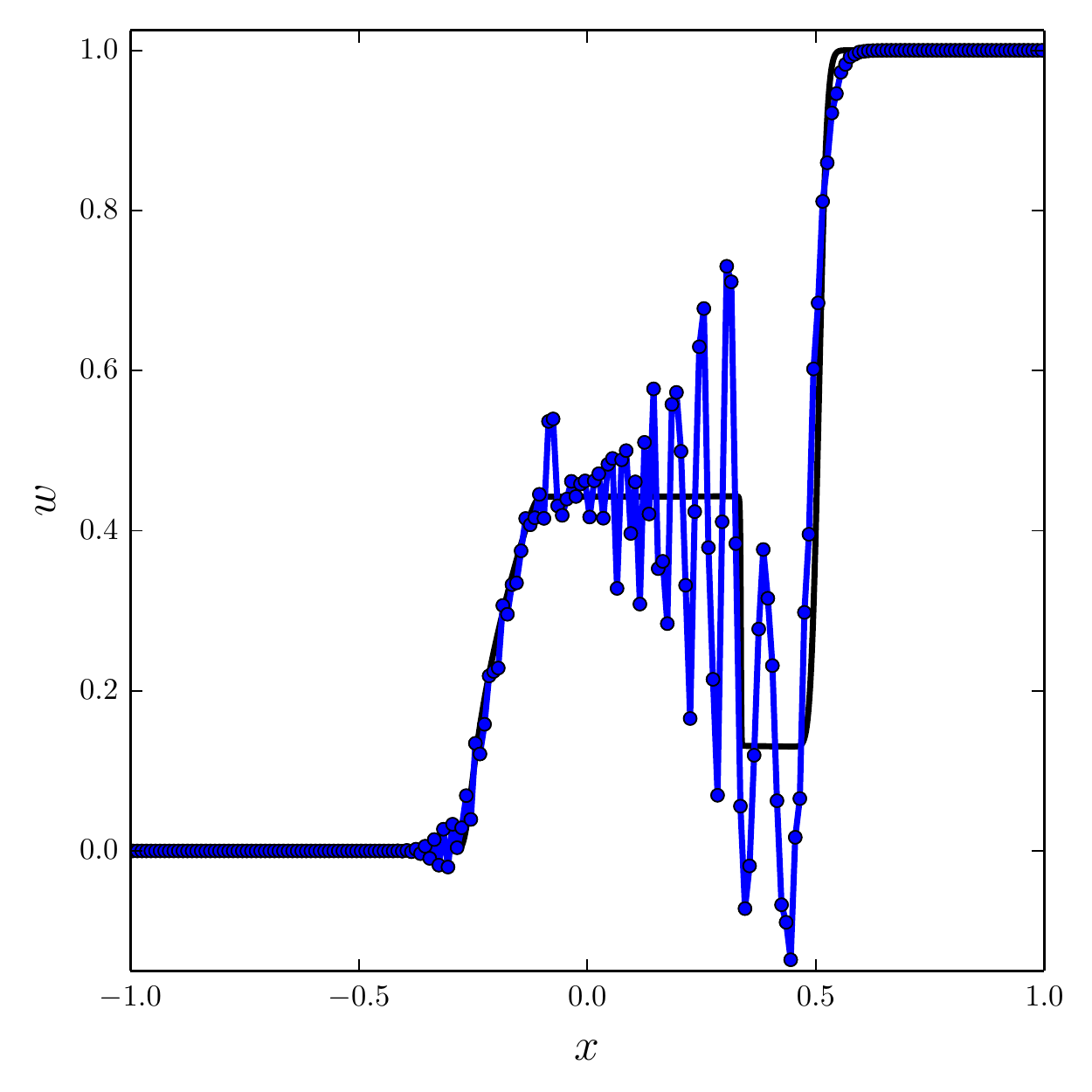}
}
\\
{
\includegraphics[scale=0.575]{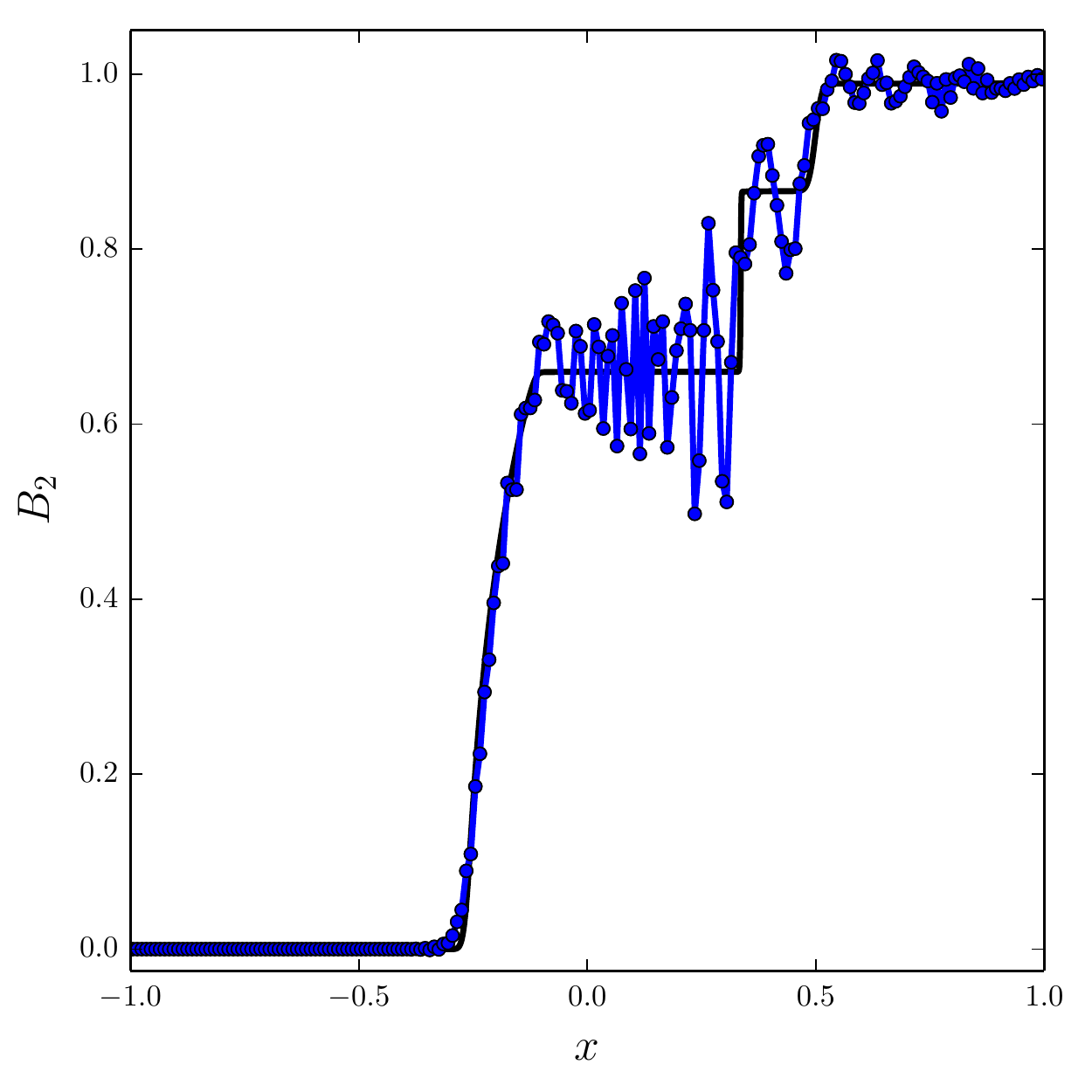}
}
{
\includegraphics[scale=0.575]{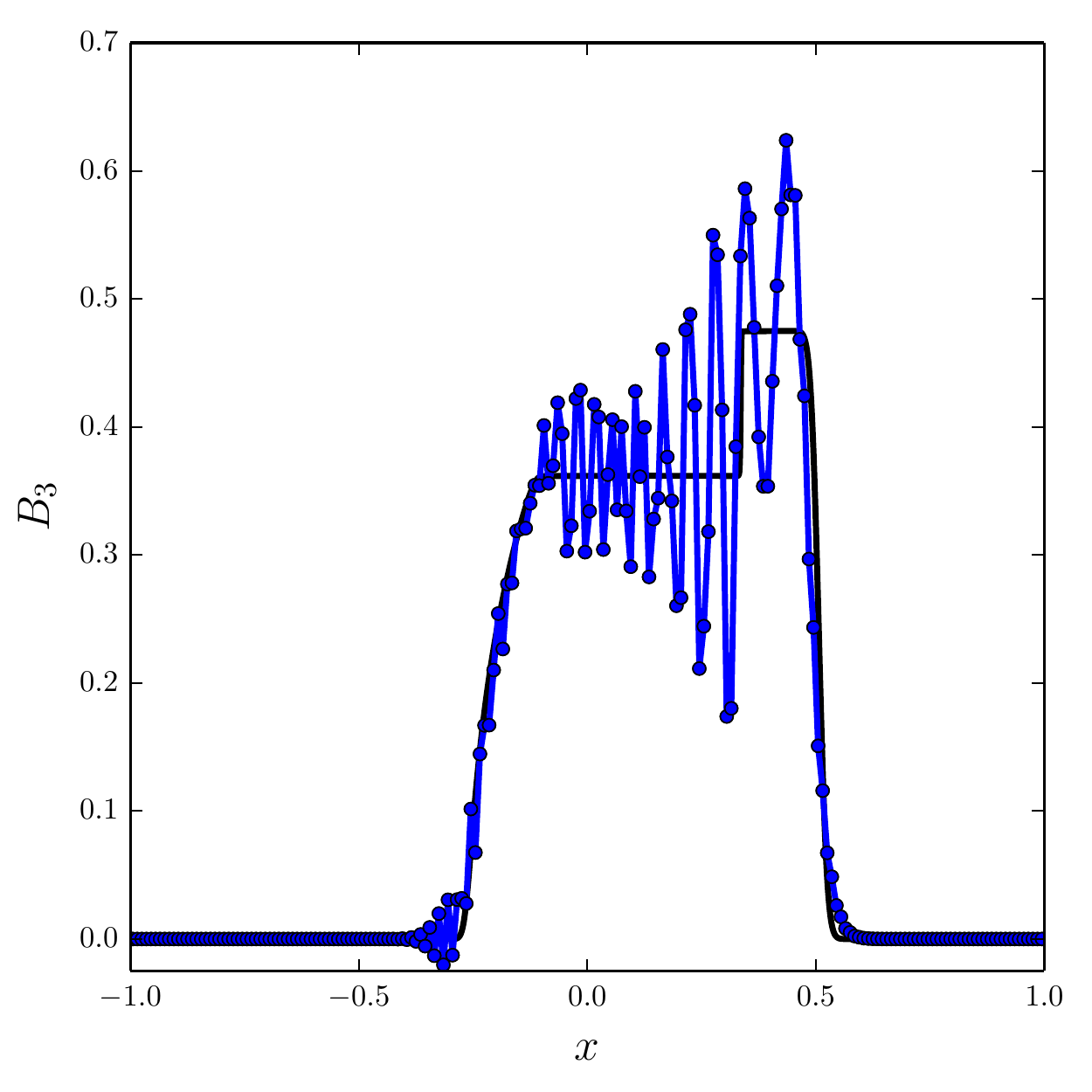}
}
\caption{We present, in blue, the entropy conserving approximations of the Ryu and Jones Riemann problem density $\rho$, pressure $p$, $x$, $y$, and $z$ velocity components, and the $y$ and $z$ components of the magnetic field at $T = 0.4$. Solid black represents the reference solution.}
\label{fig:EC_Ryu}
\end{center}
\end{figure}

\subsubsection{Torrilhon Riemann Problem}\label{TorrEC}

Last of the EC computations, we show the results of the EC approximate solution for the Torrilhon Riemann problem in Fig. \ref{fig:EC_Tor}. This computed solution tracks the complex features of the flow, but there large oscillations are still present, as we have come to expect for a discretely entropy conserving approximation.
\begin{figure}[!ht]
\begin{center}
{
\includegraphics[scale=0.575]{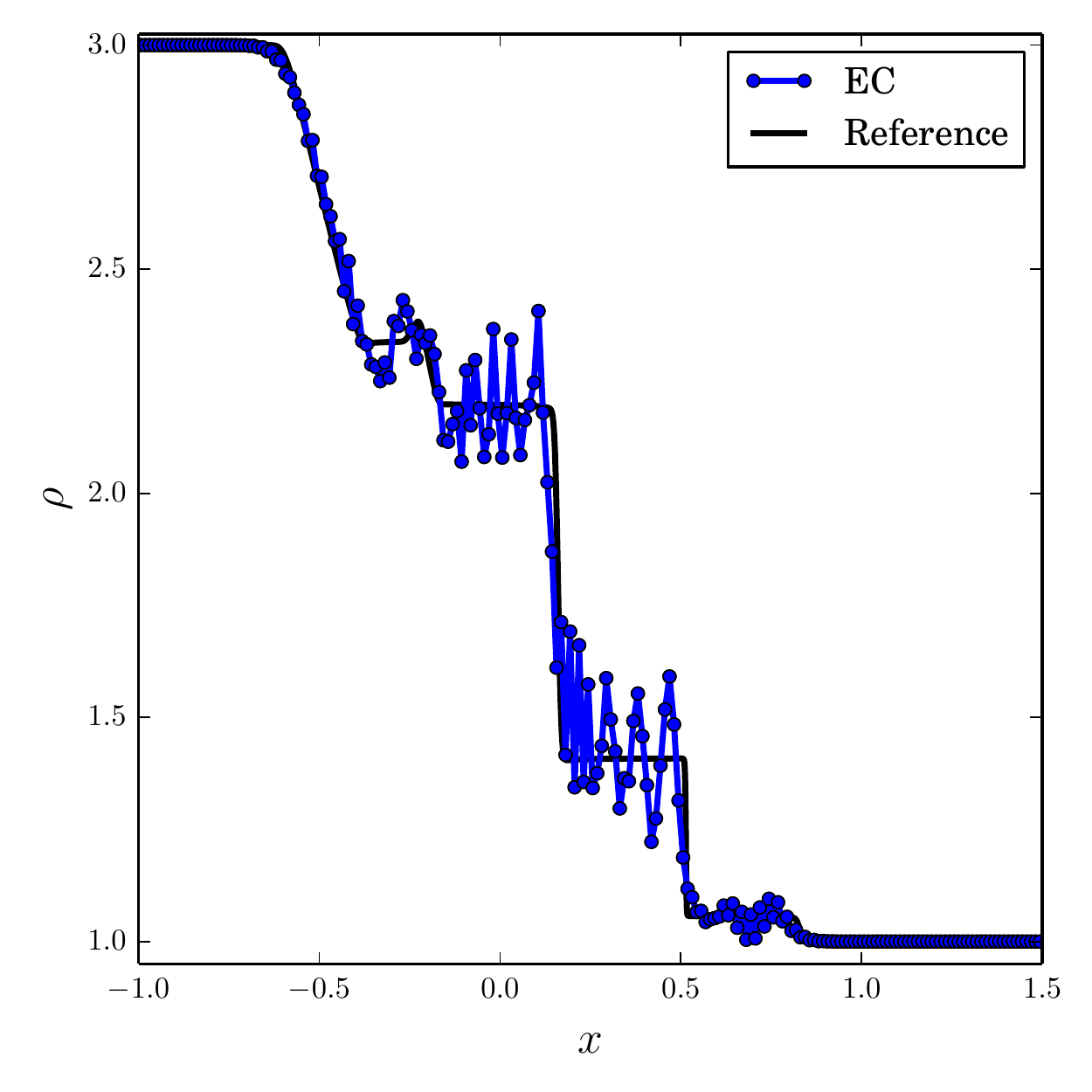}
}
{
\includegraphics[scale=0.575]{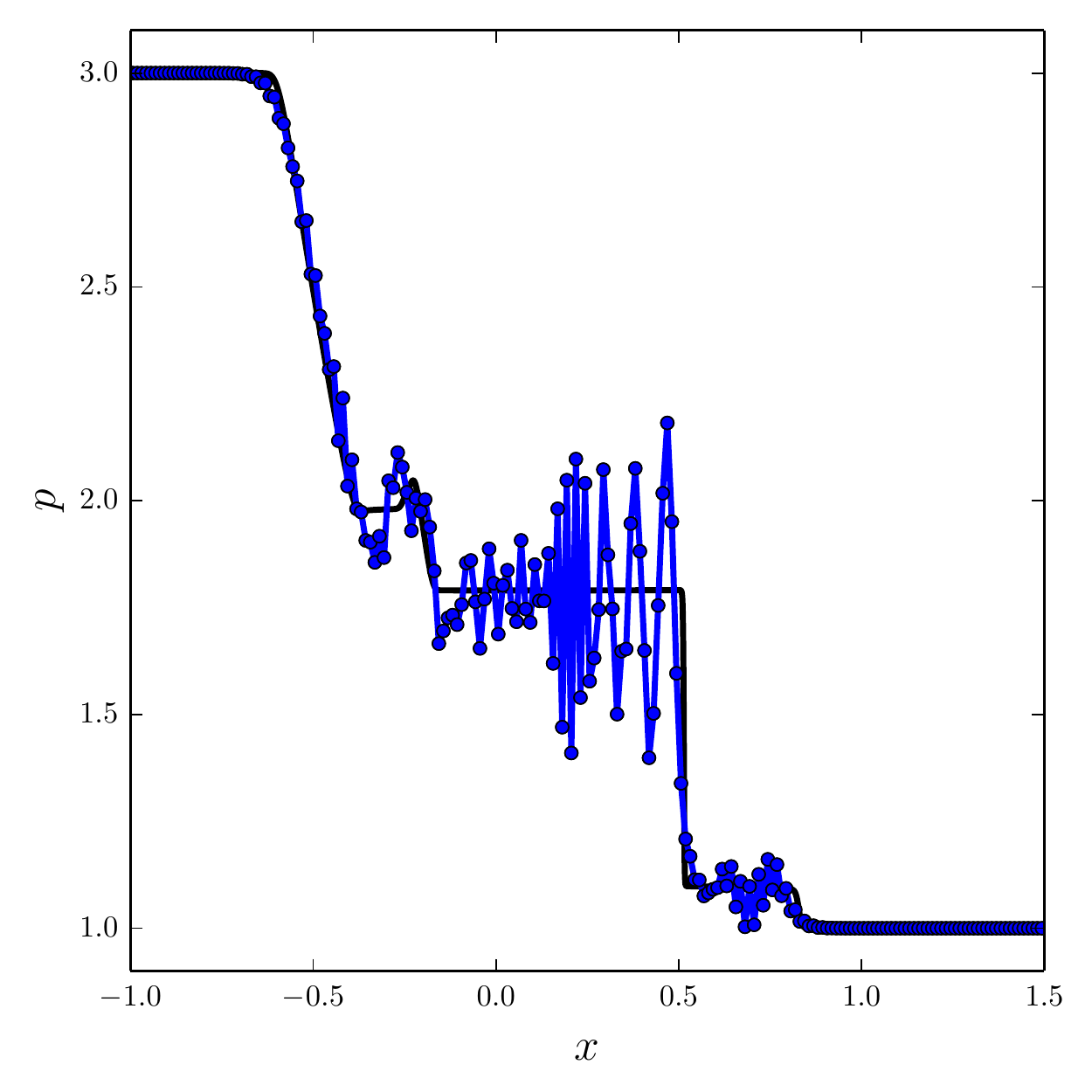}
}
\\
{
\includegraphics[scale=0.415]{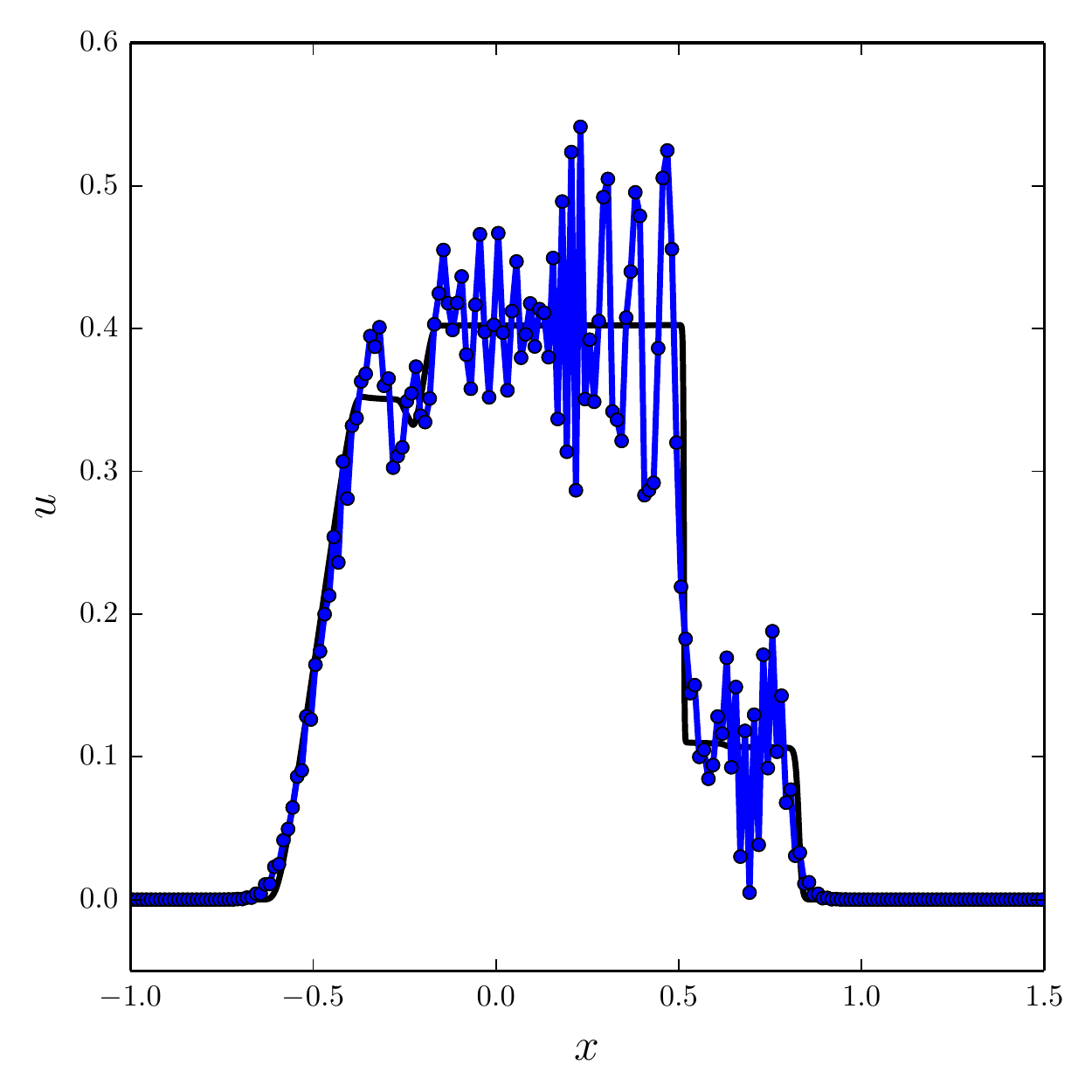}
}
{
\includegraphics[scale=0.415]{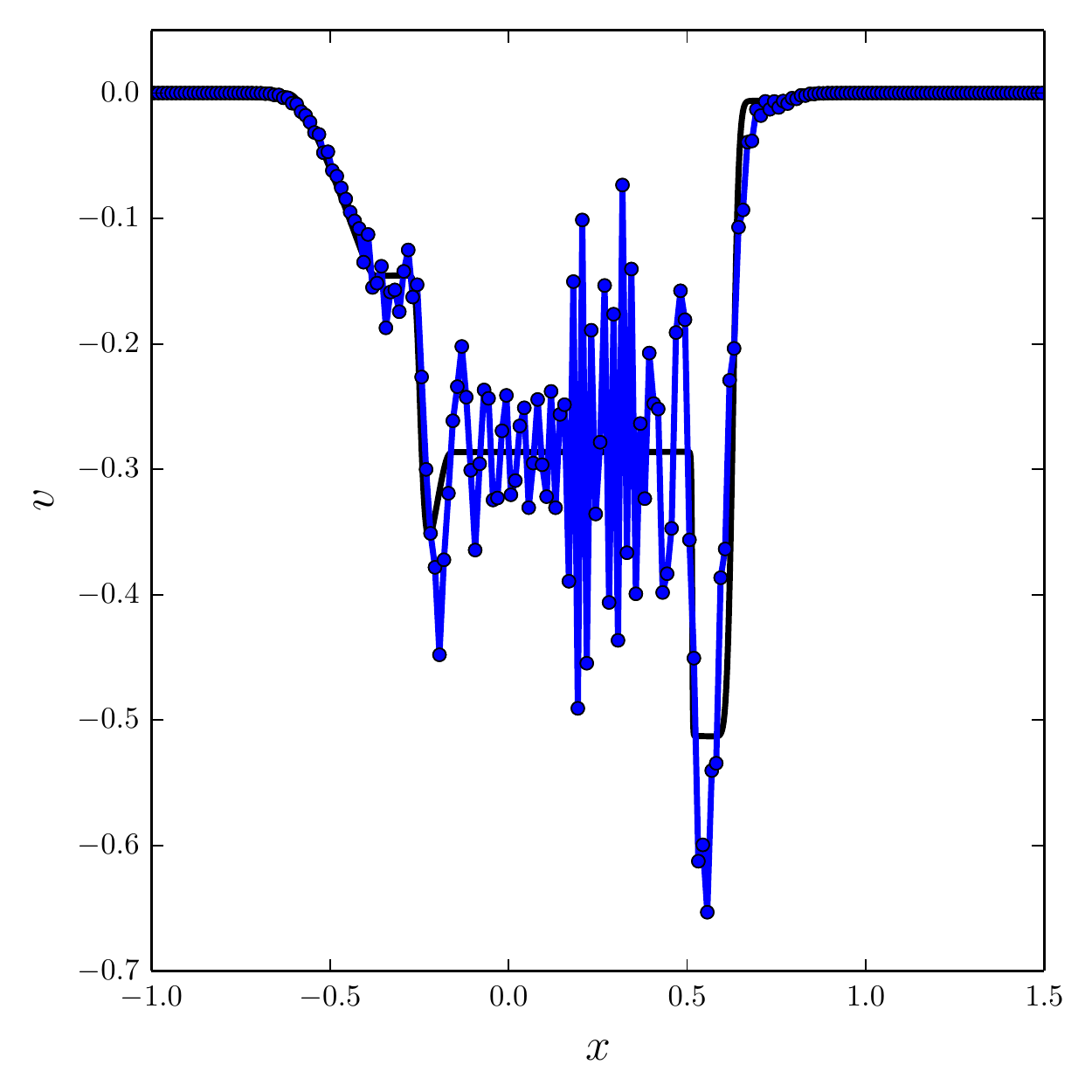}
}
{
\includegraphics[scale=0.415]{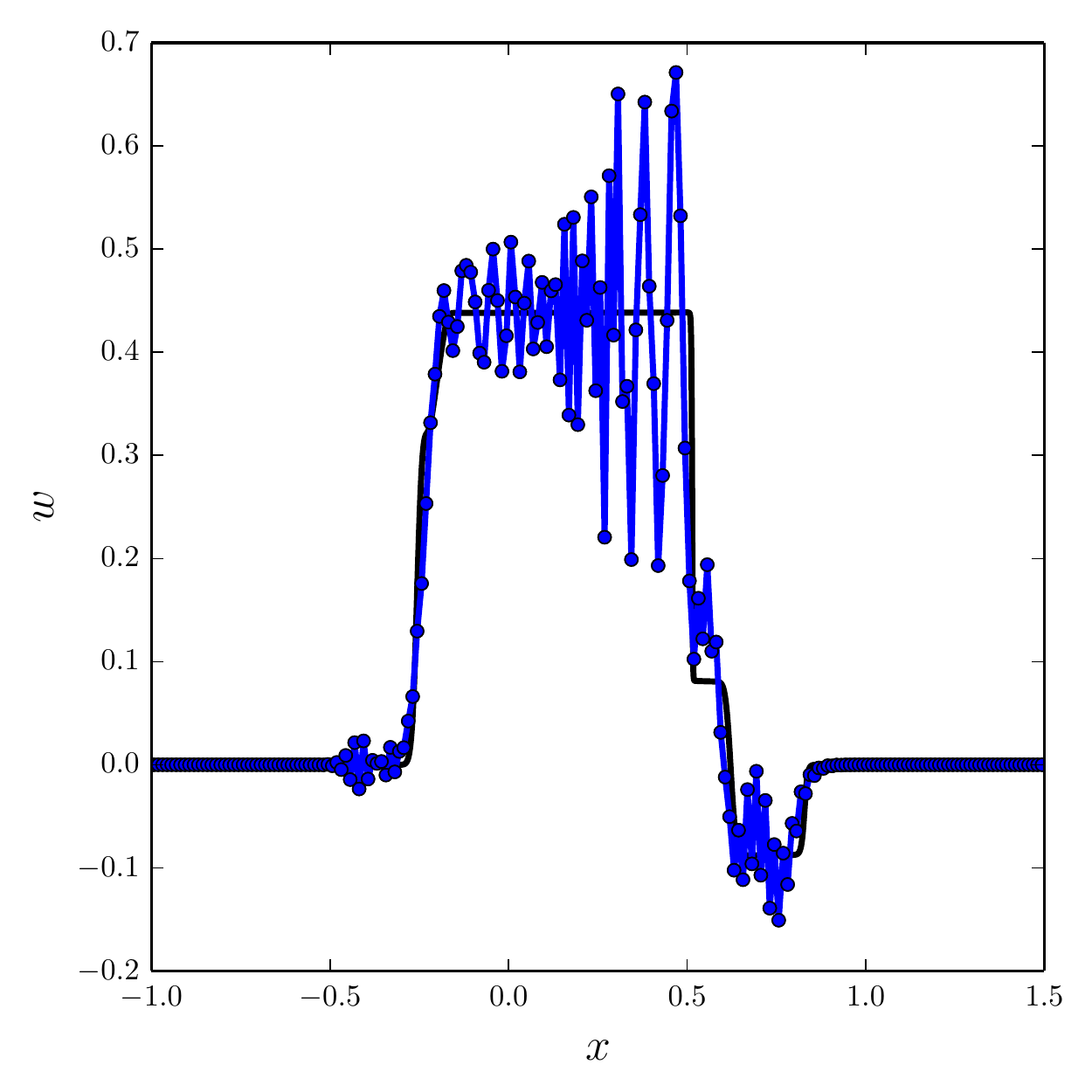}
}
\\
{
\includegraphics[scale=0.575]{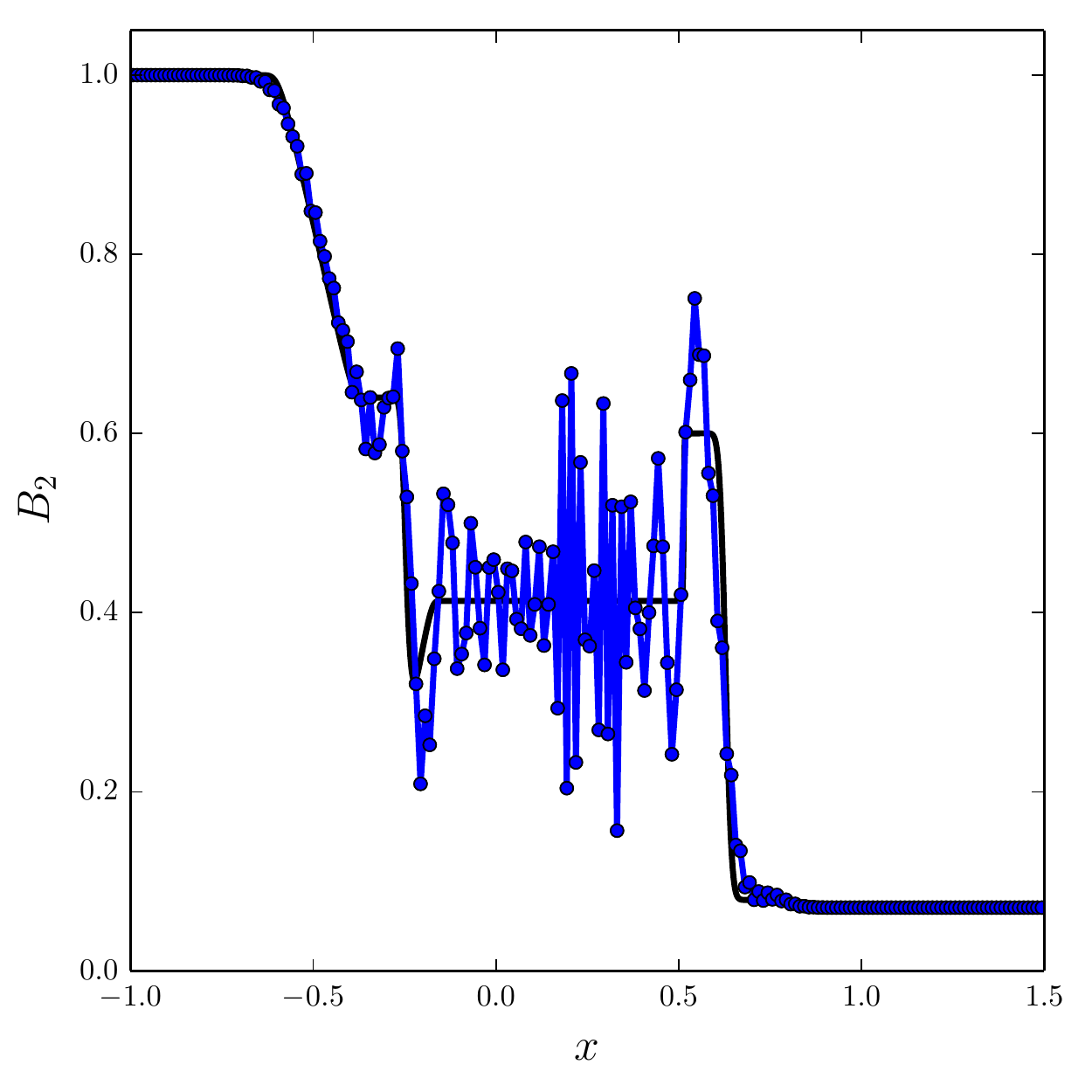}
}
{
\includegraphics[scale=0.575]{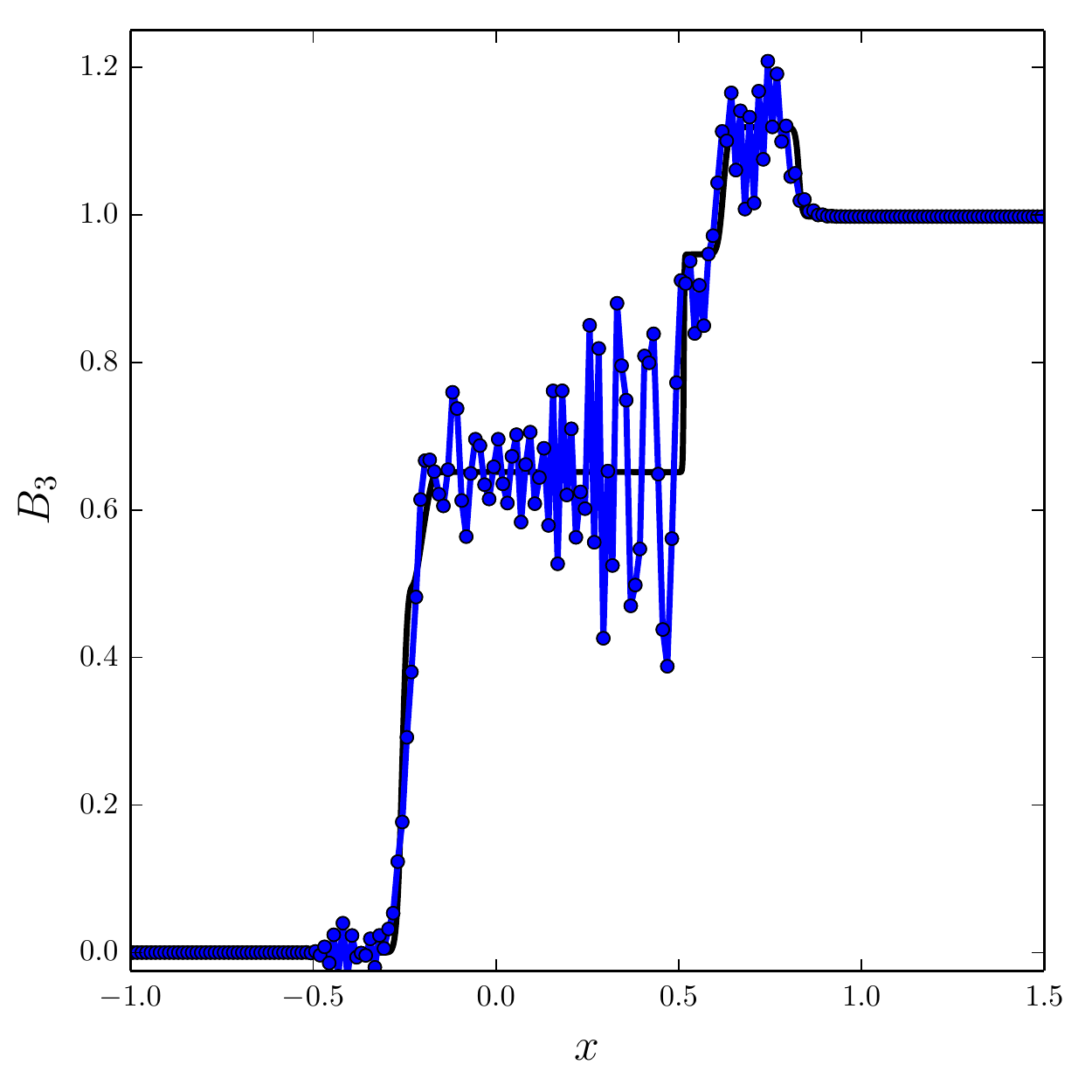}
}
\caption{The computed entropy conserving approximations (in blue) of the density $\rho$, pressure $p$, $x$, $y$, and $z$ velocity components, and the $y$ and $z$ components of the magnetic field at $T = 0.4$ for the Torrilhon Riemann problem. Solid black represents the reference solution.}
\label{fig:EC_Tor}
\end{center}
\end{figure}

\subsection{Entropy Stable Riemann Problem}\label{ESRiemann}

For the final set of numerical tests we use the ES-Roe and ES-LLF schemes to compute the solution of the three Riemann problems with inflow/outflow type boundary conditions. We then compare the results to those present in the literature. The computation is performed on 200 regular grid cells with CFL number 0.1 up to a final time dictated by the test problem. Again, we use a reference solution created using a high-resolution run of the ES-Roe scheme on 5000 grid cells. 

\subsubsection{Brio and Wu Shock Tube}\label{BrioWuES}

Figure \ref{fig:ES_Brio} shows results for the ES-Roe and ES-LLF finite volume schemes applied to the Brio and Wu shock tube problem. Both entropy stable methods capture the complex behavior, but we see that the ES-Roe scheme is less dissipative than the ES-LLF scheme. The ES-LLF scheme has difficulty capturing the sharp features of the compile MHD flow (like the slow compound wave in the $u$ velocity component). These computations compare well with those presented in the literature \cite{brio1988,powell1994}.
\begin{figure}[!ht]
\begin{center}
{
\includegraphics[scale=0.55]{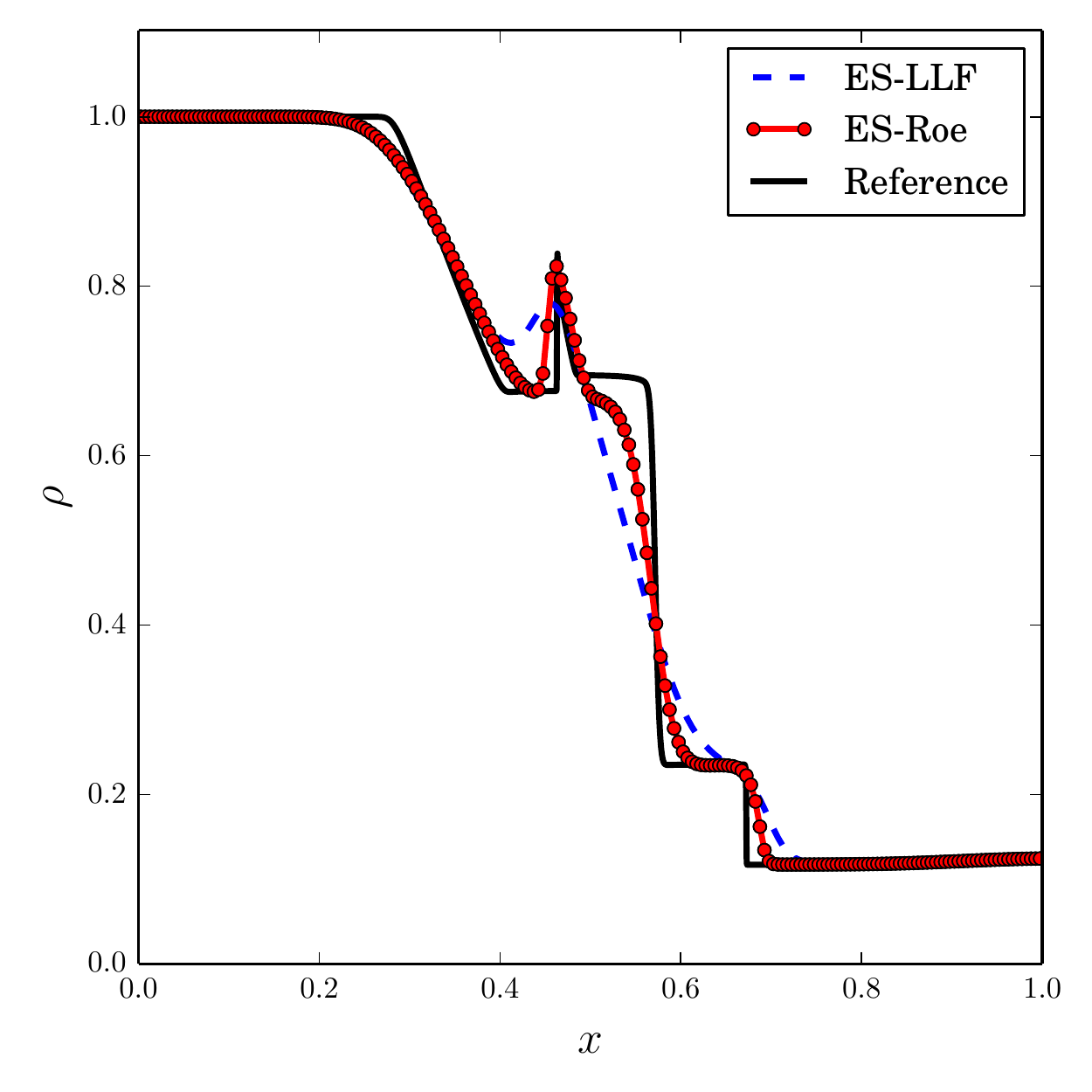}
}
{
\includegraphics[scale=0.55]{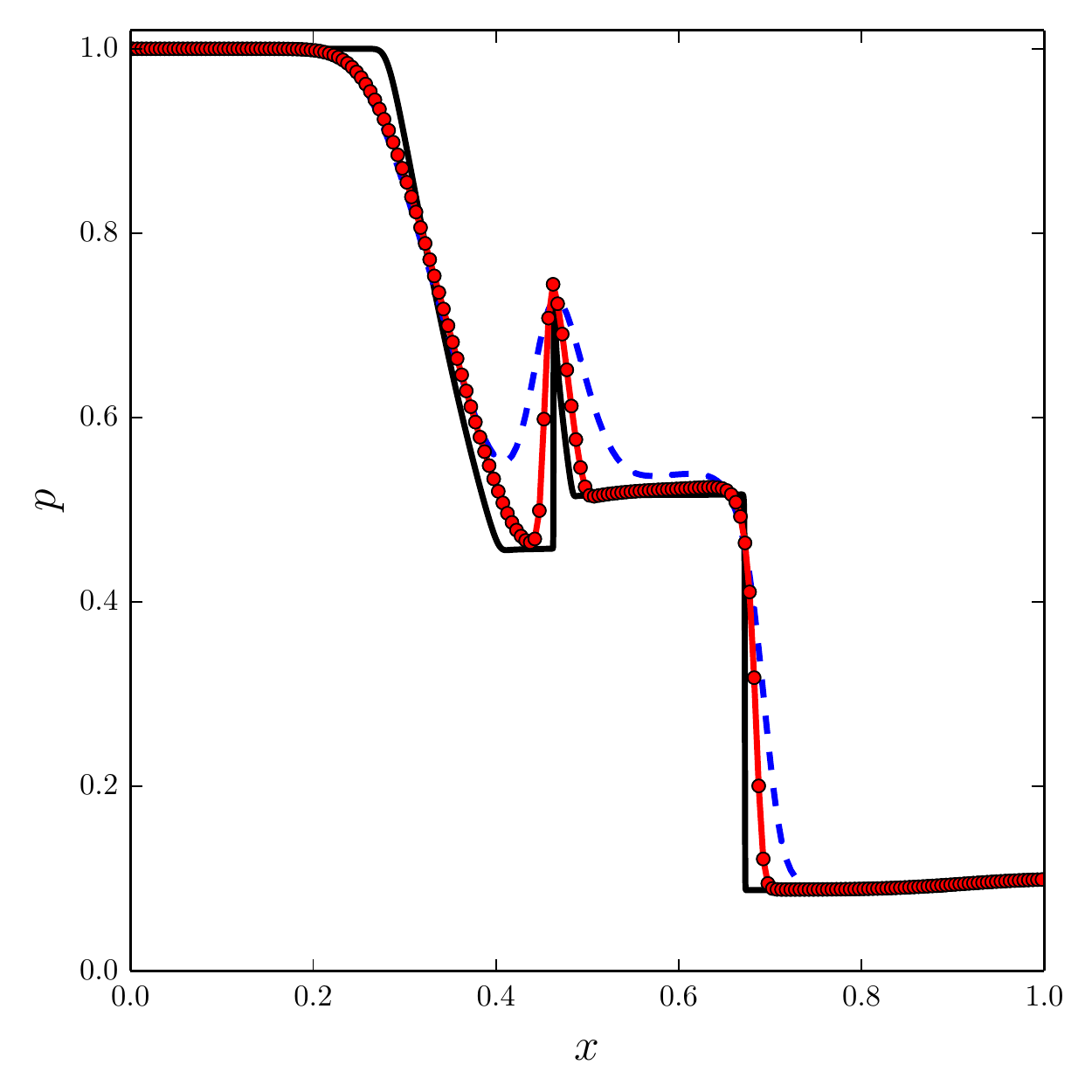}
}
\\
{
\includegraphics[scale=0.55]{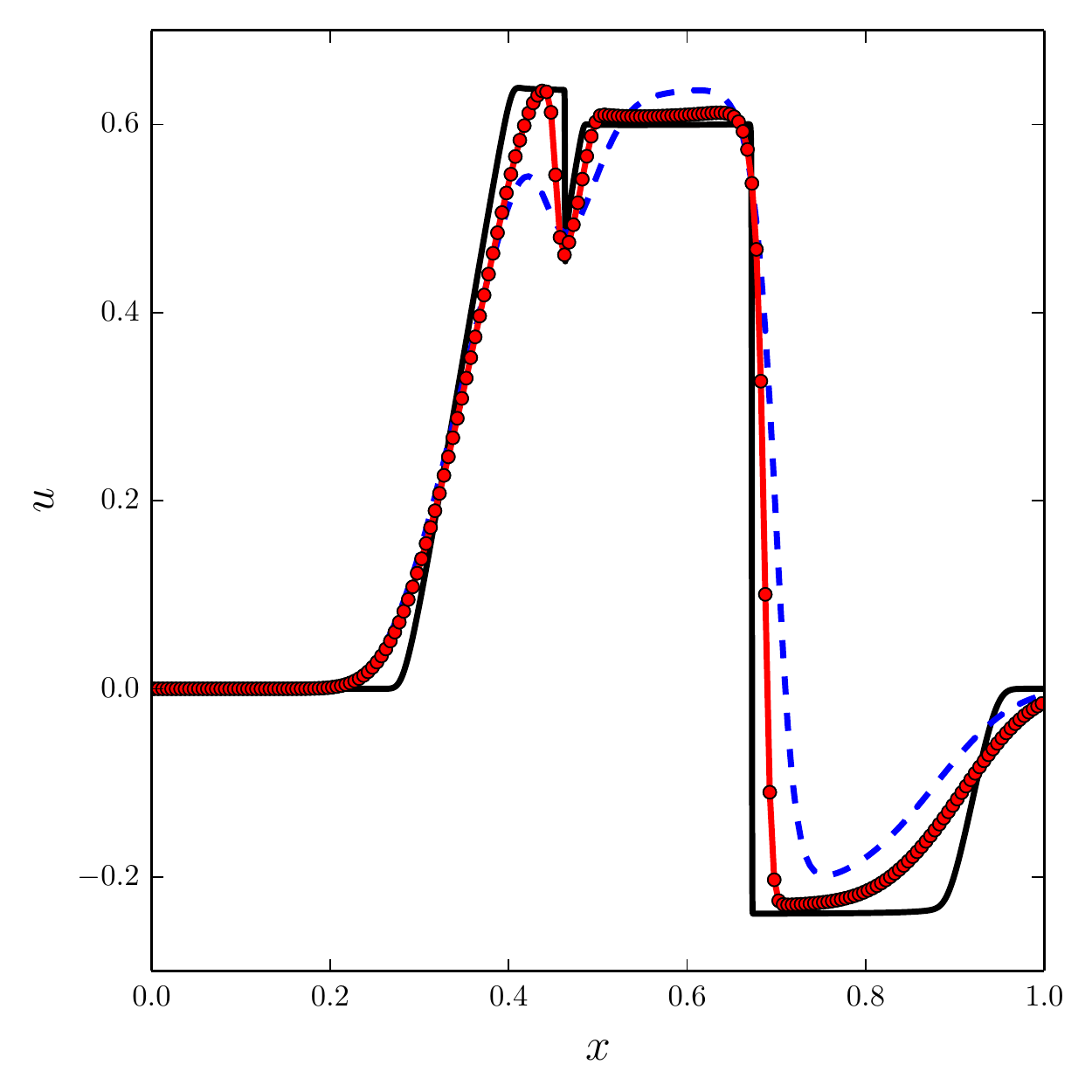}
}
{
\includegraphics[scale=0.55]{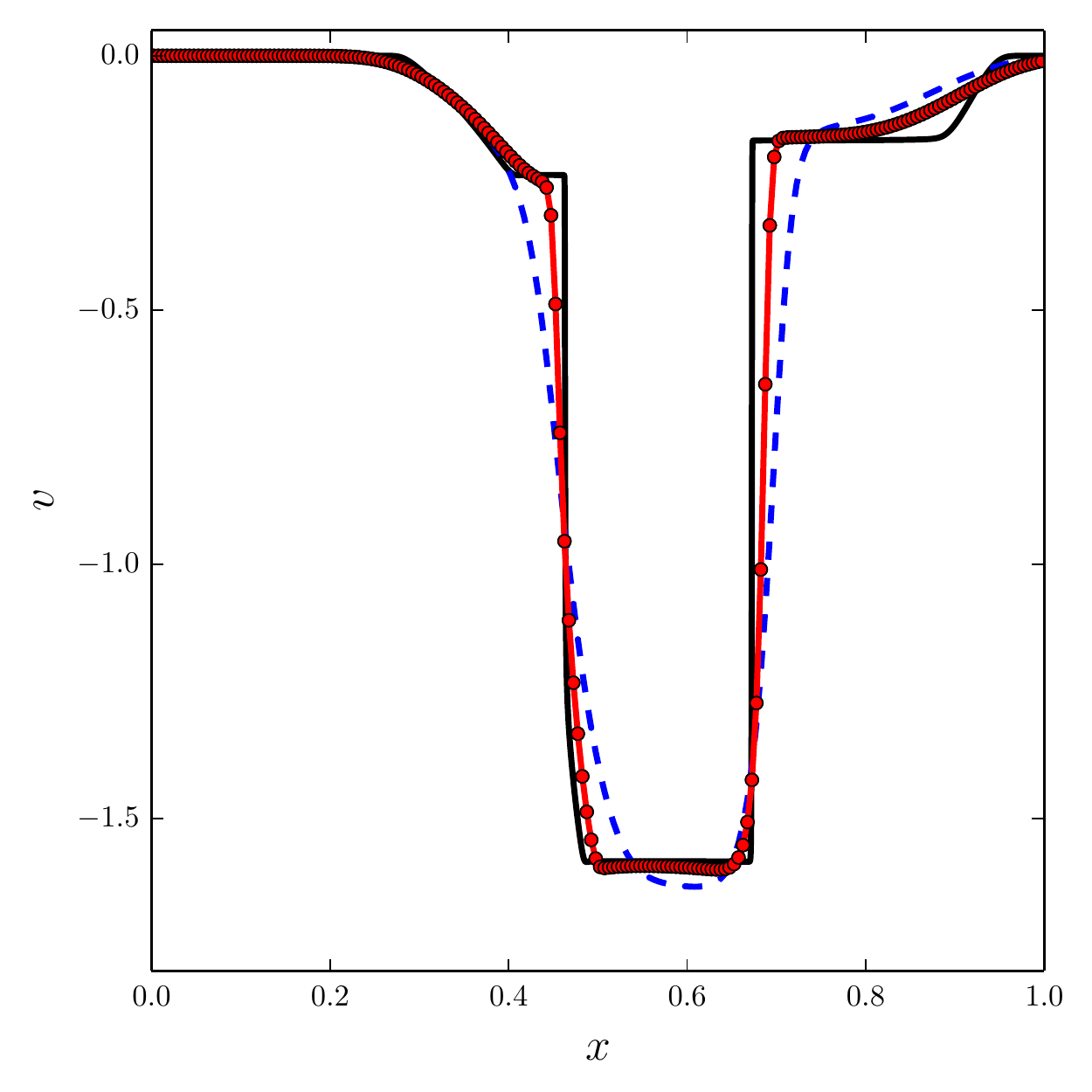}
}
\\
{
\includegraphics[scale=0.55]{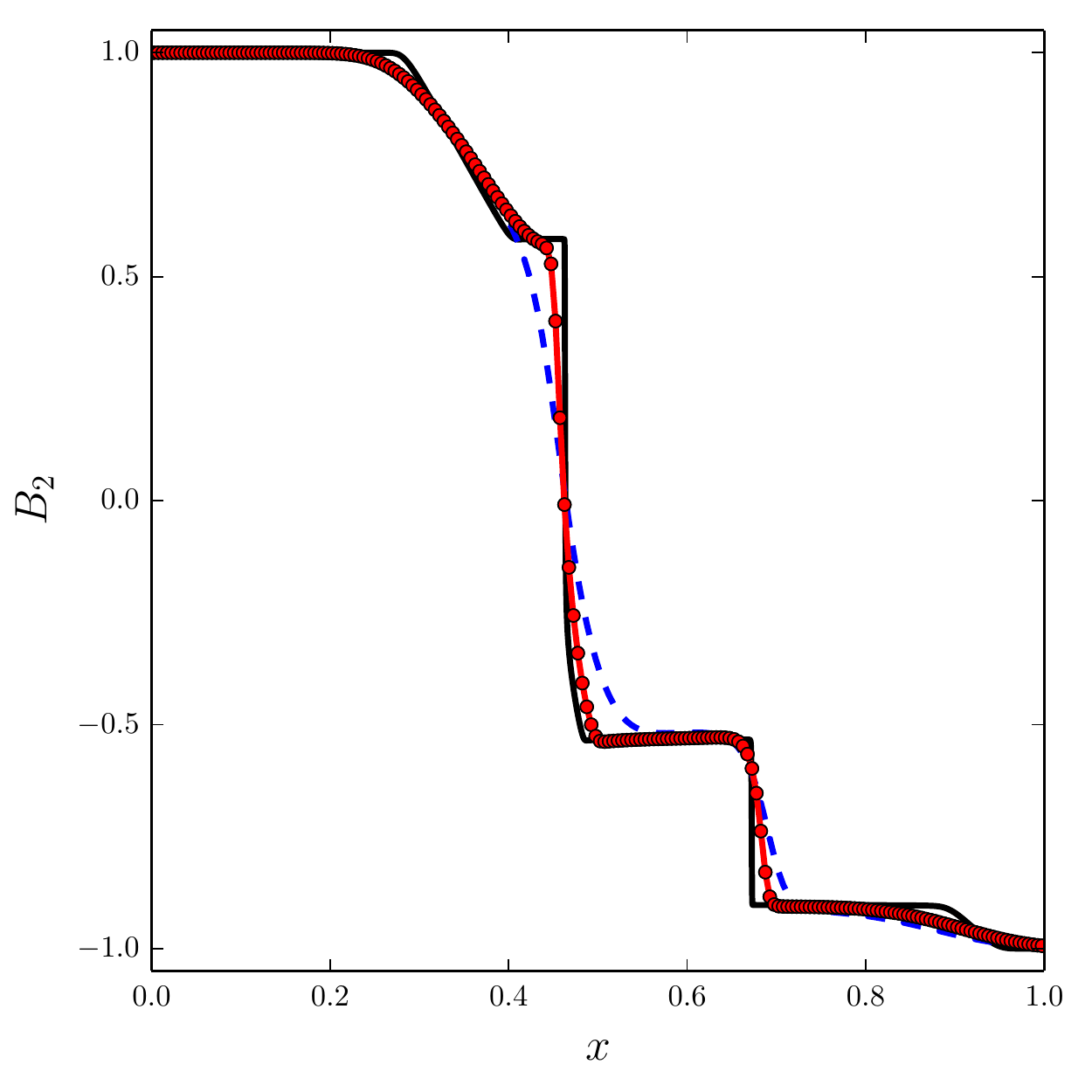}
}
\caption{The entropy stable approximations of the the density $\rho$, pressure $p$, $x$ and $y$ velocity components, and the $y$ of the magnetic field of the Brio and Wu shock tube at $T = 0.12$. Solid black is the reference solution, dashed blue is the ES-LLF scheme, and red with knots is the ES-Roe scheme.}
\label{fig:ES_Brio}
\end{center}
\end{figure}

\subsubsection{Ryu and Jones Riemann Problem}\label{RyuES}

Figure \ref{fig:ES_Ryu} presents the computed ES-Roe and ES-LLF solution for the Ryu and Jones Riemann problem. Each entropy stable scheme capture the complex behavior of the MHD flow and, for the weaker shocks, we see that there is less difference between the ES-Roe scheme and the more dissipative ES-LLF scheme. Our computations are, again, comparable to those found in the literature \cite{rossmanith2002,ryu1994}.
\begin{figure}[!ht]
\begin{center}
{
\includegraphics[scale=0.575]{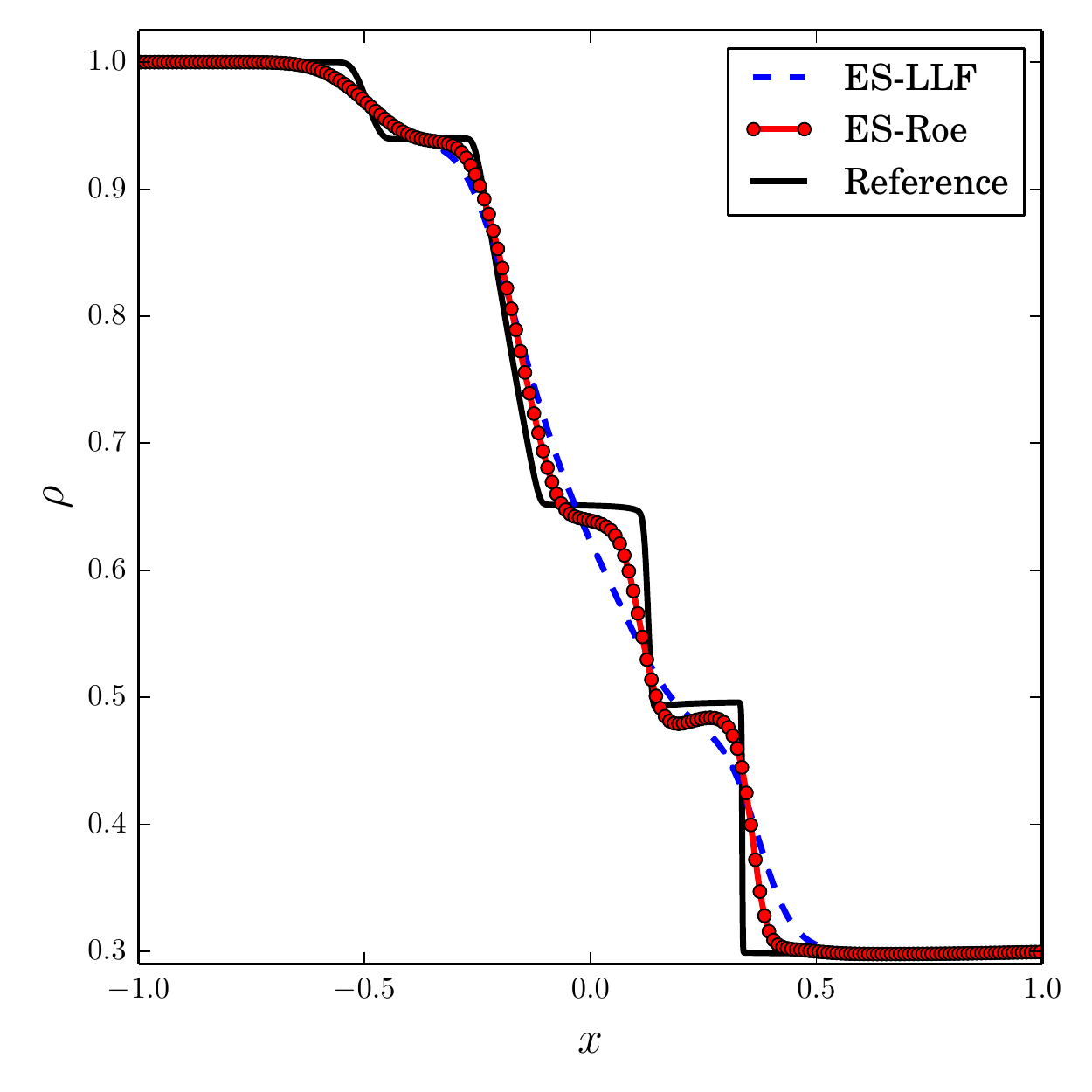}
}
{
\includegraphics[scale=0.575]{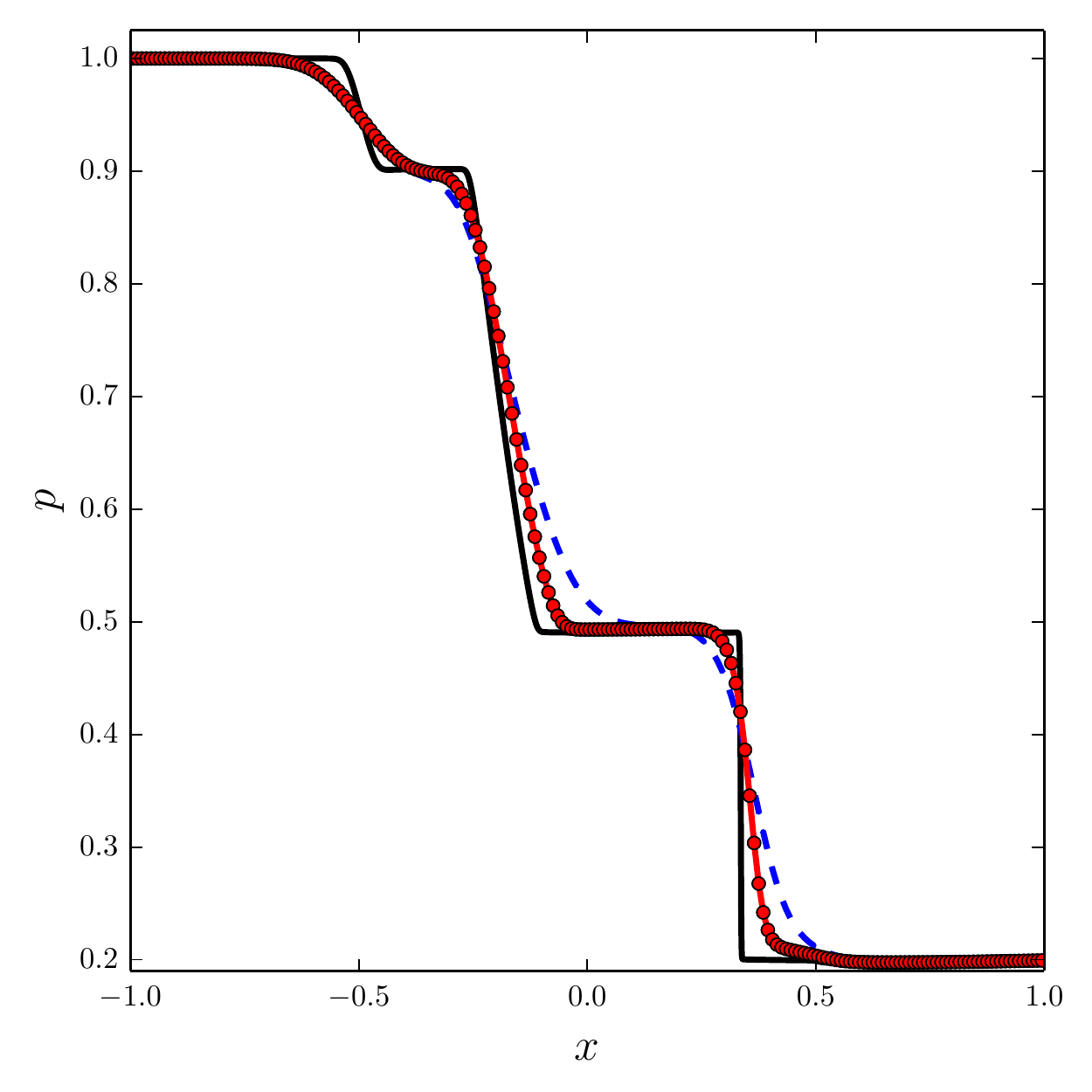}
}
\\
{
\includegraphics[scale=0.415]{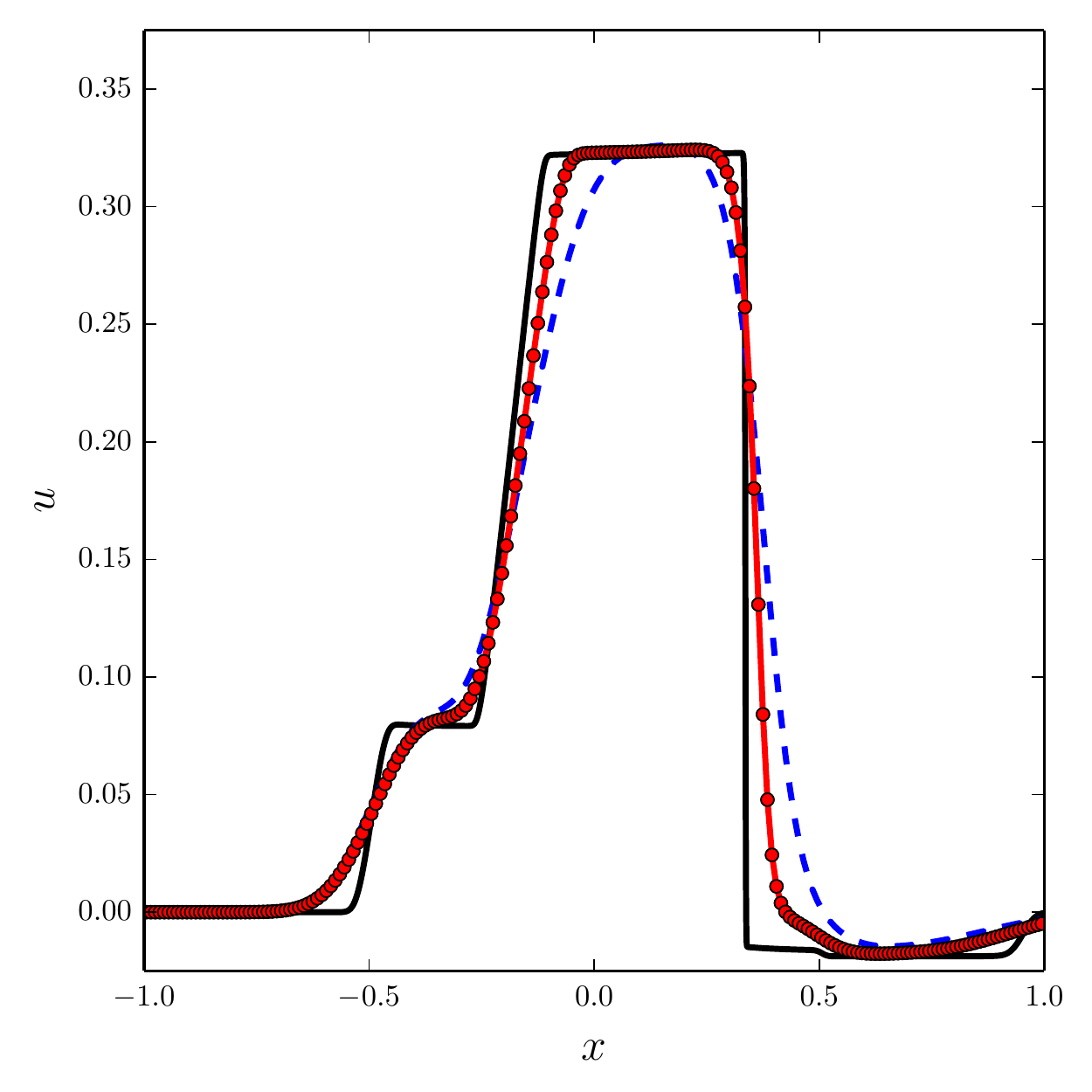}
}
{
\includegraphics[scale=0.415]{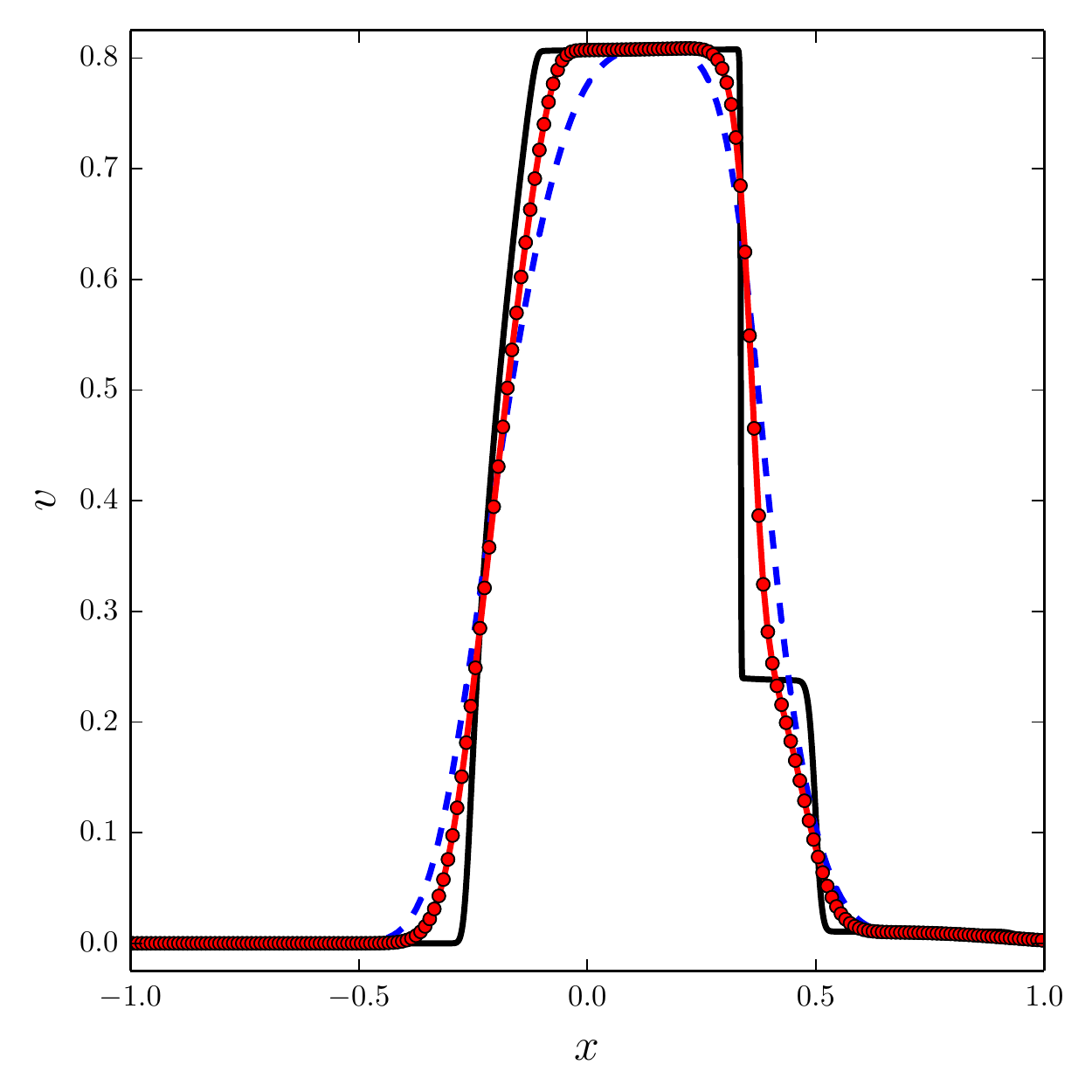}
}
{
\includegraphics[scale=0.415]{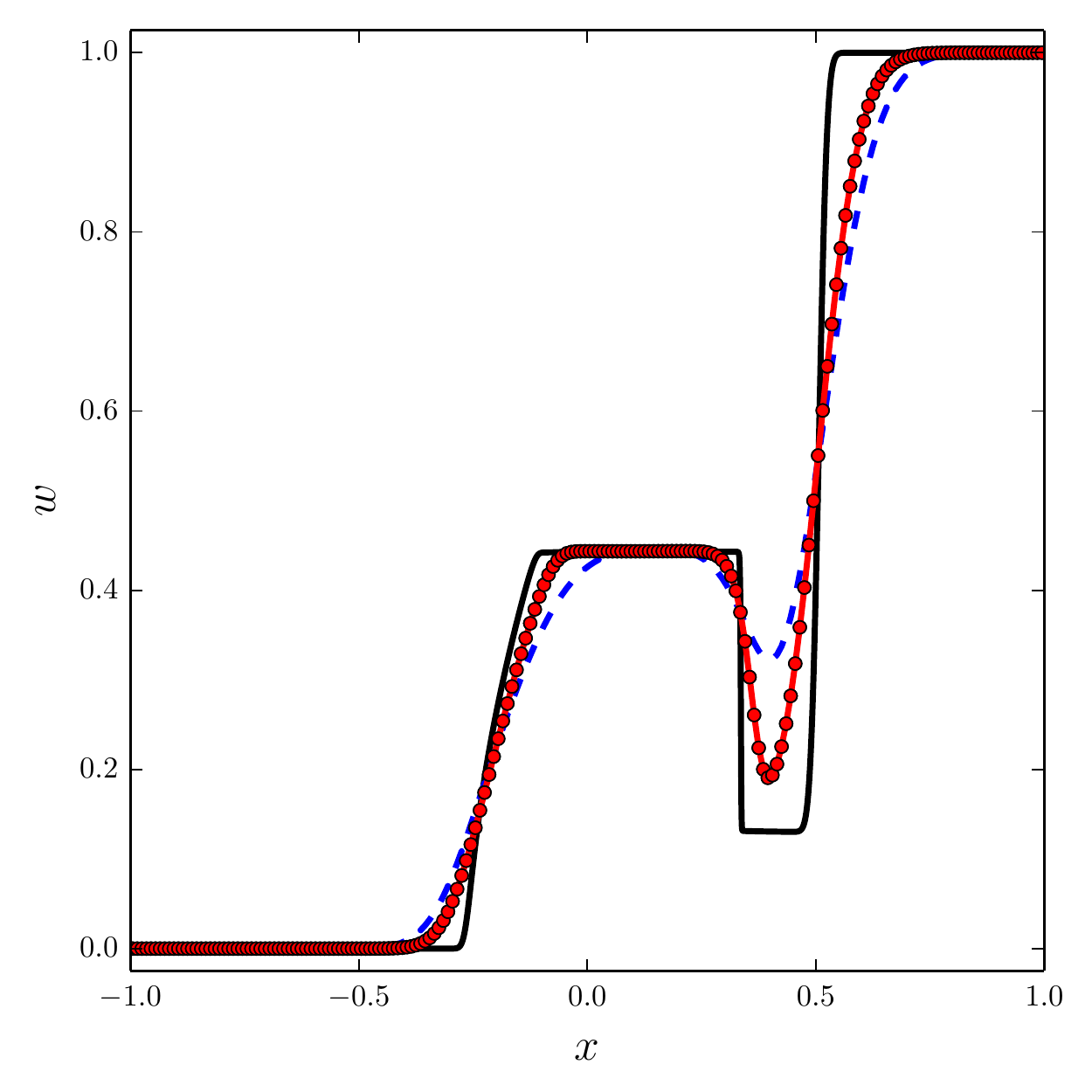}
}
\\
{
\includegraphics[scale=0.575]{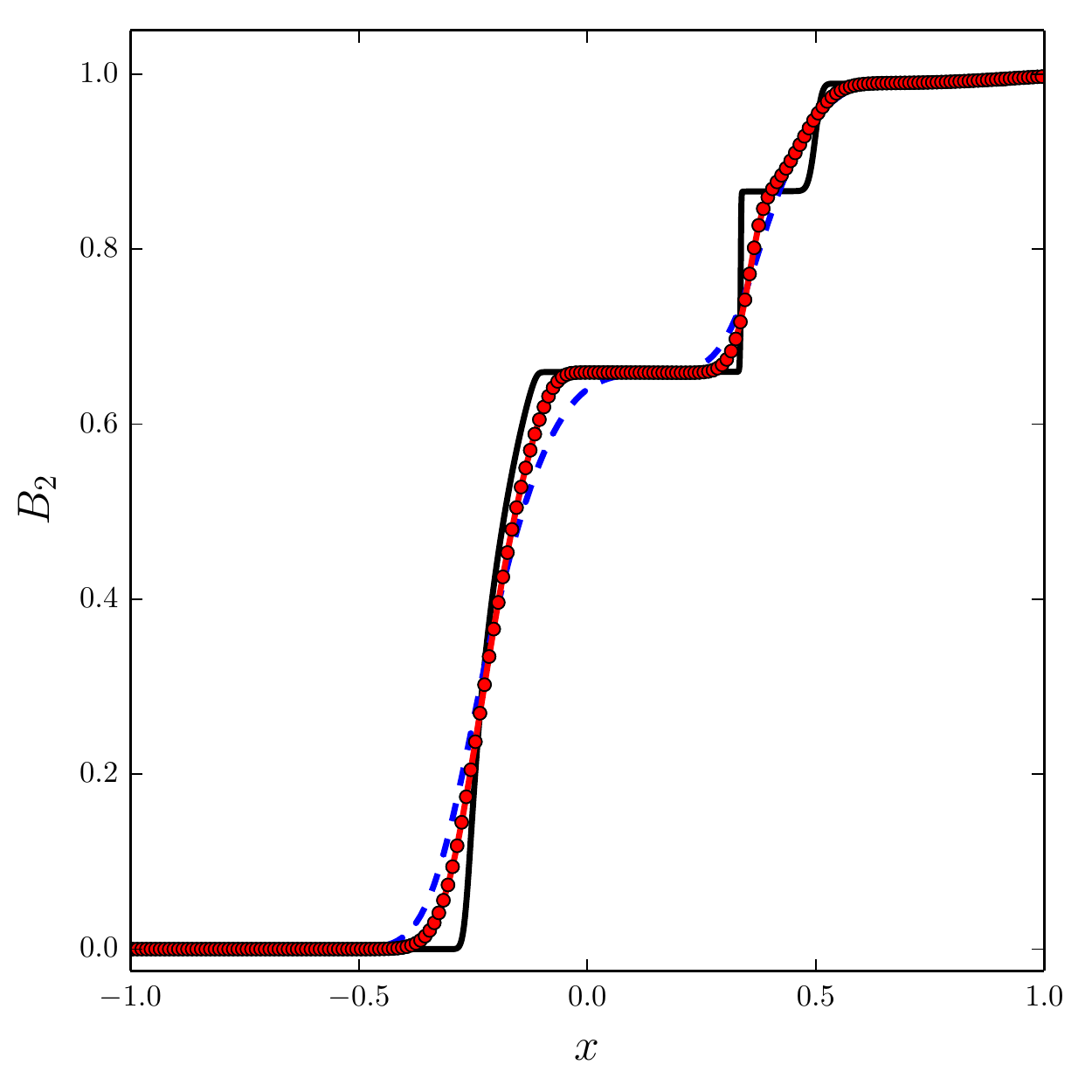}
}
{
\includegraphics[scale=0.575]{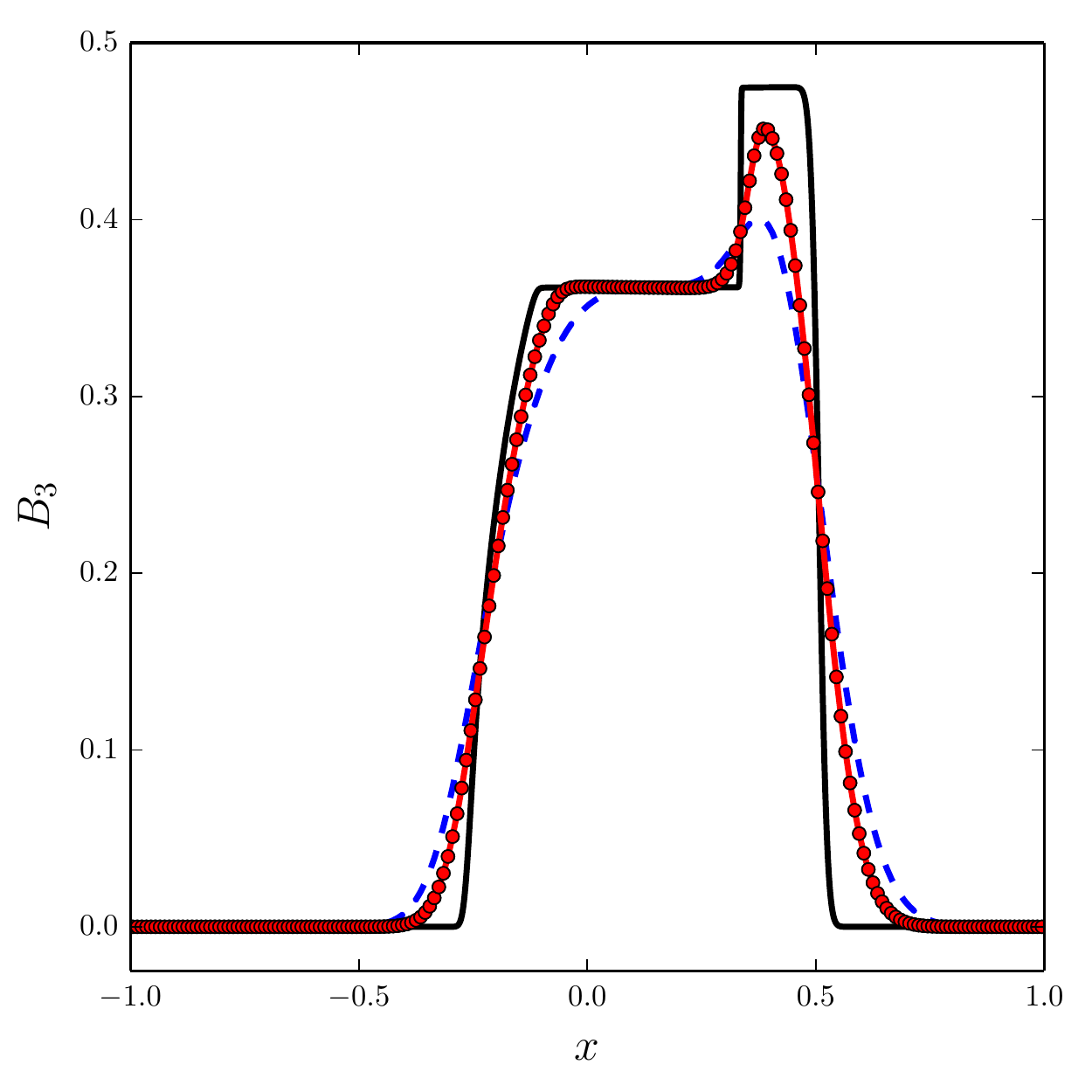}
}
\caption{The entropy stable approximations of the the density $\rho$, pressure $p$, $x$, $y$, and $z$ velocity components, and the $y$ and $z$ components of the magnetic field of the Ryu and Jones problem at $T = 0.4$. Solid black is the reference solution, dashed blue is the ES-LLF scheme, and red with knots is the ES-Roe scheme.}
\label{fig:ES_Ryu}
\end{center}
\end{figure}

\subsubsection{Torrilhon Riemann Problem}\label{TorrES}

Finally, Fig. \ref{fig:ES_Tor} presents the computed ES-Roe and ES-LLF solution for the Torrilhon Riemann problem. Again, both entropy stable schemes capture the complex behavior of the MHD flow. Interestingly, we note that for this test problem that the behavior of both entropy stable schemes is nearly identical. Our computations match well with those found in Torrilhon \cite{torrilhon2003}.
\begin{figure}[!ht]
\begin{center}
{
\includegraphics[scale=0.575]{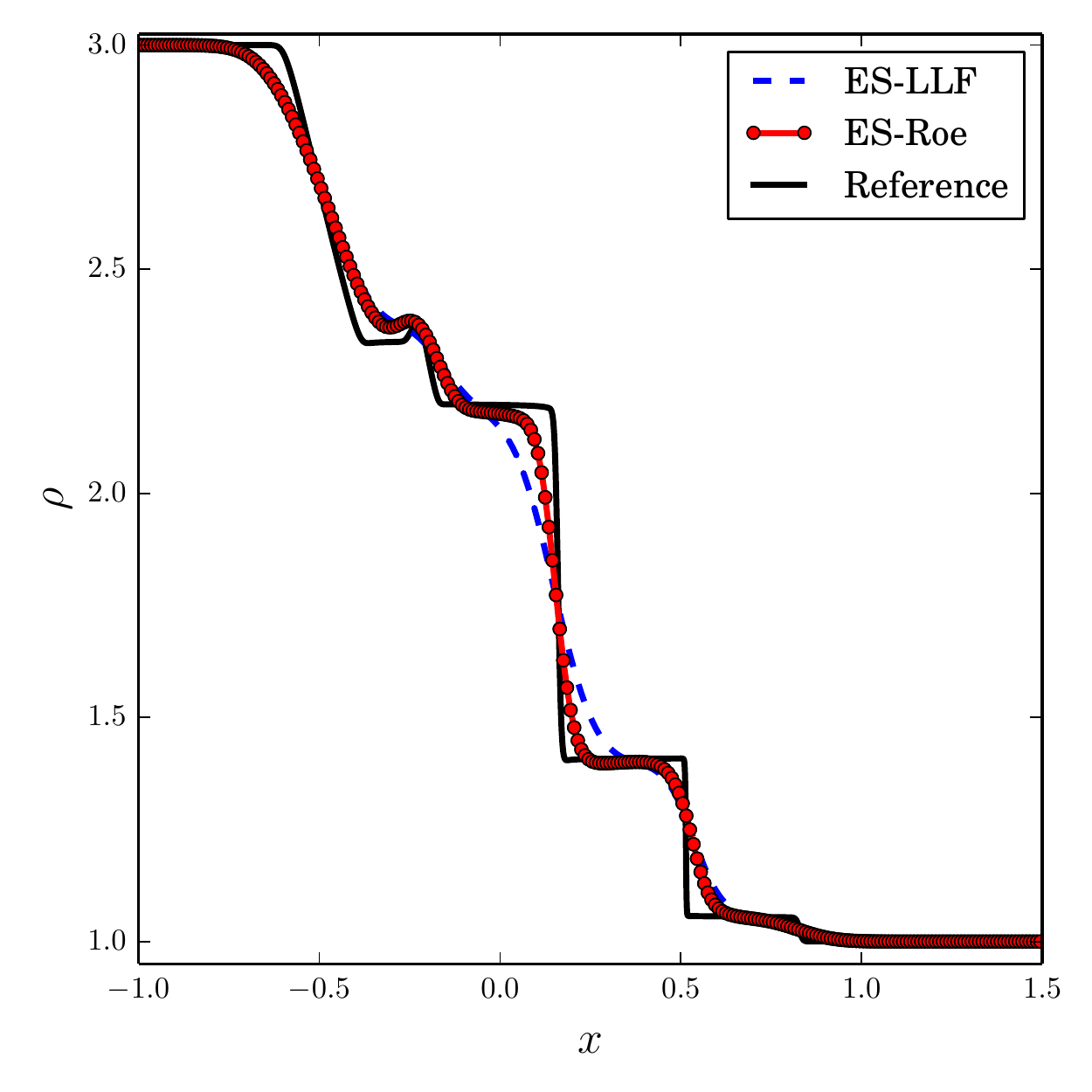}
}
{
\includegraphics[scale=0.575]{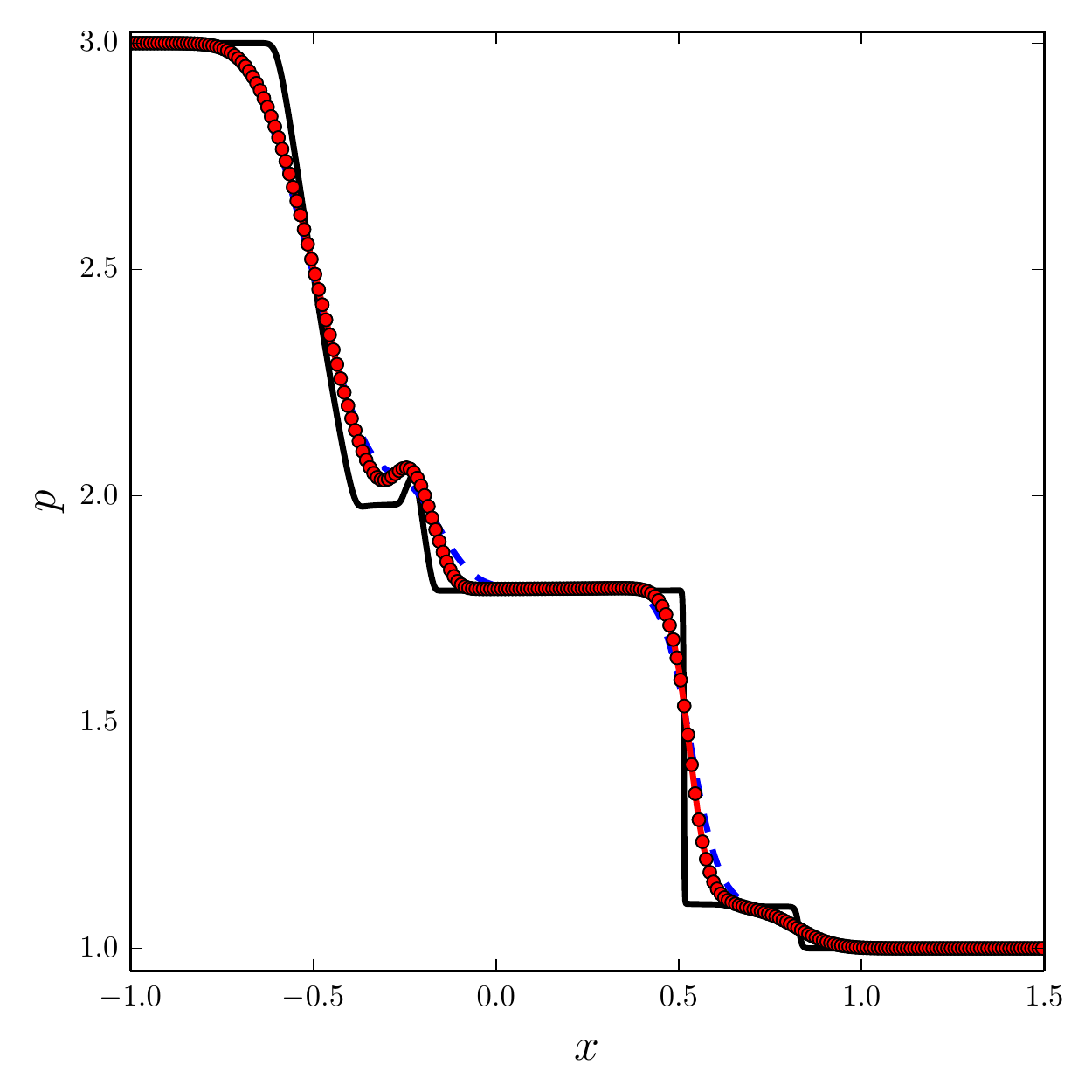}
}
\\
{
\includegraphics[scale=0.415]{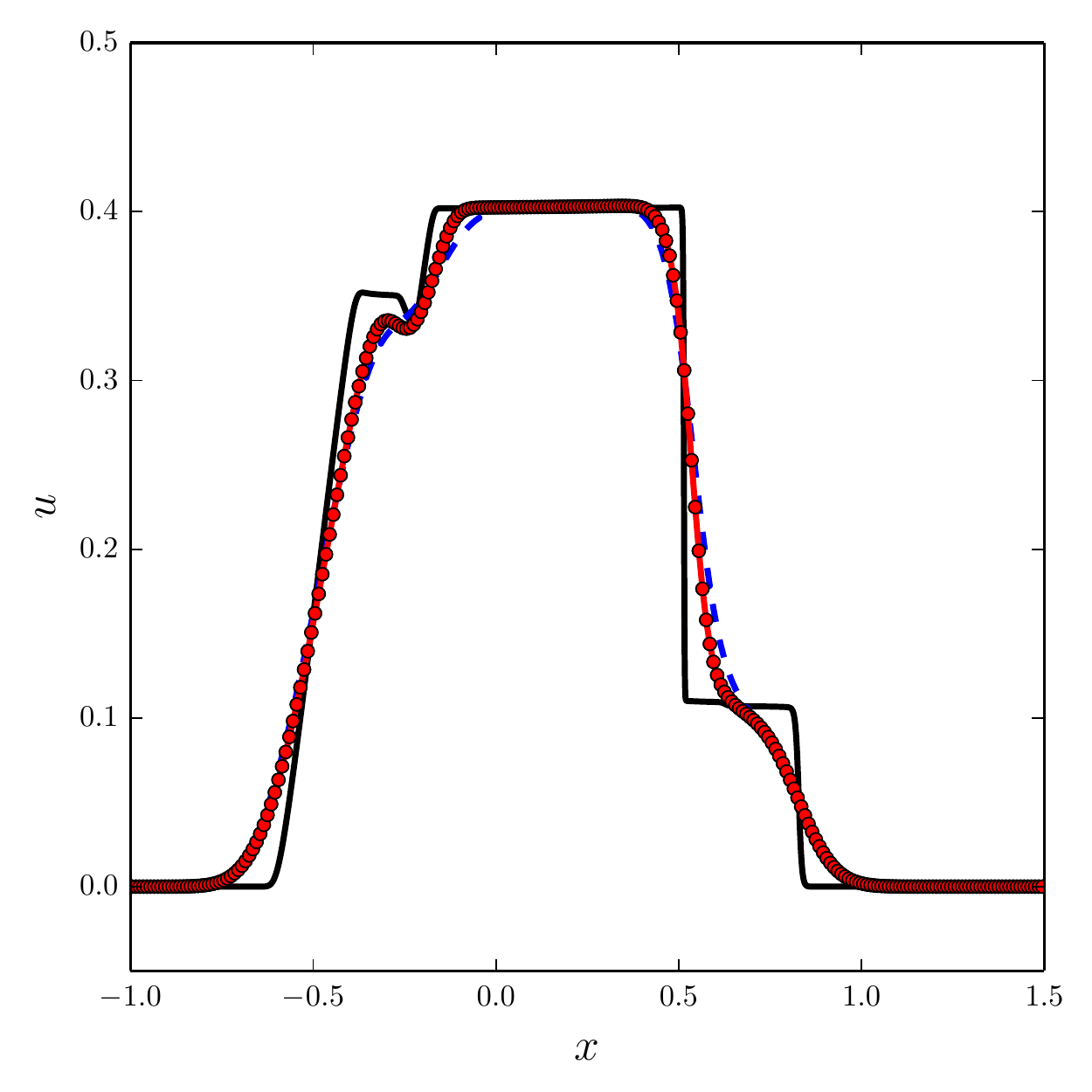}
}
{
\includegraphics[scale=0.415]{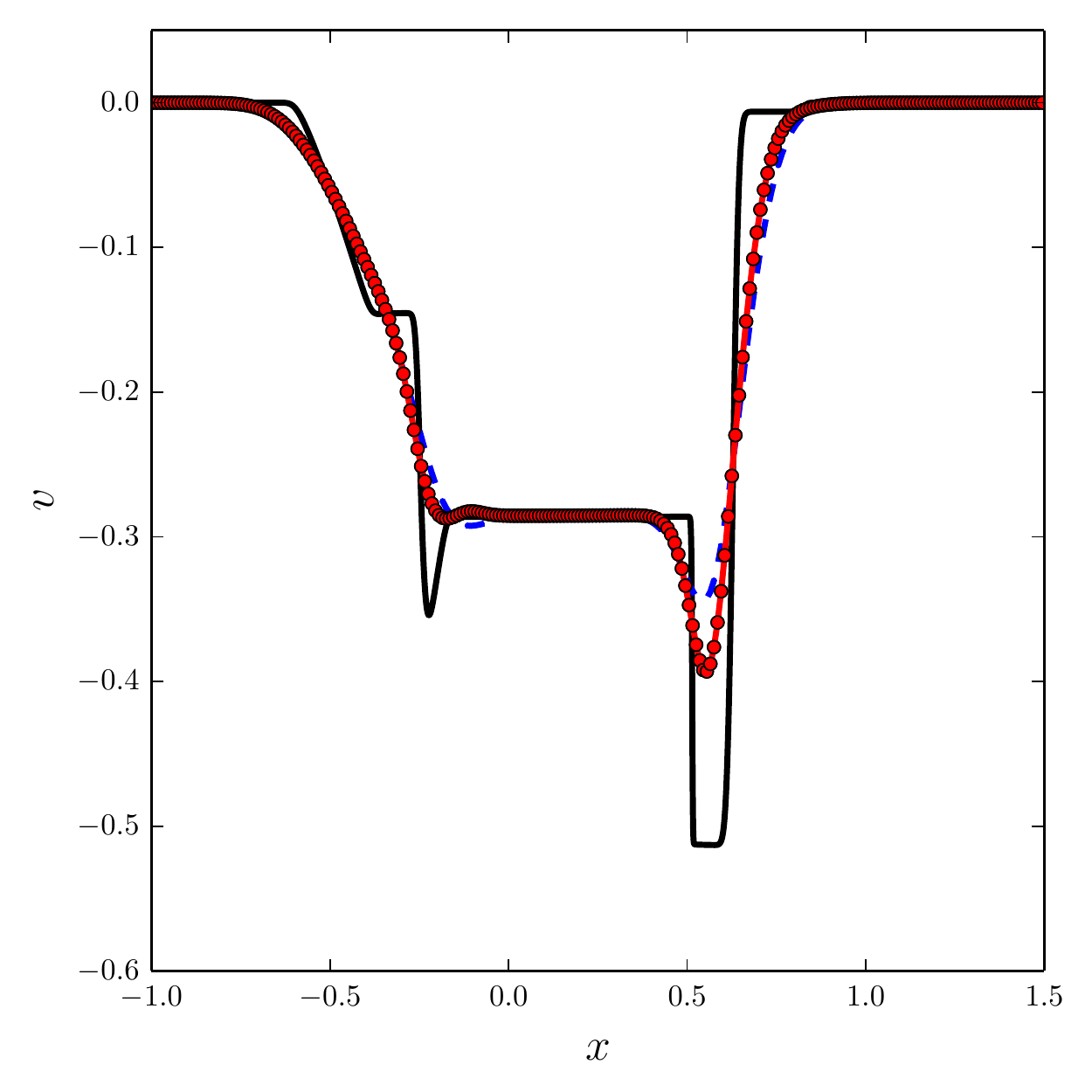}
}
{
\includegraphics[scale=0.415]{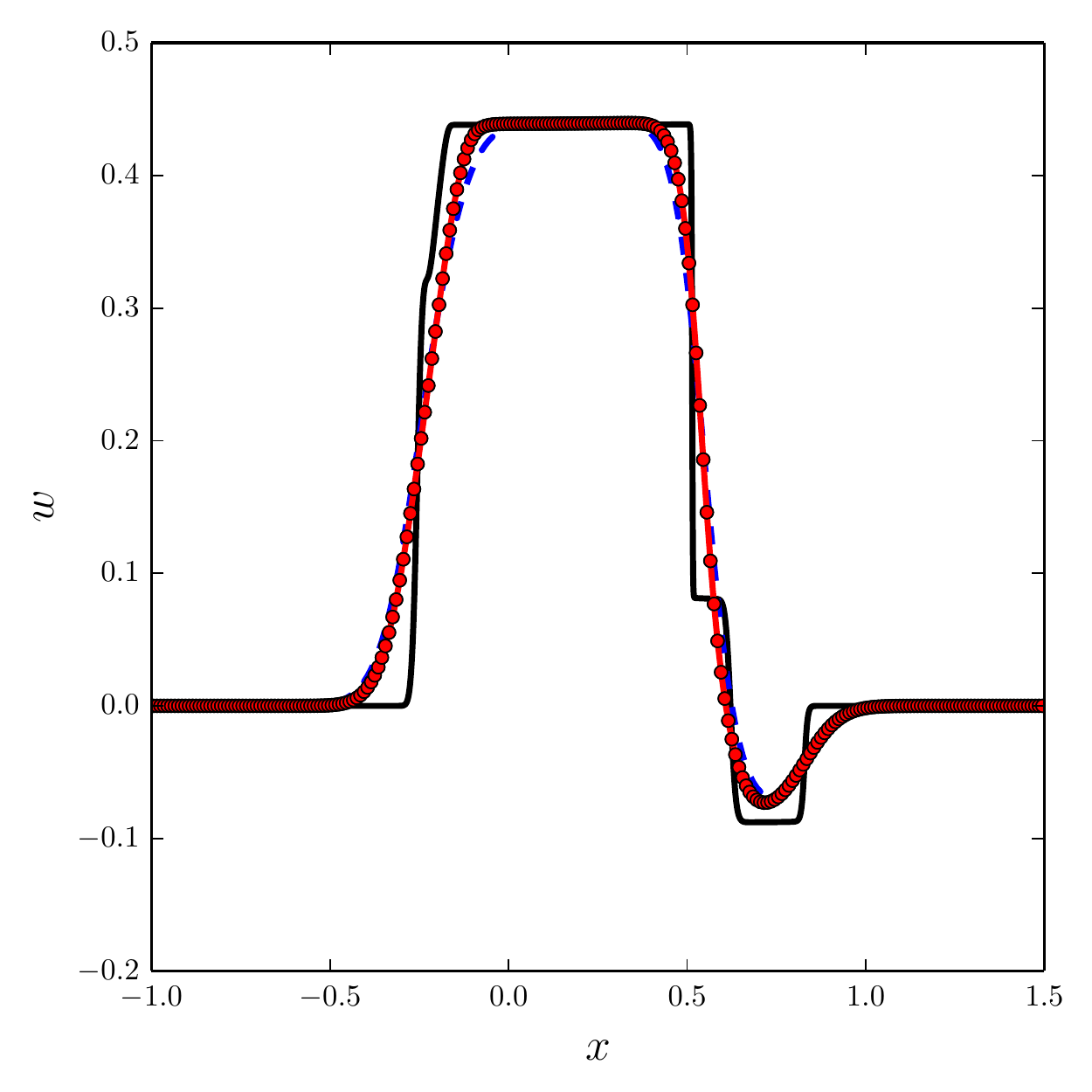}
}
\\
{
\includegraphics[scale=0.575]{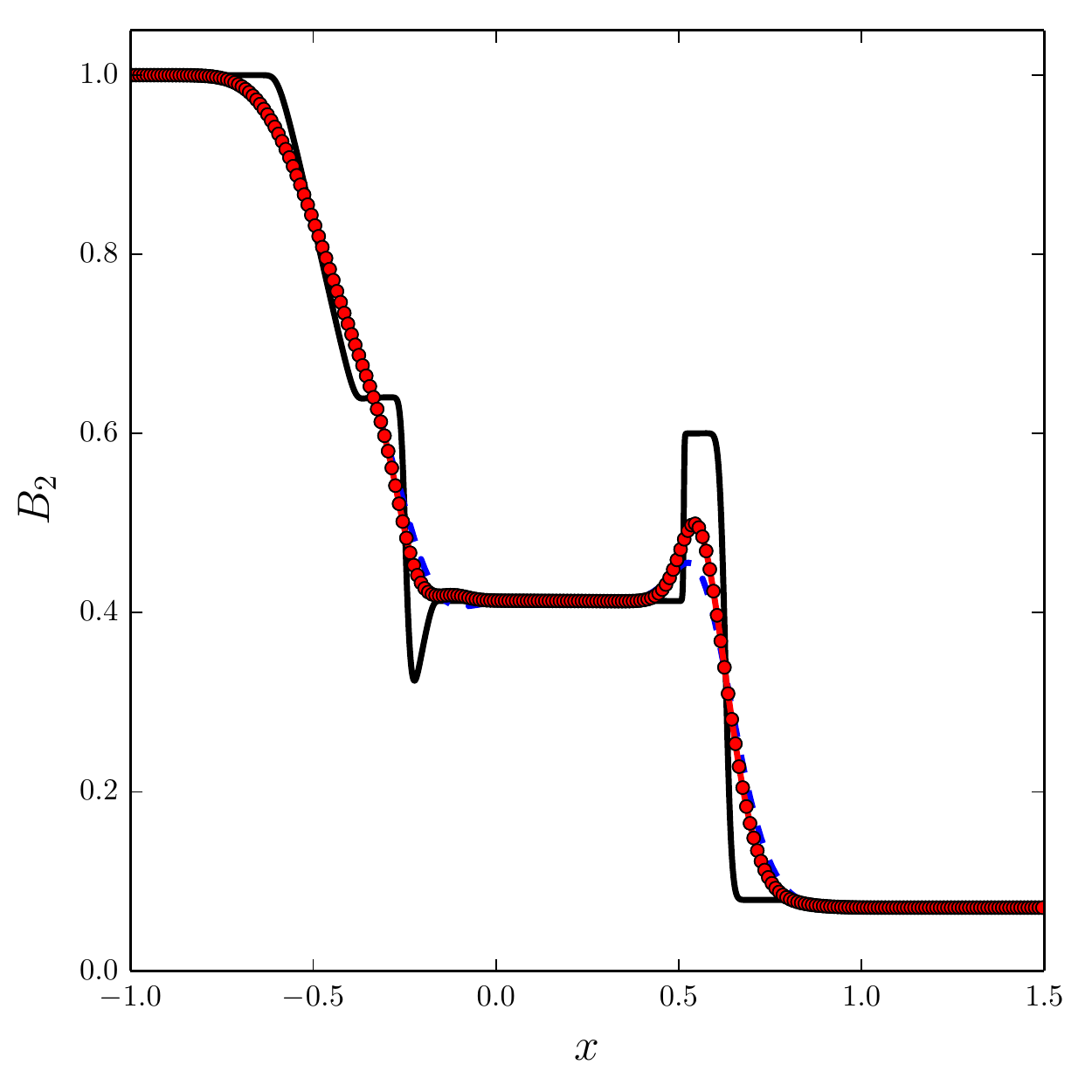}
}
{
\includegraphics[scale=0.575]{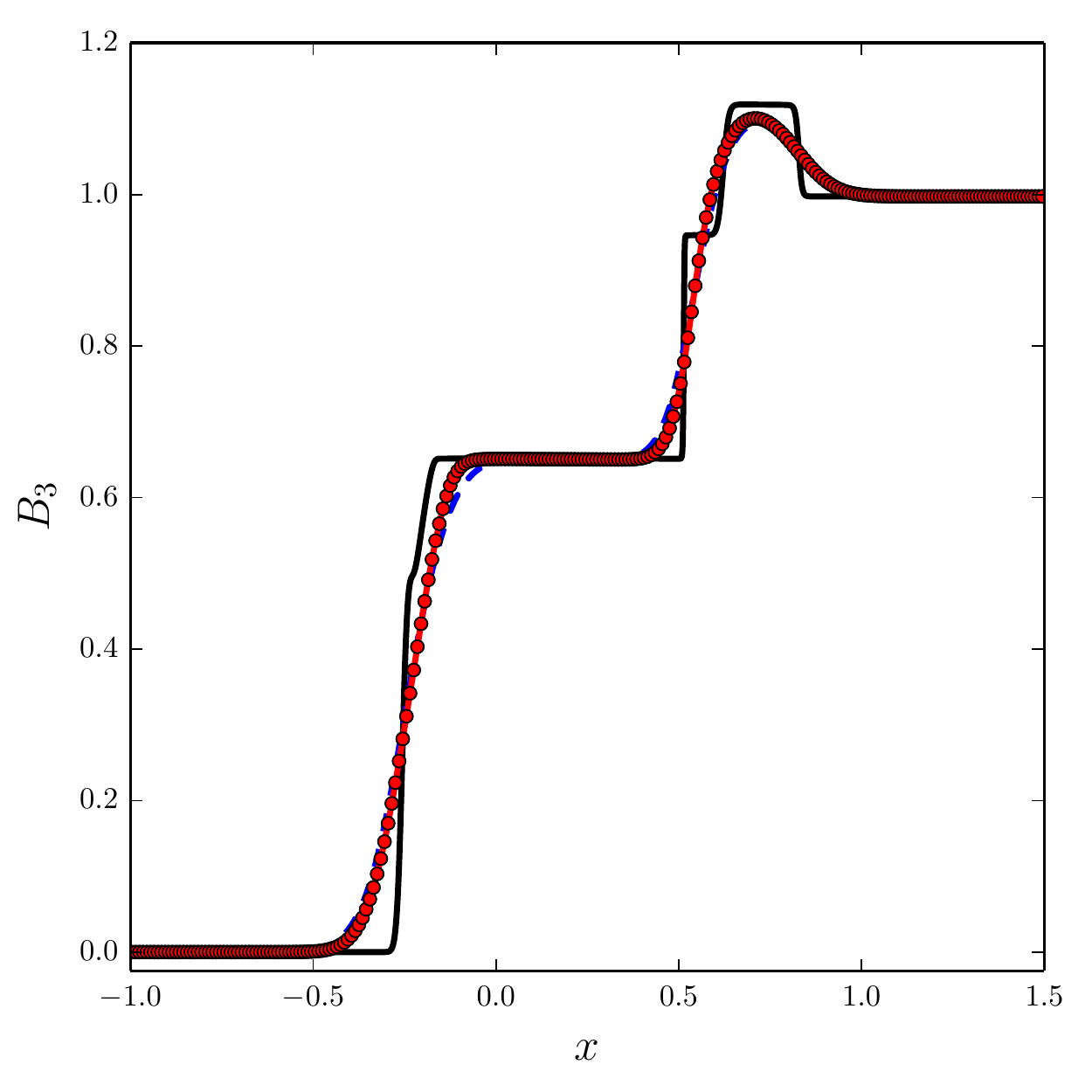}
}
\caption{The entropy stable approximations of the the density $\rho$, pressure $p$, $x$, $y$, and $z$ velocity components, and the $y$ and $z$ components of the magnetic field of the Torrilhon Riemann problem at $T = 0.4$. Solid black is the reference solution, dashed blue is the ES-LLF scheme, and red with knots is the ES-Roe scheme.}
\label{fig:ES_Tor}
\end{center}
\end{figure}

{\color{black}{
\subsection{Two Dimensional Results}\label{2DNumResults}

Next, we include two numerical examples to verify the entropy conservation of the approximation in higher dimensions as well as demonstrate the hyperbolic divergence cleaning capability of the Janhunen source term. For the two dimensional approximations we implement a first order finite volume method in space and a second order Runge-Kutta time integrator.

\subsubsection{2.5D Shock Tube: Entropy Conservation}\label{2DEntCons}

This rotated shock tube problem is denoted a 2.5 dimensional test because all three components of the velocity and magnetic fields are non-zero \cite{dai1994,ryu1995,toth2000}. The solution propagates at an angle of $45^\circ$ relative to the $x-$axis. We set periodic boundary conditions on the domain $\Omega = [0,1]\times[0,1]$ with $\gamma = \frac{5}{3}$. Just as was done in Sec. \ref{ECRiemann} this numerical example is used to verify the entropy conservative properties of the newly developed entropy conserving flux functions in two spatial dimensions. The flux in the $x-$direction is given by \eqref{Eq:entropyconservative} and the flux in the $y-$direction can be found in the appendix \eqref{Eq:entropyconservative-y}. The initial Riemann problem is given by
\begin{equation}\label{25DProblem}
\begin{aligned}
left:&\quad \begin{bmatrix}\rho, \rho u, \rho v, \rho w, p, B_1, B_2, B_3 \end{bmatrix}^T = \left[1.08,1.2,0.01,0.5,0.95,\frac{2}{\sqrt{4\pi}},\frac{2}{\sqrt{4\pi}},\frac{3.6}{\sqrt{4\pi}}\right]^T, \\ 
right:&\quad \begin{bmatrix}\rho, \rho u, \rho v, \rho w, p, B_1, B_2, B_3 \end{bmatrix}^T =\left[1,0,0,0,1,\frac{2}{\sqrt{4\pi}},\frac{4}{\sqrt{4\pi}},\frac{2}{\sqrt{4\pi}}\right]^T.
\end{aligned}
\end{equation}
where the discontinuity is along the line $x + y = 1/2$.

The approximation is calculated on a $50\times 50$ uniform spatial grid. In the Tab. \ref{tab:conservation_2D} we present the conservative properties of the scheme. We see that the mass, momentum, and total energy remain conserved quantities. The introduction of the Janhunen source term causes a loss in conservation of the magnetic field components, as expected. Finally, similar to the results in Sec. \ref{ECRiemann}, the change in the entropy demonstrates the temporal accuracy of the approximation. If we shrink the time step by a factor ten we see that the difference in the entropy shrinks by a factor of one hundred, as we expect for a second order time integrator. 
\begin{table}[!ht]
\footnotesize
\begin{center}
\begin{tabular}{lccccccccc}
\toprule
CFL & $\rho$ & $\rho u$ & $\rho v$ & $\rho w$ & $\rho e$ & $B_1$ & $B_2$ & $B_3$ & $U$ \\[0.075cm]\toprule
$1.0$ & $2.00\text{E \!-}16$ & $2.08\text{E \!-}16$ & $4.51\text{E \!-}16$ &  $2.44\text{E \!-}15$ & $8.88\text{E \!-}16$ & $3.66\text{E \!-}03$ & $3.13\text{E \!-}03$ & $4.22\text{E \!-}03$ & $6.45\text{E \!-}03$ \\\midrule
$0.1$ & $2.66\text{E \!-}16$ & $2.77\text{E \!-}16$ & $2.77\text{E \!-}16$ & $-2.22\text{E \!-}15$ & $1.11\text{E \!-}15$ & $3.66\text{E \!-}03$ & $3.13\text{E \!-}03$ & $4.22\text{E \!-}03$ & $6.71\text{E \!-}05$ \\\midrule
$0.01$ & $2.66\text{E \!-}16$ & $2.22\text{E \!-}16$ & $2.22\text{E \!-}16$ & $-2.22\text{E \!-}15$ & $1.11\text{E \!-}15$ & $3.66\text{E \!-}03$ & $3.13\text{E \!-}03$ & $4.22\text{E \!-}03$ & $6.57\text{E \!-}07$ \\\bottomrule
\end{tabular}
\end{center}
\caption{Conservation errors (integrated over the whole domain) of the entropy conserving approximation applied to the $2.5D$ shock tube problem \eqref{25DProblem} for different CFL numbers, final time $T=0.2$, and $50$ regular grid cells in each direction.}
\label{tab:conservation_2D}
\end{table}%
\subsubsection{MHD Rotor}\label{ROTOR}

The ideal MHD rotor test was first proposed by Balsara and Spicer \cite{balsara1999}, in T\'{o}th \cite{toth2000} this configuration is referred to as the {\it{first rotor problem}}. The computational domain is $\Omega = [0, 1] \times [0, 1]$ with periodic boundary conditions on each side. We take $\gamma = 1.4$ and the initial condition is given as follows. For $\sqrt{{\color{black}{(x-0.5)^2 + (y-0.5)^2}}} = r < r_0$,
\begin{equation}
\rho = 10,\quad (u,v) = \frac{u_0}{r_0}\left(-\left[y-\frac{1}{2}\right],\left[x - \frac{1}{2}\right]\right),
\end{equation}
for $r_0<r<r_1$
\begin{equation}
\rho = 1 + 9f,\quad (u,v) = \frac{fu_0}{r}\left(-\left[y-\frac{1}{2}\right],\left[x - \frac{1}{2}\right]\right),\quad f = \frac{r_1 - r}{r_1 - r_0},
\end{equation}
and for $r > r_1$
\begin{equation}
\rho = 1,\quad (u,v) = \left(0,0\right),
\end{equation}
with $r_0 = 0.1$, $r_1 = 0.115$, and $u_0 = 2$. The rest of the primitive quantities are constants given by
\begin{equation}
w = 0,\quad p = 1,\quad \vec{B}=\frac{5}{\sqrt{4\pi}}\left(1,0,0\right).
\end{equation}
The final time is ${\color{black}{T=0.15}}$ for the first rotor problem. 

The computed density, pressure, {\color{black}{M}}ach number
\begin{equation}
Ma\coloneqq \frac{\|\vec{u}\|}{a},\quad a = \sqrt{\frac{\gamma p}{\rho}},
\end{equation}
and magnetic pressure using the ES-Roe scheme are presented in Fig. \ref{fig:2DRotor}. The computation was performed on a uniform $512 \times 512$ grid with $CFL= 0.8$. Our computed results compare well to those found in the literature, e.g. \cite{balsara1999,toth2000}. The computation captures the circularly rotating velocity field in the central portion of the Mach number. 
\begin{figure}[!ht]
\begin{center}
{
\includegraphics[scale=0.2,trim=25 10 80 25, clip]{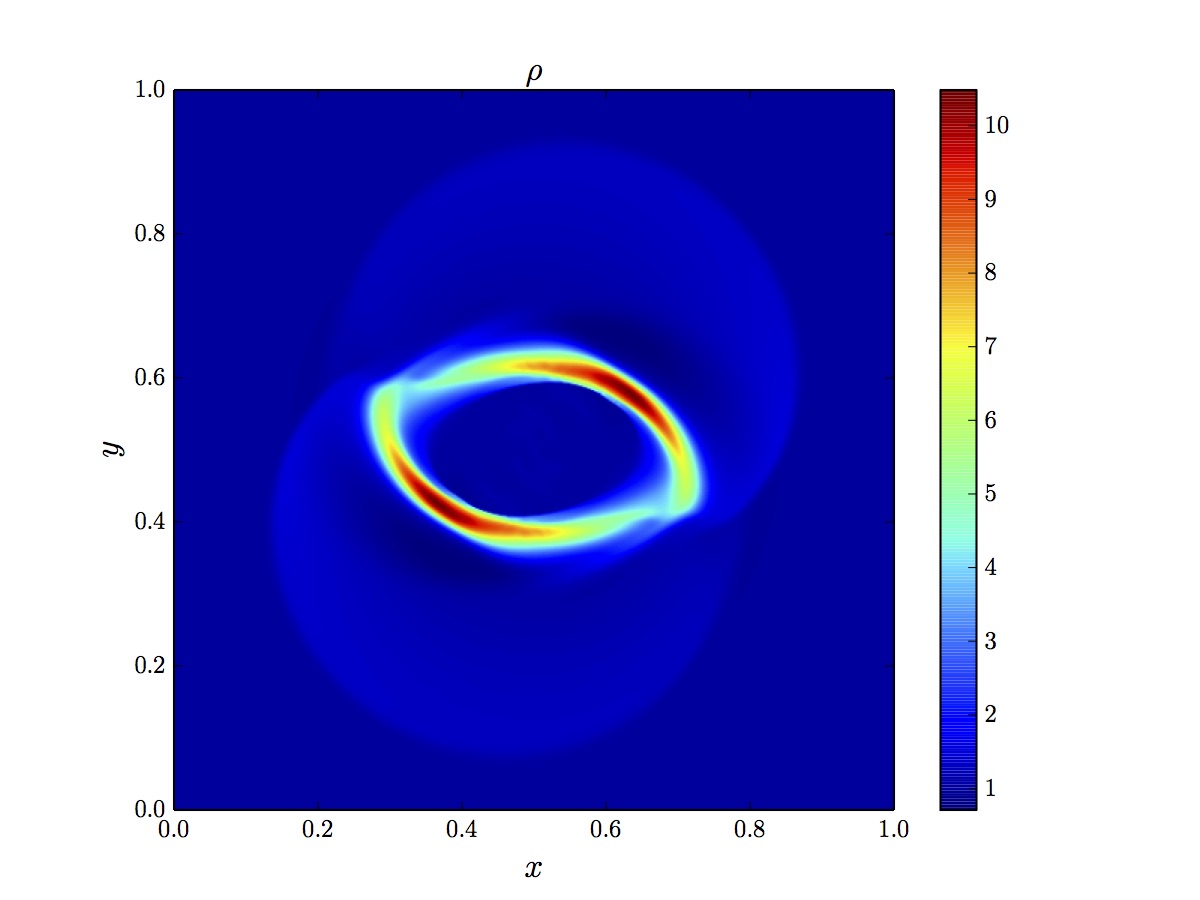}
}
{
\includegraphics[scale=0.2,trim=25 10 80 25, clip]{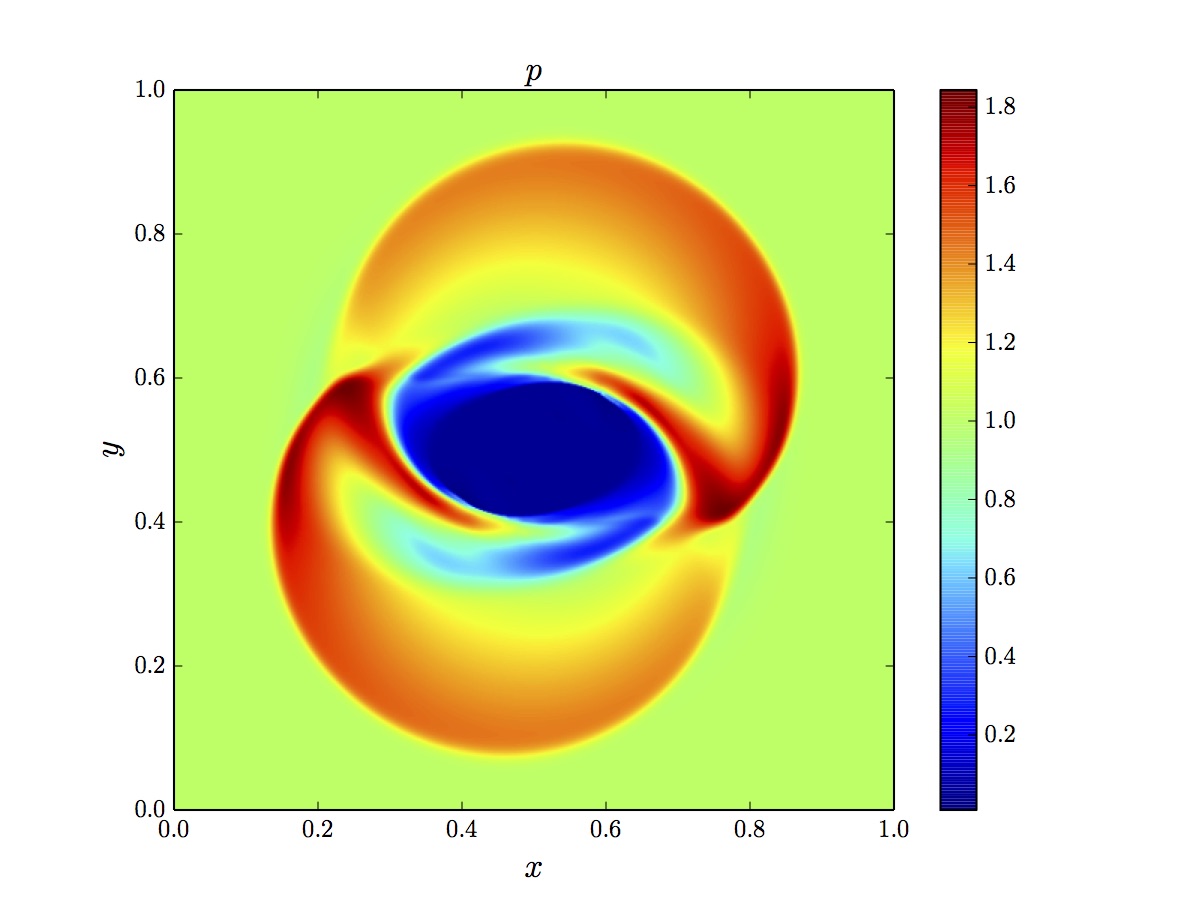}
}
\\
{
\includegraphics[scale=0.2,trim=25 10 80 20, clip]{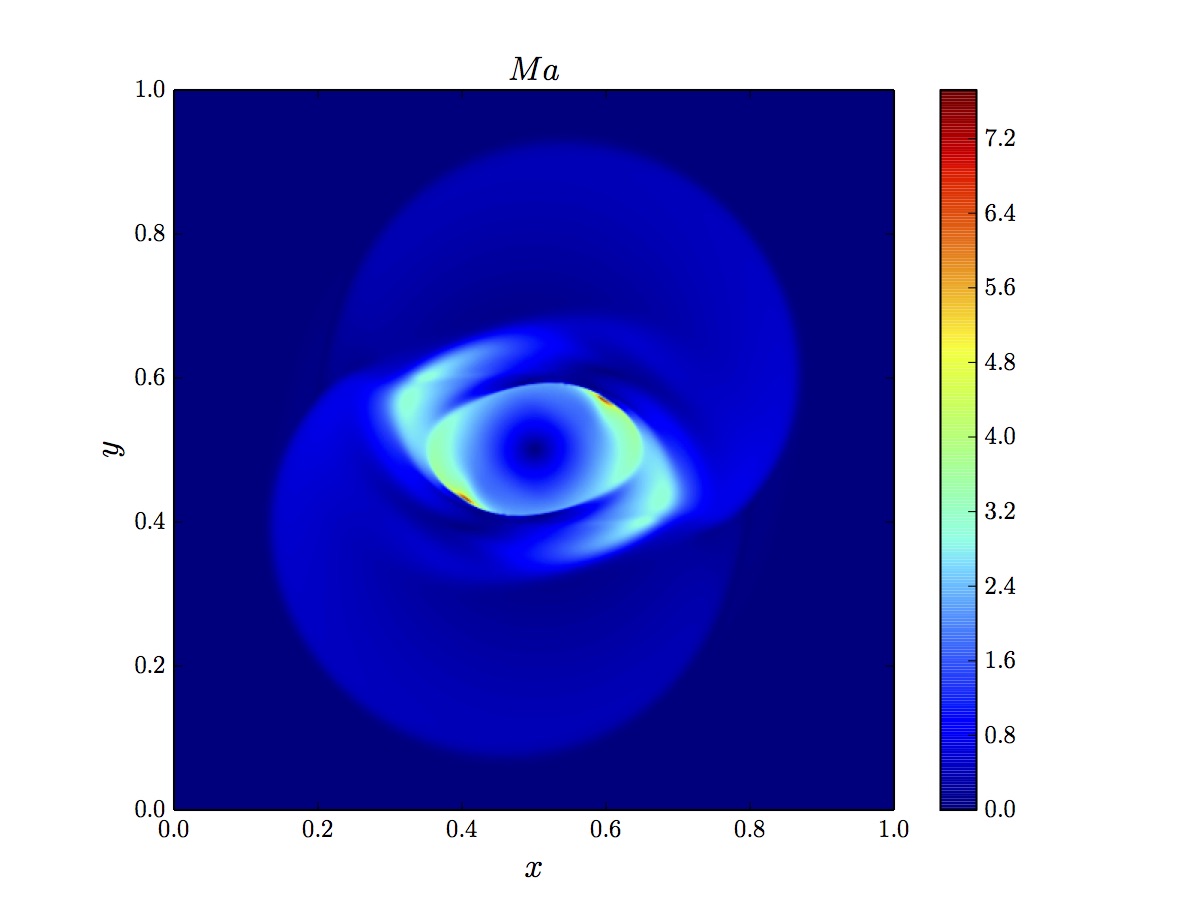}
}
{
\includegraphics[scale=0.2,trim=25 10 80 20, clip]{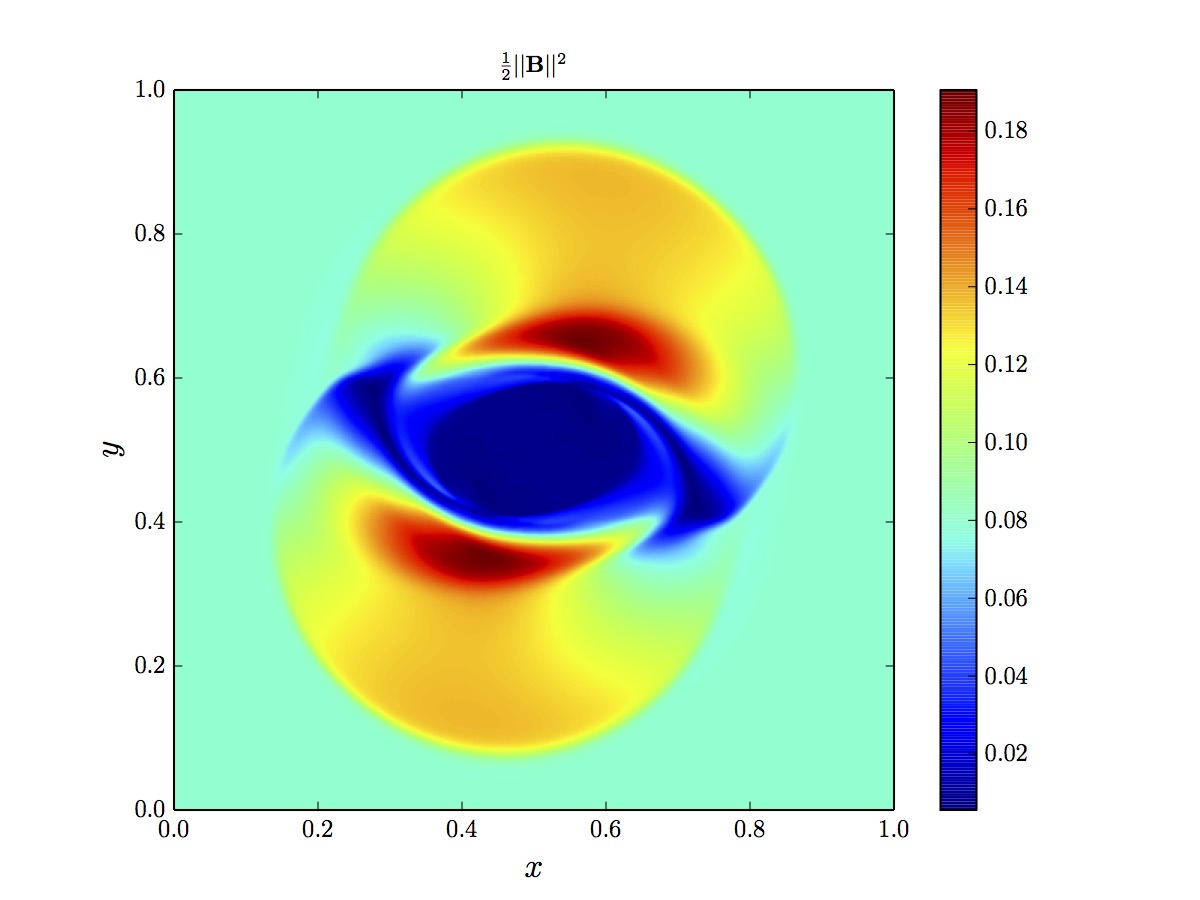}
}
\caption{The computed solution to the first rotor problem on a uniform $512\times 512$ grid using the ES-Roe scheme. We present at ${\color{black}{T=0.15}}$ the computed density, pressure, {\color{black}{M}}ach number, and magnetic pressure.}
\label{fig:2DRotor}
\end{center}
\end{figure}

We noted in Sec. \ref{GoverningEquations} that the Janhunen source term acts in an analogous fashion to a hyperbolic divergence cleaning method. We compute the discrete divergence of the computed magnetic field variables using a first order approximation of the derivative consistent with the discretization used to approximate the Janhunen source term, i.e.,
\begin{equation}\label{discretedivB}
\nabla\cdot\vec{B} \approx \frac{(B_1)_{i+1,j} - (B_1)_{i,j}}{\Delta x} +  \frac{(B_2)_{i,j+1} - (B_2)_{i,j}}{\Delta y}.
\end{equation}
We present in Fig. \ref{fig:divB} the discrete divergence. We see from the plot that we discretely recover the divergence-free condition with three digits of accuracy. It is also interesting to compare the divergence plot Fig. \ref{fig:divB} and the Mach number plot from Fig. \ref{fig:2DRotor}. We observe that the largest divergence errors are concentrated near regions of motion, as expected for a hyperbolic divergence cleaning type method.
\begin{figure}[!ht]
\begin{center}
{
\includegraphics[scale=0.5,trim=25 10 40 25, clip]{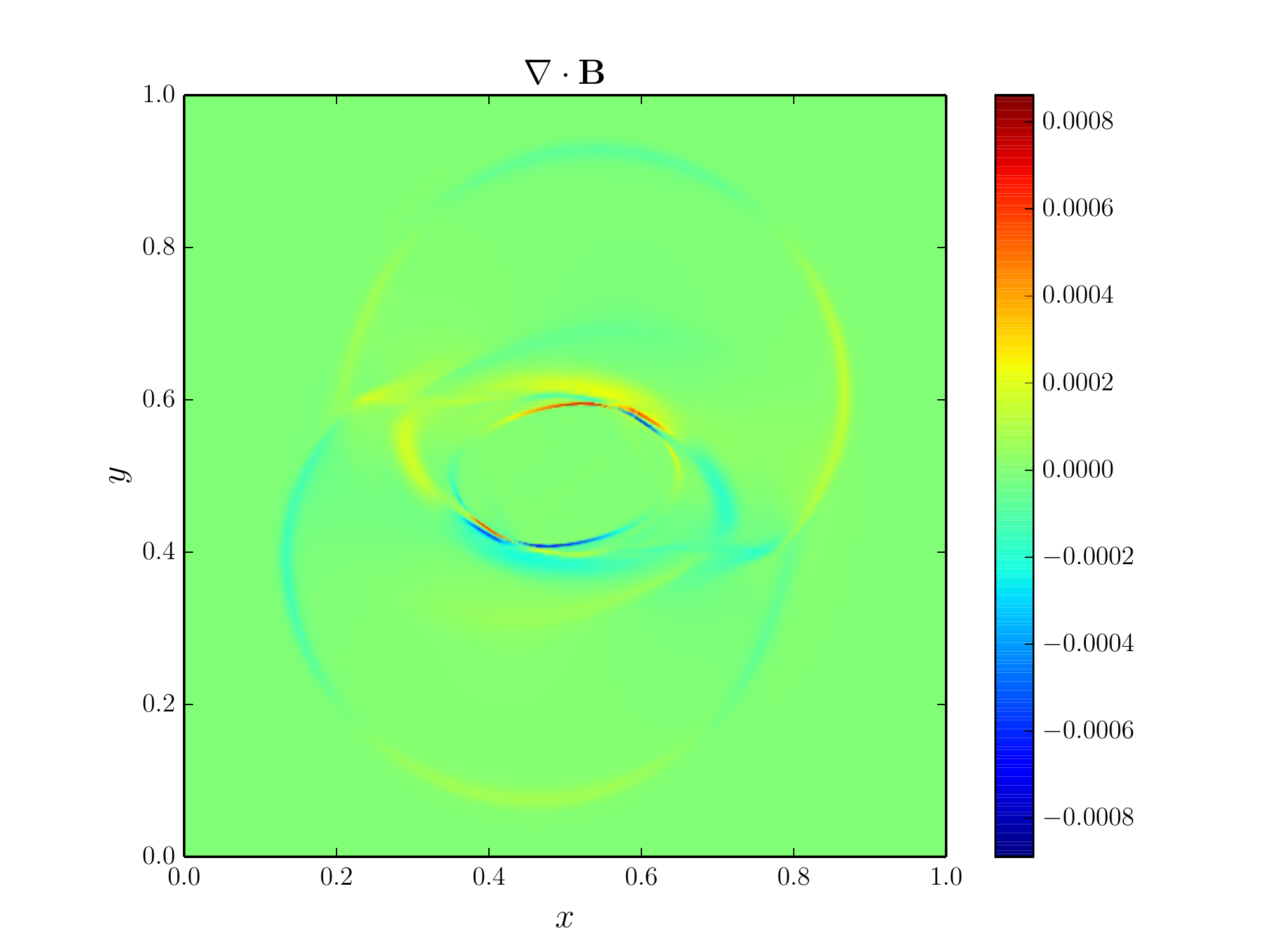}
}
\caption{The discrete divergence of $\vec{B}$ computed from \eqref{discretedivB} at T = 0.15.}
\label{fig:divB}
\end{center}
\end{figure}
}}

\section{Concluding Remarks}\label{conclusion}

In this work we present a novel, affordable, and entropy stable flux for the one-dimensional ideal MHD equations. Upon relaxing the divergence-free condition such that  $\partial_x(B_1)\approx 0$ we showed that it was possible to derive a discrete entropy conserving numerical flux, which we denote by $\vec{f}^{*,ec}$. This assumption was also important to keep the derivations and proofs contained in this paper generalizable to multi-dimensional MHD approximations. We show in the appendix of this work that it is possible, in three dimensions, to create an entropy conserving and stable flux in each Cartesian direction. The derivation of the entropy conserving flux revealed that special care had to be taken for the discretization of the {\color{black}{Janhunen}} source term in order to guarantee discrete entropy conservation. Because entropy conserving approximations can suffer breakdown at shocks we extended our analysis and derived two stabilizing dissipation terms that we add to the entropy conserving flux. We used a variety of numerical test examples from the literature to demonstrate and underscore the theoretical findings. {\color{black}{Lastly, we demonstrated that the Janhunen source acts analogously to a hyperbolic divergence cleaning technique for multidimensional computations.}}

\appendix
\section{Flux Functions in Higher Dimensions}\label{3DFluxes}

The derivation of the entropy conserving and stable numerical fluxes in this paper focused on the one-dimensional ideal MHD equations. The restriction to one spatial dimension was because the analysis proved to be quite intense. However, the discussion was kept general, so the derivations in this paper readily extend to provably entropy conserving and entropy stable approximations of ideal MHD problems on multi-dimensional Cartesian grids. {\color{black}{We note that for two and three dimensional approximations one adds the appropriate component of the Janhunen source term to each of the magnetic field equations.}}

\subsection{Entropy Conservative Fluxes in 3D}

We first note that the entropy potential in the $y-$direction is
\begin{equation}\label{EntropyPotentialy}
\phi_y = \vec{v}\cdot\vec{g} - G = \rho v + \frac{\rho v\|\vec{B}\|^2}{2p} - \frac{\rho B_2(\vec{u}\cdot\vec{B})}{p},
\end{equation}
and in the $z-$direction
\begin{equation}\label{EntropyPotentialz}
\phi_z = \vec{v}\cdot\vec{h} - H = \rho w + \frac{\rho w\|\vec{B}\|^2}{2p} - \frac{\rho B_3(\vec{u}\cdot\vec{B})}{p},
\end{equation}
where we have the entropy fluxes
\begin{equation}
G = -\frac{\rho v s}{\gamma-1},\quad H = -\frac{\rho w s}{\gamma-1}.
\end{equation}
Thus, the discrete entropy conservation condition \eqref{entropyConservationCondition1} will have the same structure in each Cartesian direction. Lastly, the Janhunen source term contributes symmetrically to each direction. With a proof analogous of that for $\vec{f}^{*,ec}$ in Sec. \ref{Sec:EntropyConservingFlux} we present the entropy conserving numerical flux for the $y$ and $z-$directions
\begin{cor}(Entropy Conserving Numerical Flux; $y-$direction)
If we introduce the parameter vector 
\begin{equation}
\vec{z} = \left[\sqrt{\frac{\rho}{p}},\sqrt{\frac{\rho}{p}}u,\sqrt{\frac{\rho}{p}}v,\sqrt{\frac{\rho}{p}}w,\sqrt{\rho p},B_1, B_2, B_3\right]^T,
\end{equation}
the averaged quantities for the primitive variables and products
\begin{equation}
\begin{aligned}
&\hat{\rho} = \average{z_1}z_5^{\ln},\;\;\hat{u}_1=\frac{\average{z_2}}{\average{z_1}},\;\; \hat{v}_1=\frac{\average{z_3}}{\average{z_1}},\;\;\hat{w}_1 = \frac{\average{z_4}}{\average{z_1}},\;\; \hat{p}_1 = \frac{\average{z_5}}{\average{z_1}},\;\;\hat{p}_2 = \frac{\gamma+1}{2\gamma}\frac{z_5^{\ln}}{z_1^{\ln}} + \frac{\gamma-1}{2\gamma}\frac{\average{z_5}}{\average{z_1}}, \\[0.1cm]
&\qquad\hat{u}_2 = \frac{\average{z_1 z_2}}{\average{z_1^2}},\;\;\hat{v}_2 = \frac{\average{z_1 z_3}}{\average{z_1^2}},\;\;\hat{w}_2 = \frac{\average{z_1 z_4}}{\average{z_1^2}},\;\;\hat{B_1} = \average{z_6},\;\;\hat{B_2} = \average{z_7},\;\;\hat{B_3} = \average{z_8}, \\[0.1cm]
&\qquad\;\;\;\accentset{\circ}{B_1} = \average{z_6^2},\;\;\accentset{\circ}{B_2} = \average{z_7^2},\;\;\accentset{\circ}{B_3} = \average{z_8^2},\;\;\widehat{B_1 B_2} = \average{z_6 z_7},\;\;\widehat{B_2B_3} = \average{z_7 z_8},
\end{aligned}
\end{equation}
and discretize the source term in the finite volume method to contribute to each element as
\begin{equation}\label{SourceTermDiscy}
\vec{s}_{ijk} ={\color{black}{\frac{1}{2}}}\left( \vec{s}_{i,j+\tfrac{1}{2},k} + \vec{s}_{i,j-\tfrac{1}{2},k}\right) = -\frac{1}{2}\left(
\jump{B_2}_{j+\tfrac{1}{2}}\begin{bmatrix} 
0\\
0\\
0\\
0\\
0\\
\frac{\average{z_1 z_2}\hat{B_1}}{\average{\Delta y z_1^2 B_1}}\\[0.15cm]
\frac{\average{z_1 z_3}\hat{B_2}}{\average{\Delta y z_1^2 B_2}}\\[0.15cm]
\frac{\average{z_1 z_4}\hat{B_3}}{\average{\Delta y z_1^2 B_3}}
\end{bmatrix}_{j+\tfrac{1}{2}}
+
\jump{B_2}_{j-\tfrac{1}{2}}\begin{bmatrix} 
0\\
0\\
0\\
0\\
0\\
\frac{\average{z_1 z_2}\hat{B_1}}{\average{\Delta y z_1^2 B_1}}\\[0.15cm]
\frac{\average{z_1 z_3}\hat{B_2}}{\average{\Delta y z_1^2 B_2}}\\[0.15cm]
\frac{\average{z_1 z_4}\hat{B_3}}{\average{\Delta y z_1^2 B_3}}
\end{bmatrix}_{j-\tfrac{1}{2}}
\right),
\end{equation}
{\color{black}{where we have included additional indices on the source term for clarity,}}
then we can determine a discrete, entropy conservative flux to be
\begin{equation}\label{Eq:entropyconservative-y}
\vec{g}^{*,ec} = \begin{bmatrix}
\hat{\rho}\hat{v}_1 \\
\hat{\rho}\hat{u}_1\hat{v}_1 - \widehat{B_1B_2} \\
\hat{p}_1 + \hat{\rho}\hat{v}^2_1  + \frac{1}{2}\left(\accentset{\circ}{B_1}+\accentset{\circ}{B_2}+\accentset{\circ}{B_3}\right) - \accentset{\circ}{B_2} \\ 
\hat{\rho}\hat{v}_1\hat{w}_1 -\widehat{B_2 B_3} \\
\frac{\gamma \hat{v}_1\hat{p}_2}{\gamma -1} + \frac{\hat{\rho}\hat{v}_1}{2}(\hat{u}_1^2 + \hat{v}_1^2 + \hat{w}_1^2) + \hat{v}_2\hat{B}_1^2 - \hat{u}_2\hat{B_1}\hat{B_2} + \hat{v}_2\hat{B}_3^2-\hat{w}_2\hat{B_2}\hat{B_3} \\
\hat{v}_2\hat{B_1} - \hat{u}_2\hat{B_2} \\
0 \\
\hat{v}_2\hat{B_3} - \hat{w}_2\hat{B_2}
\end{bmatrix}.
\end{equation}
\end{cor}

\begin{cor}(Entropy Conserving Numerical Flux; $z-$direction)
If we introduce the parameter vector 
\begin{equation}
\vec{z} = \left[\sqrt{\frac{\rho}{p}},\sqrt{\frac{\rho}{p}}u,\sqrt{\frac{\rho}{p}}v,\sqrt{\frac{\rho}{p}}w,\sqrt{\rho p},B_1, B_2, B_3\right]^T,
\end{equation}
the averaged quantities for the primitive variables and products
\begin{equation}
\begin{aligned}
&\hat{\rho} = \average{z_1}z_5^{\ln},\;\;\hat{u}_1=\frac{\average{z_2}}{\average{z_1}},\;\; \hat{v}_1=\frac{\average{z_3}}{\average{z_1}},\;\;\hat{w}_1 = \frac{\average{z_4}}{\average{z_1}},\;\; \hat{p}_1 = \frac{\average{z_5}}{\average{z_1}},\;\;\hat{p}_2 = \frac{\gamma+1}{2\gamma}\frac{z_5^{\ln}}{z_1^{\ln}} + \frac{\gamma-1}{2\gamma}\frac{\average{z_5}}{\average{z_1}}, \\[0.1cm]
&\qquad\hat{u}_2 = \frac{\average{z_1 z_2}}{\average{z_1^2}},\;\;\hat{v}_2 = \frac{\average{z_1 z_3}}{\average{z_1^2}},\;\;\hat{w}_2 = \frac{\average{z_1 z_4}}{\average{z_1^2}},\;\;\hat{B_1} = \average{z_6},\;\;\hat{B_2} = \average{z_7},\;\;\hat{B_3} = \average{z_8}, \\[0.1cm]
&\qquad\;\;\;\accentset{\circ}{B_1} = \average{z_6^2},\;\;\accentset{\circ}{B_2} = \average{z_7^2},\;\;\accentset{\circ}{B_3} = \average{z_8^2},\;\;\widehat{B_1 B_3} = \average{z_6 z_8},\;\;\widehat{B_2B_3} = \average{z_7 z_8},
\end{aligned}
\end{equation}
and discretize the source term in the finite volume method to contribute to each element as
\begin{equation}\label{SourceTermDiscz}
\vec{s}_{ijk} ={\color{black}{\frac{1}{2}}}\left(\vec{s}_{i,j,k+\tfrac{1}{2}} + \vec{s}_{i,j,k-\tfrac{1}{2}}\right) = -\frac{1}{2}\left(
\jump{B_3}_{k+\tfrac{1}{2}}\begin{bmatrix} 
0\\
0\\
0\\
0\\
0\\
\frac{\average{z_1 z_2}\hat{B_1}}{\average{\Delta z z_1^2 B_1}}\\[0.15cm]
\frac{\average{z_1 z_3}\hat{B_2}}{\average{\Delta z z_1^2 B_2}}\\[0.15cm]
\frac{\average{z_1 z_4}\hat{B_3}}{\average{\Delta z z_1^2 B_3}}
\end{bmatrix}_{k+\tfrac{1}{2}}
+
\jump{B_3}_{k-\tfrac{1}{2}}\begin{bmatrix} 
0\\
0\\
0\\
0\\
0\\
\frac{\average{z_1 z_2}\hat{B_1}}{\average{\Delta z z_1^2 B_1}}\\[0.15cm]
\frac{\average{z_1 z_3}\hat{B_2}}{\average{\Delta z z_1^2 B_2}}\\[0.15cm]
\frac{\average{z_1 z_4}\hat{B_3}}{\average{\Delta z z_1^2 B_3}}
\end{bmatrix}_{k-\tfrac{1}{2}}
\right),
\end{equation}
{\color{black}{where we have included additional indices on the source term for clarity,}} then we can determine a discrete, entropy conservative flux to be
\begin{equation}
\vec{h}^{*,ec} = \begin{bmatrix}
\hat{\rho}\hat{w}_1 \\
\hat{\rho}\hat{u}_1\hat{w}_1 - \widehat{B_1 B_3} \\ 
\hat{\rho}\hat{v}_1\hat{w}_1 -\widehat{B_2 B_3} \\
\hat{p}_1 + \hat{\rho}\hat{w}_1^2 + \frac{1}{2}\left(\accentset{\circ}{B_1}+\accentset{\circ}{B_2}+\accentset{\circ}{B_3}\right) - \accentset{\circ}{B_3} \\
\frac{\gamma \hat{w}_1\hat{p}_2}{\gamma -1} + \frac{\hat{\rho}\hat{w}_1}{2}(\hat{u}_1^2 + \hat{v}_1^2 + \hat{w}_1^2) + \hat{w}_2\hat{B}_1^2 - \hat{u}_2\hat{B_1}\hat{B_3} + \hat{w}_2\hat{B}_2^2-\hat{v}_2\hat{B_2}\hat{B_3} \\
\hat{w}_2\hat{B_1} - \hat{u}_2\hat{B_3} \\
\hat{w}_2\hat{B_2} - \hat{v}_2\hat{B_3} \\
0
\end{bmatrix}.
\end{equation}
\end{cor}

\subsection{Entropy Stable Fluxes in 3D}

Just as was done in Sec. \ref{Sec:StableFlux} we can create 3D entropy stable flux functions. This requires the eigenstructure of the flux Jacobian matrices in the $y$ and $z$ directions augmented by the Powell source term. We denote the altered flux matrix in the $y-$direction by
\begin{equation}
\widehat{B} = \vec{g}_{\vec{q}} + \matrix{P}_y,
\end{equation}
and the new $z-$direction flux Jacobian matrix
\begin{equation}
\widehat{C} = \vec{h}_{\vec{q}} + \matrix{P}_z,
\end{equation}
where $\vec{g}$, $\vec{h}$ are the physical fluxes in $y$, $z$ and $\matrix{P}_y$, $\matrix{P}_z$ are the Powell matrix \eqref{PowellMatrix} with the non-zero column shifted to the seventh or eighth column respectively.  

First, we describe the entropy stable fluxes in the $y-$ direction. To do so we require the eigendecomposition of the matrix $\widehat{\matrix{B}}$. For convenience we denote the matrix of right eigenvectors $\widehat{\matrix{R}}_y$ with columns given by the vectors in the following order
\begin{equation}\label{rightEVy}
\widehat{\matrix{R}}_y = \Big[ \hat{\,\vec{r}}_{-f} \big|\,\hat{\vec{r}}_{-a} \big|\,\hat{\vec{r}}_{-s} \big|\,\hat{\vec{r}}_{E} \big|\,\hat{\vec{r}}_{D} \big|\,\hat{\vec{r}}_{+s} \big|\,\hat{\vec{r}}_{+a} \big|\,\hat{\vec{r}}_{+f} \Big],
\end{equation}
and $\widehat{\matrix{L}}_y = \widehat{\matrix{R}}_y^{-1}$.
\begin{itemize}
\item[] \underline{Entropy and Divergence Waves}: $\lambda_{E,D} = v$
\begin{equation}
\widehat{\vec{r}}_E = \begin{bmatrix} 1 \\ u \\v \\w \\ \frac{\|\vec{u}\|^2}{2} \\0 \\0 \\0 \end{bmatrix},\quad\widehat{\vec{r}}_D = \begin{bmatrix} 0 \\ 0 \\0 \\0 \\\sqrt{\rho}\,b_2 \\0 \\1 \\0 \end{bmatrix},
\end{equation}
\item[] \underline{Alfv\'{e}n Waves}: $\lambda_{\pm a} = v\pm b_2$ 
\begin{equation}
\widehat{\vec{r}}_{\pm a} = \begin{bmatrix}
0 \\
\pm \rho^{\frac{3}{2}}\,b_3 \\
0 \\
\mp \rho^{\frac{3}{2}}\,b_1 \\
\pm \rho^{\frac{3}{2}}(b_3 u - b_1 w) \\
\rho b_3 \\
0 \\
-\rho b_1
\end{bmatrix},
\end{equation}
\item[] \underline{Magnetoacoustic Waves}: $\lambda_{\pm f,\pm s} = v\pm c_{f,s}$ 
\begin{equation}
\widehat{\vec{r}}_{\pm f} = \begin{bmatrix}
\alpha_f\rho \\[0.1cm]
\rho\left(\alpha_f u \mp \alpha_s c_s \beta_1 sgn(b_2) \right) \\[0.1cm]
\alpha_f\rho(v \pm c_{f}) \\[0.1cm]
\rho\left(\alpha_f w \mp \alpha_s c_s \beta_3 sgn(b_2) \right) \\[0.1cm]
\Psi_{\pm f} \\[0.1cm]
\alpha_s a \beta_1 \sqrt{\rho} \\[0.1cm]
0 \\[0.1cm]
\alpha_s a \beta_3 \sqrt{\rho} 
\end{bmatrix},
\qquad 
\widehat{\vec{r}}_{\pm s} = \begin{bmatrix}
\alpha_s\rho \\[0.1cm]
\rho\left(\alpha_s u \pm \alpha_f c_f \beta_1 sgn(b_2)\right) \\[0.1cm]
\alpha_s\rho\left(v \pm c_s\right) \\[0.1cm]
\rho\left(\alpha_s w \pm \alpha_f c_f \beta_3 sgn(b_2)\right) \\[0.1cm]
\Psi_{\pm s} \\[0.1cm]
-\alpha_f a \beta_1 \sqrt{\rho} \\[0.1cm]
0 \\[0.1cm]
-\alpha_f a \beta_3 \sqrt{\rho} 
\end{bmatrix},
\end{equation}
where $b_{\perp}^2 = b_1^2 + b_3^2 $ and we introduce the auxiliary variables
\begin{equation}
\begin{aligned}
\Psi_{\pm f} &= \frac{\alpha_f\rho}{2}\|\vec{u}\|^2 + a\alpha_s\rho b_{\perp} + \frac{\alpha_f\rho a^2}{\gamma - 1} \pm \alpha_f c_f \rho v \mp \alpha_s c_s \rho\,sgn(b_2)\left(u\beta_1 + w\beta_3\right), \\ 
\Psi_{\pm s} &= \frac{\alpha_s\rho}{2}\|\vec{u}\|^2 - a\alpha_f\rho b_{\perp} + \frac{\alpha_s\rho a^2}{\gamma-1} \pm \alpha_s c_s \rho v \pm \alpha_f c_f \rho\,sgn(b_2)\left(u\beta_1 + w\beta_3\right).
\end{aligned}
\end{equation}
\end{itemize}

\begin{cor} (Entropy Stable - Roe Type Stabilization (ES-Roe); $y-$direction) If we apply the diagonal scaling matrix
\begin{equation}
\matrix{T} = diag\left(\frac{1}{\sqrt{2\rho\gamma}},\,\sqrt{\frac{p}{2\rho^3b_{\perp}^2}},\,\frac{1}{\sqrt{2\rho\gamma}},\,\sqrt{\frac{\rho(\gamma-1)}{\gamma}},\,\sqrt{\frac{p}{\rho}},\,\frac{1}{\sqrt{2\rho\gamma}},\,\sqrt{\frac{p}{2\rho^3b_{\perp}^2}},\,\frac{1}{\sqrt{2\rho\gamma}}\right),
\end{equation}
to the matrix of right eigenvectors $\widehat{\matrix{R}}_y$ \eqref{rightEVy}, then we obtain the Merriam identity \cite{merriam1989} (Eq. 7.3.1 pg. 77) 
\begin{equation}
\matrix{H} = \widetilde{\matrix{R}}_y\widetilde{\matrix{R}}^T_y = \left(\widehat{\matrix{R}}_y\matrix{T}\right) \left(\widehat{\matrix{R}}_y\matrix{T}\right)^T = \widehat{\matrix{R}}_y\matrix{S}\widehat{\matrix{R}}_y^T,
\end{equation}
that relates the right eigenvectors of $\widehat{\matrix{B}}$ to the entropy Jacobian matrix \eqref{entropyJacobian}. For convenience, we introduce the diagonal scaling matrix $\matrix{S}=\matrix{T}\,^2$. We then have the guaranteed entropy stable flux interface contribution
\begin{equation}
\vec{g}^{*,ES\textrm{-}Roe} = \vec{g}^{*,ec} - \frac{1}{2} \widehat{\matrix{R}}_y|\widehat{\boldsymbol\Lambda}|\matrix{S}\widehat{\matrix{R}}_y^T\jump{\vec{v}}.
\end{equation}
\end{cor}

\begin{rem} \textit{(Entropy Stable - Local Lax Friedrichs Type Stabilization (ES-LLF); $y-$direction)}
If we choose the dissipation matrix to be
\begin{equation}
\matrix{D} = |\lambda_{max}|\matrix{I},
\end{equation}
where $\lambda_{max}$ is the largest eigenvalue of the system from $\widehat{\matrix{B}}$ and $\matrix{I}$ is the identity matrix, then we obtain a local Lax-Friedrichs type interface stabilization
\begin{equation}\label{LFDissy}
\begin{aligned}
\vec{g}^{*,ES\textrm{-}LLF} &= \vec{g}^{*,ec} - \frac{1}{2}|\lambda_{max}|\matrix{H}\jump{\vec{v}}.
\end{aligned}
\end{equation} 
\end{rem}

Finally, we describe the entropy stable fluxes in the $z-$ direction. To do so we require the eigendecomposition of the matrix $\widehat{C}$. For convenience we denote the matrix of right eigenvectors $\widehat{\matrix{R}}_z$ with columns given by the vectors in the following order
\begin{equation}\label{rightEVz}
\widehat{\matrix{R}}_z = \Big[ \hat{\,\vec{r}}_{-f} \big|\,\hat{\vec{r}}_{-a} \big|\,\hat{\vec{r}}_{-s} \big|\,\hat{\vec{r}}_{E} \big|\,\hat{\vec{r}}_{D} \big|\,\hat{\vec{r}}_{+s} \big|\,\hat{\vec{r}}_{+a} \big|\,\hat{\vec{r}}_{+f} \Big],
\end{equation}
and $\widehat{\matrix{L}}_z = \widehat{\matrix{R}}_z^{-1}$.
\begin{itemize}
\item[] \underline{Entropy and Divergence Waves}: $\lambda_{E,D} = w$
\begin{equation}
\widehat{\vec{r}}_E = \begin{bmatrix} 1 \\ u \\v \\w \\ \frac{\|\vec{u}\|^2}{2} \\0 \\0 \\0 \end{bmatrix},\quad\widehat{\vec{r}}_D = \begin{bmatrix} 0 \\ 0 \\0 \\0 \\\sqrt{\rho}\,b_3 \\0 \\0 \\1 \end{bmatrix},
\end{equation}
\item[] \underline{Alfv\'{e}n Waves}: $\lambda_{\pm a} = w\pm b_3$ 
\begin{equation}
\widehat{\vec{r}}_{\pm a} = \begin{bmatrix}
0 \\
\mp \rho^{\frac{3}{2}}\,b_2 \\
\pm \rho^{\frac{3}{2}}\,b_1 \\
0 \\
\mp \rho^{\frac{3}{2}}(b_2 u - b_1 v) \\
-\rho b_2 \\
\rho b_1\\
0
\end{bmatrix},
\end{equation}
\item[] \underline{Magnetoacoustic Waves}: $\lambda_{\pm f,\pm s} = w\pm c_{f,s}$ 
\begin{equation}
\widehat{\vec{r}}_{\pm f} = \begin{bmatrix}
\alpha_f\rho \\[0.1cm]
\rho\left(\alpha_f u \mp \alpha_s c_s \beta_1 sgn(b_3) \right) \\[0.1cm]
\rho\left(\alpha_f v \mp \alpha_s c_s \beta_2 sgn(b_3) \right) \\[0.1cm]
\alpha_f\rho(w \pm c_{f}) \\[0.1cm]
\Psi_{\pm f} \\[0.1cm]
\alpha_s a \beta_1 \sqrt{\rho} \\[0.1cm]
\alpha_s a \beta_2 \sqrt{\rho} \\[0.1cm]
0
\end{bmatrix},
\qquad 
\widehat{\vec{r}}_{\pm s} = \begin{bmatrix}
\alpha_s\rho \\[0.1cm]
\rho\left(\alpha_s u \pm \alpha_f c_f \beta_1 sgn(b_3)\right) \\[0.1cm]
\rho\left(\alpha_s v \pm \alpha_f c_f \beta_2 sgn(b_3)\right) \\[0.1cm]
\alpha_s\rho\left(w \pm c_s\right) \\[0.1cm]
\Psi_{\pm s} \\[0.1cm]
-\alpha_f a \beta_1 \sqrt{\rho} \\[0.1cm]
-\alpha_f a \beta_2 \sqrt{\rho}  \\[0.1cm]
0
\end{bmatrix},
\end{equation}
where $b_{\perp}^2 = b_1^2 + b_2^2 $ and we introduce the auxiliary variables
\begin{equation}
\begin{aligned}
\Psi_{\pm f} &= \frac{\alpha_f\rho}{2}\|\vec{u}\|^2 + a\alpha_s\rho b_{\perp} + \frac{\alpha_f\rho a^2}{\gamma - 1} \pm \alpha_f c_f \rho w \mp \alpha_s c_s \rho\,sgn(b_3)\left(u\beta_1 + v\beta_2\right), \\ 
\Psi_{\pm s} &= \frac{\alpha_s\rho}{2}\|\vec{u}\|^2 - a\alpha_f\rho b_{\perp} + \frac{\alpha_s\rho a^2}{\gamma-1} \pm \alpha_s c_s \rho w \pm \alpha_f c_f \rho\,sgn(b_3)\left(u\beta_1 + v\beta_2\right).
\end{aligned}
\end{equation}
\end{itemize}

\begin{cor} (Entropy Stable - Roe Type Stabilization (ES\textrm{-}Roe); $z-$direction) If we apply the diagonal scaling matrix
\begin{equation}
\matrix{T} = diag\left(\frac{1}{\sqrt{2\rho\gamma}},\,\sqrt{\frac{p}{2\rho^3b_{\perp}^2}},\,\frac{1}{\sqrt{2\rho\gamma}},\,\sqrt{\frac{\rho(\gamma-1)}{\gamma}},\,\sqrt{\frac{p}{\rho}},\,\frac{1}{\sqrt{2\rho\gamma}},\,\sqrt{\frac{p}{2\rho^3b_{\perp}^2}},\,\frac{1}{\sqrt{2\rho\gamma}}\right),
\end{equation}
to the matrix of right eigenvectors $\widehat{\matrix{R}}_z$ \eqref{rightEVz}, then we obtain the Merriam identity \cite{merriam1989} (Eq. 7.3.1 pg. 77) 
\begin{equation}
\matrix{H} = \widetilde{\matrix{R}}_z\widetilde{\matrix{R}}^T_z = \left(\widehat{\matrix{R}}_z\matrix{T}\right) \left(\widehat{\matrix{R}}_z\matrix{T}\right)^T = \widehat{\matrix{R}}_z\matrix{S}\widehat{\matrix{R}}_z^T,
\end{equation}
that relates the right eigenvectors of $\widehat{\matrix{B}}$ to the entropy Jacobian matrix \eqref{entropyJacobian}. For convenience, we introduce the diagonal scaling matrix $\matrix{S}=\matrix{T}\,^2$. We then have the guaranteed entropy stable flux interface contribution
\begin{equation}
\vec{h}^{*,ES\textrm{-}Roe} = \vec{h}^{*,ec} - \frac{1}{2} \widehat{\matrix{R}}_z|\widehat{\boldsymbol\Lambda}|\matrix{S}\widehat{\matrix{R}}_z^T\jump{\vec{v}}.
\end{equation}
\end{cor}

\begin{rem} \textit{(Entropy Stable - Local Lax-Friedrichs Type Stabilization (ES-LLF); $z-$direction)}
If we choose the dissipation matrix to be
\begin{equation}
\matrix{D} = |\lambda_{max}|\matrix{I},
\end{equation}
where $\lambda_{max}$ is the largest eigenvalue of the matrix $\widehat{\matrix{C}}$ and $\matrix{I}$ is the identity matrix, then we obtain a local Lax-Friedrichs type interface stabilization
\begin{equation}\label{LFDissz}
\begin{aligned}
\vec{h}^{*,ES\textrm{-}LLF} &= \vec{h}^{*,ec} - \frac{1}{2}|\lambda_{max}|\matrix{H}\jump{\vec{v}}.
\end{aligned}
\end{equation} 
\end{rem}

{\color{black}{\section{Entropy and Kinetic Energy Conserving Numerical Flux}\label{EKEP}

Inspired by Chandrashekar \cite{chandrashekar2013} we develop an alternate baseline numerical flux function that is both entropy and kinetic energy conservative. We explicitly derive the flux in the $x-$direction, but generalization to higher dimensions is straightforward through symmetry arguments as can be seen in \ref{3DFluxes}. To do so we introduce notation for the inverse of the temperature 
\begin{equation}\label{tempInverse}
\beta = \frac{1}{RT} = \frac{\rho}{2 p}.
\end{equation}
\begin{cor}[Entropy and Kinetic Energy Conserving Numerical Flux (EKEC)]
If we define the logarithmic mean $(\cdot)^{\ln}$ \eqref{logMean}, the arithmetic mean $\average{\cdot}$, and discretize the source term in the finite volume method to contribute to each element as
\begin{equation}\label{SourceTermDiscyEKEP}
\vec{s}_{i} = \frac{1}{2}\left(\vec{s}_{i+\frac{1}{2}} + \vec{s}_{i-\frac{1}{2}} \right)= -\frac{1}{2}\left(
\jump{B_1}_{i+\tfrac{1}{2}}\begin{bmatrix} 
0\\
0\\
0\\
0\\
0\\
\frac{\average{\beta u}\average{B_1}}{\average{\Delta x \beta B_1}}\\[0.15cm]
\frac{\average{\beta v}\average{B_2}}{\average{\Delta x \beta B_2}}\\[0.15cm]
\frac{\average{\beta w}\average{B_3}}{\average{\Delta x \beta B_3}}
\end{bmatrix}_{i+\tfrac{1}{2}}
+
\jump{B_1}_{i-\tfrac{1}{2}}\begin{bmatrix} 
0\\
0\\
0\\
0\\
0\\
\frac{\average{\beta u}\average{B_1}}{\average{\Delta x \beta B_1}}\\[0.15cm]
\frac{\average{\beta v}\average{B_2}}{\average{\Delta x \beta B_2}}\\[0.15cm]
\frac{\average{\beta w}\average{B_3}}{\average{\Delta x \beta B_3}}
\end{bmatrix}_{i-\tfrac{1}{2}}
\right),
\end{equation}
then we determine a discrete entropy and kinetic energy conservative flux to be
\begin{equation}\label{Eq:entropyconservative-yEKEP}
\resizebox{0.925\hsize}{!}{$\vec{f}^{*,ekec} = \begin{bmatrix}
\rho^{\ln}\average{u} \\[0.15cm]
\rho^{\ln}\average{u}^2 + \frac{\average{\rho}}{2\average{\beta}} +\frac{1}{2}\left(\average{B_1^2}+\average{B_2^2}+\average{B_3^2}\right) - \average{B_1^2} \\[0.15cm]
\rho^{\ln}\average{u}\average{v} - \average{B_1B_2} \\[0.15cm]
\rho^{\ln}\average{u}\average{w} - \average{B_1B_3} \\[0.15cm]
\frac{\rho^{\ln}\average{u}}{2(\gamma-1)\beta^{\ln}} +\frac{\average{\rho}\average{u}}{2\average{\beta}} -\frac{1}{2}\rho^{\ln}\average{u}\left(\average{u^2}+\average{v^2}+\average{w^2}\right)+\rho^{\ln}\average{u}\left(\average{u}^2+\average{v}^2+\average{w}^2\right)\\
+\frac{\average{B_2}}{\average{\beta}}\left(\average{\beta u}\average{B_2} - \average{\beta v}\average{B_1}\right) + \frac{\average{B_3}}{\average{\beta}}\left(\average{\beta u}\average{B_3} - \average{\beta w}\average{B_1}\right)\\[0.15cm]
0 \\[0.15cm]
\frac{1}{\average{\beta}}\left(\average{\beta u}\average{B_2} - \average{\beta v}\average{B_1}\right) \\[0.15cm]
\frac{1}{\average{\beta}}\left(\average{\beta u}\average{B_3} - \average{\beta w}\average{B_1}\right)
\end{bmatrix}.$}
\end{equation}
{\color{black}{From the definition of the inverse of the temperature \eqref{tempInverse}, we note that the source term discretization \eqref{SourceTermDiscyEKEP} is identical to the discretization \eqref{SourceTermDisc}.}}
\end{cor}
\begin{proof}
We begin by rewriting the entropy variables \eqref{entropyVariables} and the entropy conservation condition \eqref{entropyConservationCondition2} in terms of the quantity $\beta$
\begin{equation}\label{entVarsEKEP}
\vec{v} = \left[ \frac{\gamma - s}{\gamma - 1} - \beta\|\vec{u}\|^2,\;2\beta u,\;2\beta v,\;2\beta w,\;-2\beta,\;2\beta B_1,\;2\beta B_2,\;2\beta B_3\right]^T.
\end{equation}
where we rewrite the physical entropy as
\begin{equation}
s = \ln(p) - \gamma\ln(\rho) = -(\gamma-1)\ln(\rho) - \ln(\beta) - \ln(2),
\end{equation}
and
\begin{equation}\label{EKEPConservation}
\jump{\,\vec{v}\,}^T\vec{f}^* = \jump{\, \rho u \,} +\jump{\,\beta u\|\vec{B}\|^2\,} - 2\jump{\,\beta B_1(\vec{u}\cdot\vec{B})\,} - \average{\Delta x \vec{v}}^T{\vec{s}}_{i+\tfrac{1}{2}}.
\end{equation}
Now the proof follows the same steps as the proof given for Thm. 1. We expand each of the linear jump terms such that there is no mixing between the hydrodynamic and magnetic field variables. This process generates a system of eight equations for the eight unknown numerical flux components. Again, a consistent source term discretization reveals used to cancel extraneous terms in the $f_6^*$ flux component. Once the source term discretization is determined we solve for the remaining flux components and obtain \eqref{Eq:entropyconservative-yEKEP}.
\end{proof}
\begin{rem}
The entropy and kinetic energy conserving flux \eqref{Eq:entropyconservative-yEKEP} can, again, be split into Euler and magnetic components. We recover in the Euler components the entropy and kinetic energy conserving flux originally developed by Chandrashekar \cite{chandrashekar2013}.
\end{rem}
}}


\bibliographystyle{elsarticle-num}
\bibliography{References.bib}

\begin{thebibliography}{10}
\expandafter\ifx\csname url\endcsname\relax
  \def\url#1{\texttt{#1}}\fi
\expandafter\ifx\csname urlprefix\endcsname\relax\def\urlprefix{URL }\fi
\expandafter\ifx\csname href\endcsname\relax
  \def\href#1#2{#2} \def\path#1{#1}\fi

\bibitem{tadmor1984}
E.~Tadmor, Numerical viscosity and the entropy condition for conservative
  difference schemes, Mathematics of Computation 43 (1984) 369--381.

\bibitem{tadmor1987}
E.~Tadmor, The numerical viscosity of entropy stable schems for systems of
  conservation laws. {I}, Mathematics of Computation 49~(179) (1987) 91--103.

\bibitem{tadmor2003}
E.~Tadmor, Entropy stability theory for difference approximations of nonlinear
  conservation laws and related time-dependent problems, Acta Numerica 12
  (2003) 451--512.

\bibitem{carpenter_esdg}
M.~Carpenter, T.~Fisher, E.~Nielsen, S.~Frankel, Entropy stable spectral
  collocation schemes for the {N}avier--{S}tokes equations: Discontinuous
  interfaces, SIAM Journal on Scientific Computing 36~(5) (2014) B835--B867.

\bibitem{dedner2002}
A.~Dedner, F.~Kemm, D.~Kr{\"o}ner, C.-D. Munz, T.~Schnitzer, M.~Wesenberg,
  Hyperbolic divergence cleaning for the {MHD} equations, Journal of
  Computational Physics 175~(2) (2002) 645--673.

\bibitem{ismail2009}
F.~Ismail, P.~L. Roe, Affordable, entropy-consistent {E}uler flux functions
  {II}: Entropy production at shocks, Journal of Computational Physics 228~(15)
  (2009) 5410--5436.

\bibitem{chandrashekar2013}
P.~Chandrashekar, Kinetic energy preserving and entropy stable finite volume
  schemes for compressible {E}uler and {N}avier-{S}tokes equations,
  Communications in Computational Physics 14~(5) (2013) 1252--1286.

\bibitem{lefloch_rhode_2000}
P.~LeFloch, C.~Rohde, High-order schemes, entropy inequalities, and
  nonclassical shocks, SIAM Journal on Numerical Analysis 37~(6) (2000)
  2023--2060.

\bibitem{gassner_skew_burgers}
G.~Gassner, A skew-symmetric discontinuous {Galerkin} spectral element
  discretization and its relation to {SBP-SAT} finite difference methods, SIAM
  Journal on Scientific Computing 35~(3) (2013) A1233--A1253.

\bibitem{gassner2014}
G.~J. Gassner, A.~R. Winters, D.~A. Kopriva, A well balanced and entropy
  conservative discontinuous {G}alerkin spectral element method for the shallow
  water equations, Applied Mathematics and Computation (2015)
  http://dx.doi.org/10.1016/j.amc.2015.07.014.

\bibitem{rossmanith2002}
J.~A. Rossmanith, A wave propagation method with constrained transport for
  ideal and shallow water magnetohydrodynamics, Ph.D. thesis, University of
  Washington (2002).

\bibitem{toth2000}
G.~T{\'o}th, The $\nabla\cdot{B}=0$ constraint in shock-capturing
  magnetohydrodynamics codes, Journal of Computational Physics 161~(2) (2000)
  605--652.

\bibitem{leveque2002}
R.~J. Le{V}eque, Finite volume methods for hyperbolic problems, Vol.~31,
  Cambridge university press, 2002.

\bibitem{godunov1972}
S.~K. Godunov, The symmetric form of magnetohydrodynamics equation, Num. Meth.
  Mech. Cont. Media 1 (1972) 26--34.

\bibitem{barth99}
T.~J. Barth, Numerical methods for gasdynamic systems on unstructured meshes,
  in: D.~Kr\"{o}ner, M.~Ohlberger, C.~Rohde (Eds.), An Introduction to Recent
  Developments in Theory and Numerics for Conservation Laws, Vol.~5 of Lecture
  Notes in Computational Science and Engineering, Springer Berlin Heidelberg,
  1999, pp. 195--285.

\bibitem{janhunen2000}
P.~Janhunen, A positive conservative method for magnetohydrodynamics based on
  {HLL} and {R}oe methods, Journal of Computational Physics 160~(2) (2000)
  649--661.

\bibitem{powell1994}
K.~G. Powell, An approximate {R}iemann solver for magnetohydrodynamics (that
  works more than one dimension), Tech. rep., {DTIC} {D}ocument (1994).

\bibitem{dellar2001}
P.~J. Dellar, A note on magnetic monopoles and the one-dimensional {MHD}
  riemann problem, Journal of Computational Physics 172~(1) (2001) 392--398.

\bibitem{fjordholm2011}
U.~S. Fjordholm, S.~Mishra, E.~Tadmor, Well-blanaced and energy stable schemes
  for the shallow water equations with discontiuous topography, Journal of
  Computational Physics 230~(14) (2011) 5587--5609.

\bibitem{winters2015}
A.~R. Winters, G.~J. Gassner, An entropy stable finite volume scheme for the
  equations of shallow water magnetohydrodynamics, Journal of Scientific
  Computing~(http://dx.doi.org/10.1007/s10915-015-0092-6).

\bibitem{fjordholm2012}
U.~S. Fjordholm, S.~Mishra, E.~Tadmor, Arbitrarily high-order accurate entropy
  stable essentially nonoscillatory schemes for systems of conservation laws,
  SIAM J. Numer. Anal. 50~(2) (2012) 544--573.

\bibitem{merriam1989}
M.~L. Merriam, An entropy-based approach to nonlinear stability, {NASA}
  Technical Memorandum 101086~(64) (1989) 1--154.

\bibitem{roe1996}
P.~L. Roe, D.~S. Balsara, Notes on the eigensystem of magnetohydrodynamics,
  {SIAM} Journal on Applied Mathematics 56~(1) (1996) 57--67.

\bibitem{brio1988}
M.~Brio, C.~C. Wu, An upwind differencing scheme for the equations of ideal
  magnetohydrodynamics, Journal of Computational Physics 75~(2) (1988)
  400--422.

\bibitem{Carpenter&Kennedy:1994}
M.~Carpenter, C.~Kennedy, Fourth-order 2{N}-storage {R}unge-{K}utta schemes,
  Tech. Rep. NASA TM 109111, {NASA} Langley Research Center (1994).

\bibitem{ryu1994}
D.~Ryu, T.~W. Jones, Numerical magnetohydrodynamics in astrophysics: Algorithm
  and tests for one-dimensional flow, The Astrophysical Journal 442 (1995)
  228--258.

\bibitem{torrilhon2003}
M.~Torrilhon, Uniqueness conditions for {R}iemann problems of ideal
  magnetohydrodynamics, Journal of Plasma Physics 69~(3) (2003) 253--276.

\bibitem{dai1994}
W.~Dai, P.~R. Woodward, Extension of the piecewise parabolic method to
  multidimensional ideal magnetohydrodynamics, Journal of Computational Physics
  115~(2) (1994) 485--514.

\bibitem{ryu1995}
D.~Ryu, T.~W. Jones, A.~Frank, Numerical magnetohydrodynamics in astrophysics:
  Algorithm and tests for multi-dimensional flow, Astrophysical Journal 452~(2)
  (1995) 785--796.

\bibitem{balsara1999}
D.~S. Balsara, D.~S. Spicer, A staggered mesh algorithm using high order
  {G}odunov fluxes to ensure solenoidal magnetic fields in magnetohydrodynamic
  simulations, Journal of Computational Physics 149~(2) (1999) 270--292.

\end{thebibliography}

\end{document}